\documentclass[12pt]{article}
\usepackage{geometry}                
\usepackage{graphicx}
\usepackage{amssymb}
\usepackage{amsmath}
\usepackage{epstopdf}
\usepackage{ebproof}
\usepackage{multicol}
\usepackage{ bbold }
\usepackage{hyperref}
\usepackage{turnstile}
\usepackage{array}
\usepackage{xcolor}
\usepackage{ebproof}
\usepackage{imakeidx}
\usepackage{amsthm}
\usepackage{comment}
\usepackage{authblk}
\usepackage{biblatex}
\newtheorem{Theorem}{Theorem}[section] 
\newtheorem{Lemma}[Theorem]{Lemma}
\newtheorem{definition}[Theorem]{Definition}

\newtheorem{readingconvention}[Theorem]{Reading Convention}
\newtheorem{remark}[Theorem]{Remark}
\newtheorem{claim}{Claim}[Theorem]
\newtheorem{ex}[Theorem]{Example}
\newtheorem{corollary}[Theorem]{Corollary}
\newtheorem{question}[Theorem]{Question}

\newtheorem*{theorem**}{Theorem\theoremnum}
\newenvironment{theorem*}[1][]{%
  \edef\theoremnum{\if\relax\detokenize{#1}\relax\else~#1\fi}
  \begin{theorem**}
}{%
  \end{theorem**}
}

\let\phi=\varphi

\let\epsilon=\varepsilon

\def\L{\mathbb{L}}
\def\La{\mathbb{L}_\alpha(X)}
\def\Lo{\mathbb{L}_{\Omega_\omega}(X)}
\def\Ln{\mathbb{L}_{\Omega_n}(X)}
\def\Lm{\mathbb{L}_{\Omega_m}(X)}
\def\Lnm{\mathbb{L}_{\Omega_{n+1+m_0}}(X)}
\def\Lnpm{\mathbb{L}_{\Omega_{n_0+m}}(X)}
\def\H{\mathcal{H}^\theta}
\def\dash{\sststile}
\def\Vstile{\dststile}

\def\C{\mathcal{C}}
\def\RS{{\sf RS}_l(X)}
\def\Ax{{\sf Ax}}

\def\acc{\mathrm{Acc}_\theta}

\def\prog{\mathrm{Prog}}

\def\tc{\pi}
\def\ot{\mathrm{otype}}

\def\acc{\mathrm{Acc}_\theta}

\def\prog{\mathrm{Prog}}

\def\ord{\mathrm{Ord}}

\title{The Provably Total Set-Recursive Functions of KPl}
\author[1]{J. P. Aguilera} 
\author[1]{A. Fern\'andez} 
\author[2]{J. J. Joosten}
\affil[1]{Institute of Discrete Mathematics and Geometry, TU Wien.}
\affil[2]{Department of Philosophy, University of Barcelona.}
\date{}
\makeindex[name=General,title={General Index}]
\makeindex[name=Symbols,title={Index of Symbols}]

\begin{document}
\maketitle


\begin{abstract}
Using relativized ordinal analysis, we give a proof-theoretic characterization of the provably total set-recursive-from-$\omega$ functions of {\sf KPl} and related theories.
\end{abstract}

\section{Introduction}
Ordinal analysis aims at using ordinal numbers or notation systems for ordinal numbers to derive proof-theoretically meaningful information about mathematical theories. It originated with Gentzen's \cite{gentzen} consistency proof for Peano Arithmetic ($\mathsf{PA}$). Nowadays, Gentzen's result is stated as 
\[|\mathsf{PA}|_{\Pi^1_1} = \varepsilon_0,\]
where $|T|_{\Pi^1_1}$ is the $\Pi^1_1$-norm (or ``proof-theoretic ordinal'') of a theory $T$, defined as the supremum of order-types of recursive wellorderings of $\mathbb{N}$ which are $T$-provably wellfounded. Gentzen's result has been generalized to a great extent, with the first landmark result being Takeuti's \cite{takeuti} ordinal analysis of $\Pi^1_1$-$\mathsf{CA}_0$, showing that 
\[|\Pi^1_1{-}\mathsf{CA}_0|_{\Pi^1_1} = \psi_0(\Omega_\omega).\]
Part of the process of determining $|T|_{\Pi^1_1}$ is giving an explicit (recursive) description of the ordinal in question.
A feature in Takeuti's and similar notation systems is the use of \textit{ordinal collapsing functions} $\psi$ which yield names for large recursive ordinals using names for uncountable ordinals $\Omega$. Recently, Arai \cite{Ar23} has announced an ordinal analysis of second-order arithmetic. Rather than dealing with subsystems of second-order arithmetic directly, Takeuti and Arai carry out ordinal analyses of set theories with the same $\Pi^1_1$ consequences, as is common in the field.

The methods of ordinal analysis can yield more information than proof-theoretic ordinals. For instance, Beklemishev \cite{Be04} has given a proof-theoretic reduction of $\mathsf{PA}$ to iterated consistency assertions via so-called \textit{worms}, and Beklemishev and Pakhomov \cite{BP22} have extended this result beyond first-order arithmetic. In what essentially amounts to a $\Pi^0_2$-analysis of arithmetic, Kreisel \cite{Kr52} famously gave a proof-theoretic classification of the provably total recursive functions of $\mathsf{PA}$ in terms of the Hardy hierarchy: a recursive function $f$ is provably total in $\mathsf{PA}$ if and only if $f$ is dominated by $H_\alpha$ for some $\alpha<\varepsilon_0$. Results such as these (in which one is concerned with consequences of complexity simpler than $\Pi^1_1$) usually depend on the precise notation system chosen for ordinals, contrary to the $\Pi^1_1$ case.

Cook and Rathjen \cite{cookrathj} have given a set-theoretic version of this type of result. They gave a classification of the provably total $\Sigma_1$-definable functions of Kripke-Platek set theory {\sf KP} in terms of a relativized ordinal analysis. The goal of this article is to extend the Cook-Rathjen theorem beyond Kripke-Platek set theory to cover theories such as $\mathsf{KPl}$.

Recall that $\mathsf{KP}$ is the weakening of $\mathsf{ZFC}$ obtained by dropping Choice, the Power Set axiom and restricting Separation and Collection to $\Delta_0$ formulas. The theory $\mathsf{KPl}$ is the ``limit'' version of $\mathsf{KP}$, in which the schema of $\Delta_0$-Collection is replaced by the axiom asserting that every set belongs to a transitive model of $\mathsf{KP}$. Initially, it might seem hopeless to attempt to obtain any effective description of the total $\Sigma_1$-definable functions which are $\mathsf{KPl}$-provably total.
For instance, 
\[a \mapsto \text{smallest admissible set containing $a$}\]
is such a function, and certainly beyond the realm of effective descriptions, especially those that could potentially be provided by ordinal analysis. The key is to reframe the Cook-Rathjen theorem in the context of \textit{set-recursion} and in particular \textit{set-recursion-from-$\omega$}. Thus, the question we investigate is, for a given $T$:

\begin{question}\label{QuestionProvablyTotal}
What are the $T$-provably total set-recursive-from-$\omega$ functions?
\end{question}

The notion of set-recursion is a generalization of Kleene's higher-type recursion, essentially to arbitrary transfinite types. It is due independently to Moschovakis and Normann. We refer the reader to Sacks \cite{Sa90} for background. It has a definition in terms of a finite set of schemes (the \textit{Normann schemes}) similar to the definition of a recursive function on $\mathbb{N}$ and extending the notion of a \textit{primitive recursive set function} due to Jensen and Karp. 
Rather than repeating the definition, we mention Van de Wiele's theorem \cite{VdW82}: a function $f$ is set-recursive if and only if it is $\Sigma_1$-definable via a single definition which yields a total function in every admissible set. Slaman has generalized set-recursion to $\mathrm{ES}_p$-recursion for an arbitrary set $p$, which is set-recursion plus a new selection scheme for $p$, in \cite{slaman}. He also proved a result analogous to Van de Wiele's Theorem: a function is $\mathrm{ES}_p$-recursive if and only if $f$ is uniformly $\Sigma_1$-definable with parameter $p$ in every admissible set. In our case, for $p=\omega$, $\mathrm{ES_\omega}$-recursion amounts to set-recursion-from-$\omega$ (that is, the schemes of set-recursion using $\omega$ as a parameter). This means that, by Slaman's Theorem, a function $f$ is set-recursive-from-$\omega$ if and only if $f$ is uniformly $\Sigma_1$-definable with parameter $\omega$ in every admissible set. Since $\mathsf{KP}$ holds, by definition, in every admissible set, there is no distinction between a characterization of the $\mathsf{KP}$-provably total $\Sigma_1$-definable functions, and the $\mathsf{KP}$-provably total set-recursive-from-$\omega$ functions, which is why this issue did not arise in the work of Cook and Rathjen. Especially when discussing functions which are recursive in some sense, their ``complexity'' is often measured in terms of  growth rates. We shall give a bound on the growth rates of $\mathsf{KPl}$-provable set-recursive-from-$\omega$ functions which is optimal in some sense.

We do so via a relativized ordinal notation system introduced in \S\ref{SectOrdinalNotation}: given a set $x$ we define an ordinal notation system depending on $x$ that resembles the usual ordinal notation system for collapsing functions. This relativized ordinal notation system can be used to index ranks of proofs for a cut-free infinitary proof system admitting a collapsing theorem (see \S \ref{SectCutElimination}) and to which $\mathsf{KPl}$ can be embedded (see \S\ref{rs}, \S \ref{SectCutElimination}, and \S\ref{SectEmbedding}). Using this, we can give a bound on the growth rate for the provably total set-recursive-from-$\omega$ functions of $\mathsf{KPl}$: there is a uniform sequence of relativized ordinal notations $e_{n} = e_n^x$ (depending uniformly on $x$ and ranging over $n\in\mathbb{N}$) such that if $f$ is a set-recursive-from-$\omega$ function provably total in $\mathsf{KPl}$, then 
\[\exists n\, \forall x\, f(x) \in L_{\psi_0(e_{n})}(x),\]
where $L_{\psi_0(e_{n})}(x)$ denotes G\"odel's constructible hierarchy relativized to the transitive closure of $x$ and built up to stage $\psi_0(e_{n})$. 
This is done in \S\ref{SectTheoremMainKPl}. Here, $\psi_0$ is a collapsing function and $e_n$ is defined by $e_0 = \Omega_\omega + 1$ and $e_{m+1} = \omega^{e_m}$. As usual, $\Omega_k$ denotes a sequence of ``sufficiently large'' ordinals, but the precise interpretation depends on $x$; see also \cite[Proposition 9]{APW24} for a related construction. Specifically, it suffices to interpret $\Omega_k$ as the $(k+1)$st ordinal admissible relative to the transitive closure of $\{x\}$.

In \S\ref{SectTheoremMainWKPl} we extend the main result from $\mathsf{KPl}$ to the theories $\mathsf{W}{-}\mathsf{KPl}$ and $\mathsf{KPl^r}$ obtained from $\mathsf{KPl}$ by weakening the schema of foundation to set-foundation and full induction on $\mathbb{N}$, respectively. Some authors use the notation ${\sf KPl_0}$ for ${\sf KPl^r}$.

It follows from folklore results that the bounds given are optimal in the sense that the ordinals employed in the notation system are all necessary when $x = \varnothing$. However, even if our bounds are optimal for $x = \varnothing$, it is not immediately clear that they are optimal in the stronger sense that they cannot be replaced by slower-growing functions.
We prove optimality in this stronger sense in \S\ref{SectWellOrderingProofs}. As far as we know, the proofs of these ``folklore'' results (for $x = \varnothing$) have never appeared in print for $\mathsf{W}{-}\mathsf{KPl}$ and $\mathsf{KPl^r}$, so the results and proofs of \S\ref{SectWellOrderingProofs} serve the additional purpose of filling this gap in the literature.

\section{Preliminaries on Set Theory}
In this section, we define the first-order set theory {\sf KPl}, as well as the subtheories ${\sf KPl^r}$ and {\sf W-KPl}. These theories are strongly related to {\sf KP}, Kripke-Platek: the theory {\sf KPl} states that the universe is a limit of transitive models of {\sf KP}. This is why we expand the usual language of set theory $\mathcal L:=\{\in,\notin\}$ (we will use the connective $\neg$ as a defined notion, working in negation normal form) with a predicate that intends to mean that a given set is a transitive model of {\sf KP}. So we consider the language $\mathcal L':=\{\in,\notin,Ad,\neg Ad\}$.\\
We define the formulas of the language $\mathcal L'$ of {\sf KPl}, ${\sf KPl^r}$ and {\sf W-KPl}. We will use $a,b,c,\dots$ to denote free variables and $x,y,z,\dots$ to denote bound variables.
\begin{definition}[Formulas of the language $\mathcal{L'}$ of {\sf KPl}]\label{formulas}
    The atomic formulas are $a\in b$, $a\notin b$, $Ad(a)$ and $\neg Ad(a)$.\\
    If $A$ and $B$ are two formulas, then $A\vee B$ and $A\wedge B$ are formulas.\\
    If $A(b)$ is a formula in which $x$ does not occur then $\forall x\ A(x)$ and $\exists x\ A(x)$ are formulas (in which the quantifiers $\forall x$ and $\exists x$ are unbounded or unrestricted) and $\forall x\in a\ A(x)$ and $\exists x\in a\ A(x)$ are formulas (in which the quantifiers $\forall x\in a$ and $\exists x\in a$ are bounded by, or restricted to, $a$).
\end{definition}
The following definition introduces important classes of {\sf KPl}-formulas.
\begin{definition}[{\sf KPl}-formulas]\label{ax}
    Let $A$ be a formula. Let $a$ be a set. The formula obtained by restricting all the unbounded quantifiers of the universal closure of $A$ to $a$ will be denoted by $A^a$.\\
    A formula $A$ is $\Delta_0$ if no unbounded quantifier occurs in $A$. A formula is $\Sigma_1$ if the formula is of the form $\exists x\ \phi(x)$ for some $\Delta_0$-formula $\phi$.\\
    A formula $A$ is $\Sigma$ if no unbounded universal quantifier occurs in $A$.
\end{definition}
\noindent
In particular, for any formula $A$ and any set $a$ the formula $A^a$ is $\Delta_0$.\\
Moreover, we will use the equality symbol as a defined notion, and we also define the connectives $\neg$ and $\rightarrow$:
\begin{definition}\label{equalconn}
    The formula $a\subseteq b$ will stand for $\forall x\in a(x\in b)$.\\
    The formula $a=b$ will stand for $a\subseteq b\wedge b\subseteq a$.\\
    The formula $\neg A$ is obtained by replacing in $A$ the symbol $\in$ by $\notin$ and vice-versa, $\wedge$ by $\vee$ and vice-versa, $\forall$ by $\exists$ and vice-versa, and $Ad(\cdot)$ by $\neg Ad(\cdot)$ and vice-versa.\\ 
    We define $A\rightarrow B\equiv \neg A \vee B$.\\
    The formula $a\neq b$ means $\neg a = b$.
\end{definition}
The theory {\sf KP} is a weak subtheory of {\sf ZF}, where we remove the Power Set Axiom and we restrict the Separation and Collection schemas to $\Delta_0$-formulas. The transitive set models of {\sf KP} are called $\textit{admissible}$ sets. We refer to \cite{barwise} for a study of {\sf KP} and admissible sets. We will not define well-known set-theoretic $\Delta_0$-notions such as the ordinals $0$ or $\omega$ and the operation $+$ on $\omega$. The theory {\sf KPl} axiomatizes the existence of unboundedly many admissible sets in the universe. {\sf KPl} comprises standard set-theoretic axioms, which are just the axioms of {\sf KP} without the scheme of $\Delta_0$-Collection:
\begin{equation*}
    \Gamma, \forall x\in a\exists y\ B(x,y)\rightarrow \exists z\forall x\in a\exists y\in z\ B(x,y)\text{, for any $\Delta_0$-formula $B(b,c)$;}
\end{equation*}
three new axioms stating the properties of admissible sets, plus a limit axiom that says that any set belongs to some admissible set. {\sf KPl} is formulated in a single sided Tait-style sequent calculus where sequents shall be sets of formulas. Thus, the axioms of {\sf KPl} are the following.

\begin{definition}
    The axioms of {\sf KPl} are the following axioms and schemes in the language $\mathcal{L'}$, where $\Gamma$ is any set of $\mathcal{L}'$-formulas and the formulas appearing in the schemes allow parameters.\\
    \ \\
         \begin{tabular}{ll}
        Logical axioms: & $\Gamma,A,\neg A$ for any formula $A$, \\
         Leibniz Principle: & $\Gamma, \big(a=b\wedge B(a)\big)\rightarrow B(b)$ for any formula $B(a)$,\\
         Pair: & $\Gamma,\exists z(a\in z\wedge b\in z)$,\\
         Union: & $\Gamma,\exists z\forall y\in a\forall x\in y(x\in z)$,\\
         $\Delta_0$-Separation: & $\Gamma,\exists y[\forall x\in y\big(x\in a\wedge B(x)\big)\wedge\forall x\in a\big(B(x)\rightarrow x\in y\big)]$\\
         & for any $\Delta_0$-formula $B(a)$,\\
         Infinity: & $\Gamma, \exists x[\exists z\in x(z\in x)\wedge \forall y\in x\exists z\in x(y\in z)]$,\\
         Class Induction: & $\Gamma,\forall x[\forall y\in x B(y)\rightarrow B(x)]\rightarrow \forall x B(x)$\\
         & for any formula $B(a)$,\\
        \  & \  \\
         Ad1: & $\Gamma,\forall x[Ad(x)\rightarrow \omega\in x\wedge\ \mathrm{Tran}(x)]$,\\
        Ad2: & $\Gamma,\forall x\forall y[Ad(x)\wedge Ad(y)\rightarrow x\in y\vee x=y\vee y\in x]$,\\
        Ad3: & $\Gamma,\forall x[Ad(x)\rightarrow (Pair)^x\wedge (Union)^x\wedge (\Delta_0-Sep)^x\wedge (\Delta_0-Coll)^x]$,\\
        Lim: & $\Gamma, \forall x\exists y[Ad(y)\wedge x\in y]$.
    \end{tabular}\\

\end{definition}
We observe that, in a given model $M$ of {\sf KPl}, sets satisfying $Ad^M$ are linearly ordered by the $\in^M$-relation. Whereas admissible sets are not linearly ordered, we care about a particular linearly ordered class of admissibles and so we will use the word ``admissible'' and the predicate $Ad$ to denote admissible sets from this class. Typically in proof-theoretic analyses, one interprets the predicate $Ad$ as ranging over admissible sets of some particular form, e.g., $L_\alpha$. Here, we intend that it ranges over admissible sets of the form $L_\alpha(X)$ for some fixed $X$ (c.f. Section \ref{SectOrdinalNotation})
\\
The theory {\sf KP} consists of the axioms of {\sf KPl} in the language $\mathcal{L}$ together with the scheme of $\Delta_0$-Collection but without $Ad1$, $Ad2$, $Ad3$ and $Lim$. Finally, we obtain the two theories ${\sf KPl^r}$ and {\sf W-KPl} by restricting the scheme of induction from {\sf KPl}.
\begin{definition}
    The theory ${\sf KPl^r}$ is {\sf KPl} with the scheme of induction restricted to $\Delta_0$-formulas.\\
    The theory {\sf W-KPl} is ${\sf KPl^r}$ plus full Mathematical Induction, i.e. the scheme
    \begin{equation*}
    F(0)\wedge \forall x\in\omega \big(F(x)\rightarrow F(x+1)\big)\rightarrow \forall x\in\omega\ F(x).
\end{equation*}
for any formula $F$.
\end{definition}
We introduce the rules of inference that we will use for {\sf KPl}, written in Tait sequent style. Both in the premises and in the conclusion of each rule we have finite sets of $\mathcal L'$-formulas. That means that, below, the symbol $\Gamma$ represents any finite set of $\mathcal L'$-formulas, just as in the formulation of the axioms. Moreover, $A$ and $B$ are any $\mathcal L'$-formulas. The formulas derived in the conclusion can intuitively be thought as a disjunction. That is, when we derive $A,B$, it means that either $A$ or $B$ is true. Also, we treat bounded quantification as a primitive logical symbol: we have a rule to derive formulas of the form $\exists x\in b\ B(x)$ and $\forall x\in b\ B(x)$.
\begin{definition}
    The rules of inference of {\sf KPl} are the following.\\
    \ \\
    \begin{center}
    \begin{prooftree}
        \hypo{\Gamma, A}
        \hypo{\Gamma,B}
        \infer[left label = $(\wedge)$]2{\Gamma, A\wedge B}
    \end{prooftree}\hspace{2cm}
    \begin{prooftree}
        \hypo{\Gamma,A}
        \infer[left label = $(\vee)$]1{\Gamma,A\vee B}
    \end{prooftree}\hspace{2cm}
     \begin{prooftree}
        \hypo{\Gamma,B}
        \infer[left label = $(\vee)$]1{\Gamma,A\vee B}
    \end{prooftree}\\
    \vspace{0.5cm}
     \begin{prooftree}
        \hypo{\Gamma,a\in b\wedge B(a)}
        \infer[left label = $(b\exists)$]1{\Gamma,\exists x\in b\ B(x)}
    \end{prooftree}\hspace{2cm}
    \begin{prooftree}
        \hypo{\Gamma, B(a)}
        \infer[left label = $(\exists)$]1{\Gamma,\exists x\ B(x)}
    \end{prooftree}\\
    \vspace{0.5cm}
    \begin{prooftree}
        \hypo{\Gamma,a\in b\rightarrow B(a)}
        \infer[left label = $(b\forall)$]1{\Gamma,\forall x\in b\ B(x)}
    \end{prooftree}\hspace{2cm}
    \begin{prooftree}
        \hypo{\Gamma, B(a)}
        \infer[left label = $(\forall)$]1{\Gamma,\forall x\ B(x)}
    \end{prooftree}\\
    \vspace{0.5cm}
    \begin{prooftree}
        \hypo{\Gamma, A}
        \hypo{\Gamma, \neg A}
        \infer[left label = $(Cut)$]2{\Gamma}
    \end{prooftree}
    \end{center}
    \ \newline
    In the rules $(b\forall)$ and $(\forall)$, it must be the case that $a$ does not occur in the conclusion.
\end{definition}


\section{Our Ordinal Notation System}\label{SectOrdinalNotation}
We recall that our main objective in this paper is to classify the set-recursive-from-$\omega$ total functions of {\sf KPl}, ${\sf KPl^r}$ and {\sf W-KPl}. We will show that the image of any set $X$ under a provably set-recursive-from-$\omega$ total function in one of those theories belongs to some initial segment of the constructible hierarchy relativized to $X$, defined as follows.
\begin{definition}[Constructible hierarchy relativized to $X$]
$\label{const}$
    Let $X$ be any set. We define for every ordinal $\alpha$ the set \index[Symbols]{$L_\alpha(X)$} $L_\alpha(X)$ as:\\
         $L_0(X)=TC(\{X\})$,\\
         $L_{\alpha+1}(X)=\{Y\subseteq L_\alpha(X):\text{$Y$ is definable over $\langle L_\alpha(X),\in\rangle$}\}$,\\
         $L_\gamma(X)=\bigcup_{\alpha<\gamma}L_\alpha(X)$ if $\gamma$ is a limit.
\end{definition}
The classification theorem will rely on a relativized ordinal analysis of {\sf KPl}. We will fix a set $X$ to show that $f(X)$ indeed belongs to the claimed bound. Once our set $X$ is fixed, we will build a proof system depending on this $X$ where we will embed {\sf KPl}. We will call this system $\RS$, for Ramified Set theory relativized to $X$. To define $\RS$ we first need to define an ordinal notation notation system in this section.\\
The set-theoretic rank of a set $y$ is defined as the ordinal $\mathrm{rank}(y)=\sup\{\mathrm{rank}(z)+1:z\in y\}$, with $\sup\  \emptyset=0$. We will use the following ordinals to define a sequence of admissible sets.
\begin{readingconvention}$\label{read}$\index[Symbols]{$\Omega_\omega$} \index[Symbols]{$\theta$}
The set $X$ is a fixed set. The set-theoretic rank of $X$ is $\theta$. The sequence $\langle \Omega_n:n\leq\omega\rangle$ enumerates the $\omega+1$-first $X$-admissible ordinals, meaning $L_{\Omega_n}(X)$ is admissible. 
\end{readingconvention}

Within $\RS$, we will consider $L_{\Omega_\omega}(X)$ as our universe of {\sf KPl} since $L_{\Omega_\omega}(X)$ contains unboundedly many admissible sets: each $L_{\Omega_n}(X)$ for $n<\omega$ is admissible.\\
Furthermore, all the ordinals that we will employ in $\RS$ will depend on the fixed set $X$ in a recursive way. In this section, we define for a fixed set $X$ with set-theoretic rank $\theta$ the primitive recursive in $\theta$ set $T(\theta)$ of representations of ordinals that we will use in $\RS$.

\subsection{An overview on Veblen functions}

In this subsection, we present Veblen functions, a useful tool to define ordinal representation systems in a recursive manner.
\begin{definition}[Veblen functions]\index[General]{Veblen functions} \index[Symbols]{$\phi_\cdot(\cdot)$} \label{veblen}
    We define simultaneously $\mathrm{Cr}(\alpha)$ and $\phi_\alpha$ for every ordinal $\alpha$ by induction on $\alpha$.\\
    $\mathrm{Cr}(0):=\{\alpha:\alpha\neq 0\wedge \forall\, \beta,\delta{<}\alpha\,(\beta+\delta<\alpha)\}$ and $\phi_0$ enumerates $\mathrm{Cr}(0)$.\\
    $\mathrm{Cr}(\alpha):=\{\beta:\forall\, \delta {<} \alpha\, \big(\phi_\delta(\beta)=\beta\big)\}$ and $\phi_\alpha$ enumerates $\mathrm{Cr}(\alpha)$.
\end{definition}
\noindent
Sometimes, we will write $\phi\alpha\beta$ instead of $\phi_\alpha(\beta)$. Since each $\phi_\alpha$ enumerates a class of ordinals, every $\phi_\alpha$ is an ordinal order-preserving\footnote{The ordinal function $f$ is order-preserving iff $\forall \alpha,\beta(\alpha\leq\beta\rightarrow f(\alpha)\leq f(\beta))$} function. The class $\mathrm{Cr}(0)$ consists of all the additive principal ordinals and can be more explicitly written as
\begin{equation*}
    \mathrm{Cr}(0)=\{\omega^\beta:\beta\in \mathrm{Ord}\}.
\end{equation*}
The first element of $\mathrm{Cr}(1)$ is called $\epsilon_0$, which is the limit of $\{\omega,\omega^\omega,\omega^{\omega^\omega},\dots\}$.\\
We will make extensive use of Cantor normal form in base $\omega$.
\begin{definition}
    Let $\alpha$ be an ordinal. We define the Cantor normal form of $\alpha$ as follows.
    \begin{center}
        $\alpha=_{CNF}\omega^{\alpha_1}+\cdots+\omega^{\alpha_n}$ if 
        $\alpha\geq\alpha_1\geq\dots\geq\alpha_n$ and $\alpha=\omega^{\alpha_1}+\cdots+\omega^{\alpha_n}$.
        \end{center}
\end{definition}
The following lemma shows that Veblen functions are suitable for recursive notation systems (see Theorem 3.4.9 in \cite{pohlersbook}).
\begin{Lemma} Let $\alpha_1,\alpha_2,\beta_1,\beta_2$ be ordinals. Then\\
$\label{fixbelow}$
    $\phi\alpha_1\beta_1=\phi\alpha_2\beta_2\Leftrightarrow
    \begin{cases}
        \alpha_1=\alpha_2 \wedge \beta_1=\beta_2\text{, or}\\
        \alpha_1<\alpha_2\wedge\beta_1=\phi\alpha_2\beta_2\text{, or}\\
        \alpha_2<\alpha_1\wedge \beta_2=\phi\alpha_1\beta_1.
    \end{cases}$\\
    Moreover,\\
     $\phi\alpha_1\beta_1<\phi\alpha_2\beta_2\Leftrightarrow
    \begin{cases}
        \alpha_1=\alpha_2 \wedge \beta_1<\beta_2\text{, or}\\
        \alpha_1<\alpha_2\wedge\beta_1<\phi\alpha_2\beta_2\text{, or}\\
        \alpha_2<\alpha_1\wedge \phi\alpha_1\beta_1<\beta_2.
    \end{cases}$\\
\end{Lemma}
We define the class of \textit{strongly critical} ordinals. A non-zero ordinal $\alpha$ is strongly critical if $\alpha$ is a fixpoint of every $\phi_\beta$, for $\beta<\alpha$.
\begin{definition}[Strongly critical ordinals]$\label{strongcrit}$
    We define $\mathrm{SC}:=\{\alpha:\alpha\in \mathrm{Cr}(\alpha)\}$. The function $\lambda \beta.\Gamma_\beta$ enumerates $\mathrm{SC}$.
\end{definition}
The following lemma gives different characterizations of the class of strongly critical ordinals.
\begin{Lemma} Let $\alpha$ be an ordinal. Then
\begin{equation*}
    \begin{split}
        \alpha\in \mathrm{SC} & \text{ iff }\alpha\in \mathrm{Cr}(\alpha)\\ 
        & \text{ iff }\alpha=\phi\alpha0\\
        & \text{ iff }\phi\beta\delta<\alpha\text{ for any } \beta,\delta<\alpha \\
        & \text{ iff } \alpha=\phi\beta\alpha\text{ for any }\beta<\alpha.
    \end{split}
\end{equation*}
\end{Lemma}
\begin{proof}
    The first equivalence is Definition $\ref{strongcrit}$. The second is Theorem 13.13 in \cite{schutte}. The third is Lemma 3.4.13 in \cite{pohlersbook}. The last equivalence holds by Definition $\ref{veblen}$.
\end{proof}
So strongly critical ordinals have two closure properties: if $\Gamma_\beta$ is a strongly critical ordinal, then for any $\delta,\gamma<\Gamma_\beta$ we have $\delta+\gamma<\Gamma_\beta$ and
 $\phi\delta\gamma<\Gamma_\beta$.

A deeper study of Veblen functions can be found in \cite{pohlersbook} or in \cite{schutte}.

\subsection{Our relativized ordinal notation system}

Now, we will construct the ordinal notation system $T(\theta)$ that we will use to build the $\RS$-logic. We fix a set $X$ and follow Reading Convention $\ref{read}$.\\
There will be some initial ordinals represented in $T(\theta)$, and $T(\theta)$ will be closed under addition, Veblen functions, and the collapsing functions $\psi_n$ for $n<\omega$. We define the $\psi_n$ functions as follows.
\begin{definition}$\label{psi}$\index[Symbols]{$B_n(\alpha)$} We define simultaneously the sets $B_n(\alpha)$ of ordinals and the ordinal function $\psi_n(\alpha)$. For every $n<\omega$, we define $B_n(\alpha)=\bigcup_{k<\omega}B^k_n(\alpha)$, where $B_n^k(\alpha)$ is defined by double induction on $n$ and $k$ as follows.\\
\begin{itemize}
    \item $B_0^0(\alpha)=\{0\}\cup\{\Gamma_\beta:\beta\leq\theta\}\cup\{\Omega_m:m\leq\omega\}$,
    \begin{flalign*}
        B_0^{k+1}(\alpha)=B_0^k(\alpha) & \cup\{\delta+\delta',\phi\delta\delta':\delta,\delta'\in B_0^k(\alpha)\} &\\
        & \cup\{\psi_m(\delta):\delta\in B_0^k(\alpha)\wedge\delta<\alpha\wedge m<\omega\}. &
    \end{flalign*}
    \item $B_{n+1}^0(\alpha)=\Omega_n\cup\{\Omega_m:m\leq\omega\}$,
    \begin{flalign*}
        B_{n+1}^{k+1}(\alpha)=B_{n+1}^k(\alpha) & \cup\{\delta+\delta',\phi\delta\delta':\delta,\delta'\in B_{n+1}^k(\alpha)\} &\\
        & \cup\{\psi_m(\delta):\delta\in B_{n+1}^k(\alpha)\wedge\delta<\alpha\wedge m<\omega\}. &
    \end{flalign*}
\end{itemize}
    \ \\
    \index[Symbols]{$\psi_n(\alpha)$} The ordinal collapsing function $\psi_n$ is defined as $\psi_n(\alpha)=\min\{\beta:\beta\notin B_n(\alpha)\}$.
\end{definition}
We will sometimes write $\psi_n\alpha$ instead of $\psi_n(\alpha)$. Now, we study the sets $B_n(\alpha)$ and the collapsing functions $\psi_n$. The next lemma shows that for any natural number $n$ the function $\psi_n(\alpha)$ collapses any ordinal to some ordinal below $\Omega_n$: this is an important property, since, in our system $\RS$, we will have to collapse some ordinal bounds below some $\Omega_n$ (see Theorem $\ref{collapsing}$). In the following lemma, we write $\tc(B)$ for a set of ordinals $B$ to denote the \textit{transitive collapse} of $B$, defined as the unique transitive set $\in$-isomorphic to $B$. Since $B$ is a set of ordinals, $\tc(B)=\ot(B)$.
\begin{Lemma} For each $n<\omega$, the following is provable in ${\sf KPl^r}$. For each set $X$ and each ordinal $\alpha$, we have: 
    \begin{enumerate}
        \item $\psi_n(\alpha)$ is a strongly critical ordinal,  
        \item $\tc(B_0(\alpha))<\Omega_0$ and $\Omega_n<\tc(B_{n+1}(\alpha))<\Omega_{n+1}$,
        \item $\Gamma_{\theta+1}\leq\psi_0(\alpha)<\Omega_0$ and $\Omega_n<\psi_{n+1}(\alpha)<\Omega_{n+1}$.
 
    \end{enumerate}
\end{Lemma} 
\begin{proof}
    By an inspection of Definition \ref{psi}, we notice that $\psi_n(\alpha)$ is closed under $+$ and $\phi$, showing 1.\\
    \ \\
    2. We show that $\tc \big(B_0^k(\alpha)\big)<\Omega_0$ for all $k<\Omega_0$ by induction on $k$. First, since $\Omega_0$ is $X$-admissible, we have that $\omega,\theta<\Omega_0$, and so $\tc\big(B_0^0(\alpha)\big)=1+\theta+\omega<\Omega_0$.\\
    Now, we assume that $\tc\big(B_0^k(\alpha)\big)<\Omega_0$. Since $\Omega_0$ is admissible, we have by $\Sigma_1$-recursion that
    \begin{equation*}
        L_{\Omega_0}(X)\vDash \forall \beta,\delta\ (`\beta+\delta\text{ exists'}\wedge `\phi\beta\delta\text{ exists'}),
    \end{equation*}
    as those functions have a $\Sigma_1$ definition.
    Hence, given $\beta,\delta\in \tc\big(B_0^k(\alpha)\big)<\Omega_0$, we have $\beta+\delta,\phi\beta\delta<\Omega_0$. Moreover, we observe that the operations $\psi_k$ are also $\Sigma_1$-definable. This means that 
    \begin{equation*}
        L_{\Omega_0}(X)\vDash \forall \beta\ \forall k<\omega (`\psi_k(\beta)\text{ exists'}).
    \end{equation*}
    Therefore, given $\beta\in \tc\big(B_0^k(\alpha)\big)$, we have that $\psi_k(\beta)<\Omega_0$ for any $k<\Omega_0$.\\
    Hence, we conclude that $\tc\big(B_0^{k+1}(\alpha)\big)<\Omega_0$.\\
    Finally, since the transitive collapse of a set is $\Sigma_1$-definable, it follows that $\tc\big(B_0(\alpha)\big)$ is still bounded by $\Omega_0$.\\
    \ \\
    The proof showing $\tc\big(B_{n+1}(\alpha)\big)<\Omega_{n+1}$ is analogous. Since $\Omega_{n+1}$ is admissible, $L_{\Omega_{n+1}}(X)$ satisfies $\Sigma_1$-recursion, from which we derive that $\tc\big(B_{n+1}(\alpha)\big)<\Omega_{n+1}$. The fact that $\Omega_n<\tc\big(B_{n+1}(\alpha)\big)$ follows from $\Omega_n\subset B_{n+1}(\alpha)$.\\
    \ \\
    3. We first study the case $n=0$. By the first item of this lemma, we have that $\psi_0(\alpha)$ is strongly critical. Moreover, by definition $\{\Gamma_\beta:\beta\leq\theta\}\subseteq B_0(\alpha)$. This means that the first strongly critical ordinals up to and including $\Gamma_\theta$ must be below $\psi_0(\alpha)$. Therefore, $\Gamma_{\theta+1}\leq\psi_0(\alpha)$. Furthermore, $\psi_0$ is defined in a $\Sigma_1$ way from ordinals that belong to $L_{\Omega_0}(X)$, and so by $\Sigma_1$-recursion $\psi_0(\alpha)$ exists in $L_{\Omega_0}(X)$. Therefore, $\psi_0(\alpha)<\Omega_0$\\
    For the second part of Item 3. we observe that, by definition, for any $n<\omega$ we have that $\Omega_n\subseteq B_{n+1}(\alpha)$ and so $\Omega_n< \psi_{n+1}(\alpha)$. Moreover, by the same argument as in the previous case, we have that $\psi_{n+1}(\alpha)$ always belongs to $L_{\Omega_{n+1}}(X)$, and thus $\psi_{n+1}(\alpha)<\Omega_{n+1}$. 
\end{proof}

The next lemma states more results about the sets $B_n(\alpha)$ and the functions $\psi_n$. In particular, Items 6., 7., and 8. provide a criterion for when the functions grow or stabilize.
\begin{Lemma}$\label{311}$ For any ordinals $\alpha$ and $\beta$ and for any natural number $n$, we have:
    \begin{enumerate}
        \item If $\beta<\alpha$ then $B_n(\beta)\subseteq B_n(\alpha)$ and $\psi_n(\beta)\leq\psi_n(\alpha)$,
        \item If $\beta\in B_n(\alpha)\cap\alpha$ then $\psi_n(\beta)<\psi_n(\alpha)$,
        \item If $\beta\leq\alpha$ and $[\beta,\alpha)\cap B_n(\alpha)=\emptyset$ then $B_n(\beta)=B_n(\alpha)$,
        \item $B_n(\alpha)\cap\Omega_n=\psi_n(\alpha)$,
        \item If $\alpha$ is a limit then $B_n(\alpha)=\bigcup_{\beta<\alpha}B_n(\beta)$ and $\psi_n(\alpha)=\sup\{\psi_n(\beta):\beta<\alpha\}$,
        \item $\psi_n(\alpha+1)\in\{\psi_n(\alpha),\big(\psi_n(\alpha)\big)^\Gamma\}$, where $\delta^\Gamma$ is defined as the first strongly critical ordinal above $\delta$,
        \item If $\alpha\in B_n(\alpha)$ then $\psi_n(\alpha+1)=\big(\psi_n(\alpha)\big)^\Gamma$,
        \item If $\alpha\notin B_n(\alpha)$ then $\psi_n(\alpha+1)=\psi_n(\alpha)$.
    \end{enumerate}
\end{Lemma}
The proof of this lemma is analogous to the proof of Lemma 2.6 in \cite{cookrathj}. Using Lemma \ref{311} we can define the \textit{normal form} of an ordinal $\alpha$. This normal form is an extension of Cantor normal form that gives representations to some strongly critical ordinals.
\begin{definition}$\label{nfdef}$
    \index[General]{Normal Form}Let $\alpha$ be an ordinal. We define the normal form of $\alpha$ as follows.
    \begin{enumerate}
        \item $\alpha=_{\sf NF}\alpha_1+\cdots+\alpha_n$ iff $\alpha=\alpha_1+\cdots+\alpha_n$, $n>1$, where the ordinals $\alpha_1,\dots,\alpha_n$ are additive principal and $\alpha_1\geq\dots\geq \alpha_n$,
        \item $\alpha=_{\sf NF}\phi\alpha_1\alpha_2$ iff $\alpha=\phi\alpha_1\alpha_2$ with $\alpha_1,\alpha_2<\alpha$,
        \item $\alpha=_{\sf NF}\psi_n(\alpha_1)$ iff $\alpha=\psi_n(\alpha_1)$ with $\alpha_1\in B_n(\alpha_1)$.
    \end{enumerate}
\end{definition}
The idea is that, starting with some basic ordinals (the ordinal $0$, the first strongly critical ordinals up to $\Gamma_\theta$ and the $\Omega_n$ ordinals), we will construct ordinals using normal forms, so that each ordinal constructed this way has a unique representation. This is the reason why we add the condition $\alpha_1\in B_n(\alpha_1)$ in Item 3.: it may be the case that $\alpha=\psi_n(\beta)=\psi_n(\beta+1)=\dots=\psi_n(\alpha_1)$, but the condition $\alpha_1\in B(\alpha_1)$ forces $\psi_n(\alpha_1+1)$ to increase, by Lemma $\ref{311}$.7. Therefore, $\alpha_1$ is the greatest ordinal with image $\alpha$, and we choose $\psi_n(\alpha_1)$ to represent $\alpha$. This also motivates the next lemma that states that an ordinal is in some $B_n(\alpha)$ whenever its ``normal form components'' are.
\begin{Lemma} Let $\alpha,\gamma$ be any ordinals and let $m$ be any natural number.
$\label{normal form}$
    \begin{enumerate}
        \item If $\alpha=_{\sf NF}\alpha_1+\cdots+\alpha_n$ then $[\alpha\in B_m(\gamma)$ iff $\alpha_1,\dots,\alpha_n\in B_m(\gamma)]$,
        \item If $\alpha=_{\sf NF}\phi\alpha_1\alpha_2$ then $[\alpha\in B_m(\gamma)$ iff $\alpha_1,\alpha_2\in B_m(\gamma)]$,
        \item If $\alpha=_{\sf NF}\psi_m(\alpha_1)$ then $[\alpha\in B_m(\gamma)$ iff $\alpha_1\in B_m(\gamma)\cap\gamma]$.
    \end{enumerate}
\end{Lemma}
Again, the proof of this result is analogous to the proof of Lemma 2.8 in \cite{cookrathj}.\\
At this point, we can do a first step in the definition of $T(\theta)$: we construct $R(\theta)$, the set of ordinals that the strings of $T(\theta)$ will intend to denote. We also define the complexity $C\alpha$ of an ordinal $\alpha\in R(\theta)$ in order to perform induction on this complexity in the following results.
\begin{definition}\label{setrtheta}
    We inductively define the set $R(\theta)$ together with the complexity $C\alpha\in\omega$ of its elements.
    \begin{enumerate}
        \item $0\in R(\theta)$ and $C0:=0$.
        
        \item For every $n\leq\omega$, $\Omega_n\in R(\theta)$ and $C\Omega_n :=0$.
        
        \item For every $\beta\leq\theta$, $\Gamma_\theta\in R(\theta)$ and $C\Gamma_\theta:=0$.\\
        

        \item If $\alpha_1,\dots,\alpha_n\in R(\theta)$ and $\alpha=_{\sf NF}\alpha_1+\cdots+\alpha_n$ then $\alpha\in R(\theta)$ and $C\alpha=\mathrm{max}(C\alpha_1,\dots,C\alpha_n)+1$.

        \item If $\alpha_1,\alpha_2\in R(\theta)$, $\alpha_1,\alpha_2<\Omega_\omega$ and $\alpha=_{\sf NF}\phi\alpha_1\alpha_2$ then $\alpha\in R(\theta)$ and $C\alpha:=\mathrm{max}(C\alpha_1,C\alpha_2)+1$.
        
        \item If $\alpha_1\in R(\theta)$, $\alpha_1>\Omega_\omega$ and $\alpha=_{\sf NF}\phi0\alpha_1$ then $\alpha\in R(\theta)$ and $C\alpha:=C\alpha_1+1$.
       
        \item If $\alpha_1\in R(\theta)$ and $\alpha=_{\sf NF}\psi_n\alpha_1$ then $\alpha\in R(\theta)$ and $C\alpha:=C\alpha_1+1$.
    \end{enumerate}
\end{definition}

Any ordinal in $R(\theta)$ is either among $0$, $\Gamma_\beta$ for some $\beta\leq\theta$ and $\Omega_n$ for some $n\leq\omega$, or is included in $R(\theta)$ due to exactly one of the rules of this definition by Lemma $\ref{normal form}$. This shows that $C\alpha$ is uniquely determined for any $\alpha\in R(\theta)$. 
Now, we want to transform $R(\theta)$ into a recursive-in-$X$ representation system. The first problem is that we have to computably deal with the condition $\alpha_1\in B_n(\alpha_1)$ in Definition $\ref{nfdef}.3$. To do this, we define for each $n<\omega$ and each $\alpha\in R(\theta)$ the set of ordinals $\mathrm{Arg}_n(\alpha)$, that consists in all the ordinals that occur in the normal form of $\alpha$ as an argument of the $\psi_n$ function.

\begin{definition}
    Let $n<\omega$. We define for each $\alpha\in R(\theta)$ the set of ordinals $\mathrm{Arg}_n(\alpha)$ by induction on $C\alpha$ as follows
    \begin{enumerate}
        \item $\mathrm{Arg}_n(0)=\mathrm{Arg}_n(\Gamma_\beta)=\mathrm{Arg}_n(\Omega_m)=\emptyset$ for all $\beta\leq\theta$ and all $m\leq\omega$,
        \item If $\alpha=_{\sf NF}\alpha_1+\cdots+\alpha_m$ then $\mathrm{Arg}_n(\alpha)=\mathrm{Arg}_n(\alpha_1)\cup\dots\cup \mathrm{Arg}_n(\alpha_m)$,
        \item If $\alpha=_{\sf NF}\phi\alpha_1\alpha_2$ then $\mathrm{Arg}_n(\alpha)=\mathrm{Arg}_n(\alpha_1)\cup \mathrm{Arg}_n(\alpha_2)$,
        \item If $\alpha=_{\sf NF}\psi_m(\alpha_1)$ with $m\neq n$ then $\mathrm{Arg}_n(\alpha)=\mathrm{Arg}_n(\alpha_1)$,
        \item If $\alpha=_{\sf NF}\psi_n(\alpha_1)$ then $\mathrm{Arg}_n(\alpha)=\{\alpha_1\}\cup \mathrm{Arg}_n(\alpha_1)$.
    \end{enumerate}
\end{definition}
An easy induction on $C\alpha$ shows the next lemma, that gives a recursive equivalence to the condition $\alpha_1\in B_n(\alpha_1)$ in Definition $\ref{nfdef}$.
\begin{Lemma}
    Let $\alpha,\beta\in R(\theta)$. Let $n<\omega$. Then,
    \begin{center}
        $\alpha\in B_n(\beta)$ iff $\forall \delta\in \mathrm{Arg}_n(\alpha)\big(\delta<\beta\big)$.
    \end{center}
\end{Lemma}

We define $T(\theta)$ as the set of unique representations of ordinals in $R(\theta)$.
\begin{definition}[Ordinal notation system]\index[Symbols]{$T(\theta)$}
    We define $T(\theta)$ as the set of strings in the language $\{0,+,\phi\}\cup\{\Gamma_\beta:\beta<\theta\}\cup\{\Omega_n:n\leq\omega\}\cup\{\psi_n:n<\omega\}$ corresponding to ordinals in $R(\theta)$ written in normal form, as in Definition $\ref{nfdef}$.\\
    \ \\
    The order of strings in $T(\theta)$ is induced from the ordering of ordinals in $R(\theta)$. Let $\prec$ denote this order.
\end{definition}
Again, the following theorem can be proved by induction on $C\alpha$.
\begin{Theorem}
The mapping $\theta \mapsto T(\theta)$ is a primitive recursive set function.
\end{Theorem}
\begin{proof}
    The proof consists basically of formalizing all the previous definitions and results. We refer to \cite{cookrathj} (Theorem 2.13), where the authors give a $\theta$-primitive recursive procedure \textbf{A)} which decides for a string whether $a\in T'(\theta)$ and a $\theta$-primitive recursion procedure \textbf{B)} which decides for $a,b\in T'(\theta)$ with $a\neq b$ whether $a\prec b$ or $b\prec a$ for an ordinal notation system $T'(\theta)$ similar to ours. The proof should be adapted naturally, e.g., by adding conditions to \textbf{B)} to compare $\psi_n(a)$ and $\psi_m(a)$ by means of Lemma $\ref{311}$.
\end{proof}
\section{The infinitary proof system \texorpdfstring{$\RS$}{RSl(X)}}$\label{rs}$
In this section, we present our infinitary system $\RS$ and prove a cut-elimination result for it.
\subsection{The terms and formulas of \texorpdfstring{$\RS$}{RSl(X)}}
We fix a set $X$, and we follow Reading Convention $\ref{read}$. We will be using ordinals to a great extent in the construction of $\RS$, and we will only consider ordinals represented in $T(\theta)$.
\begin{readingconvention}$\label{ordnot}$
    All ordinals belong to $R(\theta)$, and thus are represented by strings in $T(\theta)$. We refrain from distinguishing between ordinals and notations.
\end{readingconvention}
We define the set $\mathcal{T}$ of $\RS$-terms, together with the level of a term, as follows.
\begin{definition}[Terms of $\RS$]$\label{term}$ 
\index[General]{$\RS$-terms}
\begin{itemize}
    \item $\overline{u}\in \mathcal{T}$ for every $u\in TC(\{X\})$ and $|\overline{u}|=\Gamma_{\mathrm{rank}(u)}$. Those are called basic terms.
    \item $\mathbb{L}_\alpha(X)\in\mathcal{T}$ for every $\alpha\leq\Omega_\omega$ and $|\mathbb{L}_\alpha(X)|=\Gamma_{\theta+1}+\alpha$.
    \item $[x\in\mathbb{L}_\alpha(X):B(x,s_1,\dots,s_n)^{\mathbb{L}_\alpha(X)}]\in\mathcal{T}$ for every $\alpha<\Omega_\omega$, for every {\sf KPl}-formula $B(x,y_1,\dots,y_n)$ and for every $s_1,\dots,s_n\in \mathcal{T}$ with $|s_1|,\dots,|s_n|<\Gamma_{\theta+1}+\alpha$. Moreover, $\big|[x\in\mathbb{L}_\alpha(X):B(x,s_1,\dots,s_n)^{\mathbb{L}_\alpha(X)}]\big|=\Gamma_{\theta+1}+\alpha$.
\end{itemize}
\end{definition}
Usually, we will just write $[x\in\mathbb L_\alpha(X):B(x)]$ for terms of the third kind. We observe that the cardinality of $\mathcal{T}$ is $\max\{\aleph_0,|\theta|\}$, by Reading Convention $\ref{ordnot}$.\\
Now, we define the $\RS$-formulas, together with their \textit{type}. The type of a formula is related to the rules of the system (see Definition $\ref{rules}$): a derivable formula will have $\bigvee$-type whenever we can derive it from a single derivable premise, and a derivable formula will have $\bigwedge$-type whenever we need (possibly infinitely) many derivable premises to derive the formula.
\begin{definition}[Formulas of $\RS$]\index[General]{$\RS$-formulas}
The $\RS$-formulas are exactly the {\sf KPl}-formulas replacing free variables by $\RS$-terms and restricting all unbounded quantifiers to $\RS$-terms.\\
\index[Symbols]{$\bigvee$-type and $\bigwedge$-type}The $\RS$-formulas of the form $\overline{u}\in\overline{v}$ or 
$\overline{u}\notin\overline{v}$ are called basic.\\
Moreover, each $\RS$-formula of the form $A\vee B$, $Ad(t)$, $\exists x\in t\ G(x)$ and non-basic $s\in t$ has $\bigvee$-type. Likewise, each $\RS$-formula of the form $A\wedge B$, $\neg Ad(t)$, $\forall x\in t\ G(x)$ and non-basic $s\notin t$ has $\bigwedge$-type.
\end{definition}
We observe that, by definition, there are no free variables nor unrestricted quantifiers in the $\RS$-formulas. From now on, we will call $\RS$-terms and $\RS$-formulas simply terms and formulas.
\begin{definition}\label{sigman}
    A formula $A(s_1,\dots,s_n)^{\mathbb{L}_{\Omega_n}(X)}$ is 
\index[Symbols]{$\Sigma^{\Omega_n}$-formula}$\Sigma^{\Omega_n}$ iff $A(x_1,\dots,x_n)$ is a {\sf KPl} $\Sigma$-formula and $|s_1|,\dots,|s_n|<\Omega_n$.
\end{definition}
We recall that the term $\Ln$ represents the set $L_{\Omega_n}(X)$ and, whenever a term has level below $\Omega_n$, then the term represents an element of $L_{\Omega_n}(X)$. Therefore, in essence, if a formula is $\Sigma^{\Omega_n}$ then it is a $\Sigma$ statement of $L_{\Omega_n}(X)$.\\
We want to keep track of the terms that appear in the formulas. This motivates the following definition.
\begin{definition}\index[Symbols]{$k(A)$} $\label{k}$
    For a formula $A$, we define
    \begin{center}
        $k(A)=\{|t|:t\text{ occurs in $A$ including subterms}\}$.
    \end{center}
    For a finite set of formulas $\Gamma$, we define
    \begin{equation*}
        k(\Gamma)=\bigcup_{A\in \Gamma}k(A).
    \end{equation*}
\end{definition}
For example, we have $k\big([x\in\Ln: x\in\Lm]\in \Lm\big)=\{\Omega_n,\Omega_m\}$.\\
To simplify notation, we use the following abbreviations as in \cite{cookrathj}.
\begin{definition}
$\label{dotin}$
Let $s$ and $t$ be terms such that $|s|<|t|$. For $\circ\in\{\wedge,\rightarrow\}$,\index[Symbols]{$\dot\in$} we define
 \begin{equation*}
    s\ \dot \in\  t\circ A(s,t):=\begin{cases}
        \overline{u}\in\overline{v}\circ A(\overline{u},\overline{v}) & \text{if $s\in t\equiv\overline{u}\in\overline{v}$},\\
        A(s,t) & \text{if $t=\mathbb{L}_\alpha(X)$},\\
        B(s)\circ A(s,t) & \text{if $t=[x\in\mathbb{L}_\alpha(X):B(x)]$}.
        
    \end{cases}
\end{equation*}
Furthermore, we recover the notation from Definition $\ref{equalconn}$ to write $s=t$, $\neg A$, $s\neq t$ and $A\rightarrow B$ in $\RS$.
\end{definition}
We will exhibit how the symbol $\dot\in$ is used in Example $\ref{exC}$.
\subsection{Operator-controlled derivations}

Now, we define the derivations of the $\RS$-proof system. Derivations will be controlled by operators, that are functions between sets of ordinals (as in Reading Convention $\ref{ordnot}$). We recall Reading Convention $\ref{read}$ so that the definition below depends on the set-theoretic rank $\theta$ of the fixed set $X$. 
\begin{definition}[Operator]\index[General]{Operator}\index[Symbols]{$\H$}
    Consider the set $\mathcal{P}\big(T(\theta)\big)=\{Y:Y\text{ is a set of ordinals in } T(\theta)\}$. An operator is a function $\H:\mathcal{P}\big(T(\theta)\big)\to \mathcal{P}\big(T(\theta)\big)$ such that for every $Y,Y'\in\mathcal{P}\big(T(\theta)\big)$ the following conditions are satisfied.
    \begin{enumerate}
        \item $\{0\}\cup\{\Gamma_\beta:\beta\leq\theta+1\}\cup\{\Omega_i:i\leq\omega\}\subseteq\H(Y)$.
        \item Let $\alpha=_{\sf NF}\alpha_1+\cdots+\alpha_n$. Then, $\alpha\in\H(Y)$ iff $\alpha_1,\dots,\alpha_n\in\H(Y)$.
        \item Let $\alpha=_{\sf NF}\phi\alpha_1\alpha_2$. Then, $\alpha\in \H(Y)$ iff $\alpha_1,\alpha_2\in\H(Y)$.
        \item $Y\subseteq \H(Y)$.
        \item If $Y\subseteq \H(Y')$ then $\H(Y)\subseteq\H(Y')$.
    \end{enumerate}
    Moreover, we will use the following abbreviations.
    \begin{itemize}
        \item $\H$ will often denote $\H(\emptyset)$.
        \item For a term $t$, $\H[t](Y)$ will mean $\H(Y\cup\{|t|\})$.
        \item For a formula $A$, $\H[A](Y)$ will mean $\H\big(Y\cup k(A)\big)$.
        \item For a finite set of formulas $\Gamma$, $\H[\Gamma](Y)$ will mean $\H\big(Y\cup k(\Gamma)\big)$.
    \end{itemize}
\end{definition}
We notice that our definition of operator depends on $\theta$. Likewise, the axioms and rules of $\RS$ depend on $X$ and $\theta$. It is easy to see that $\H(Y)=\H\big(\H(Y)\big)$ and that $\H(Y)$ is an extension of $Y$ that is closed under $\phi$ and $+$.
\begin{definition}[Rules of $\RS$]\index[Symbols]{$\H\dash{}{\alpha}\Gamma$}$\label{rules}$
Let $\H$ be an operator and let $\Gamma$ be a set of formulas. We have that $\Gamma$ is derived by an $\H$-controlled derivation with ordinal $\alpha$ whenever $\{\alpha\}\cup k(\Gamma)\subseteq\H$ and one of the following axioms or rules can be applied.\\  
\ \\
Axioms:
\begin{center}
$\H\dash{}{\alpha}\Gamma,\overline{u}\in\overline{v}$ for any $u,v\in TC(\{X\})$ such that $u\in v$,\\
$\H\dash{}{\alpha}\Gamma,\overline{u}\notin\overline{v}$ for any $u,v\in TC(\{X\})$ such that $u\notin v$.
\end{center}
\ \\
Rules:\\
\ \\
\begin{tabular}{b{10.5cm} m{3.5cm}}
\begin{prooftree}
\hypo{\H\dash{}{\alpha_0}\Gamma, A\wedge B,A}
\hypo{\H\dash{}{\alpha_1}\Gamma, A\wedge B,B}
\infer[left label =$(\wedge)$]2{\H\dash{}{\alpha}\Gamma,A\wedge B}
\end{prooftree}
 & $\alpha_0,\alpha_1<\alpha$
\\ 
\ \\
\ \\
    \begin{prooftree}
        \hypo{\H\dash{}{\alpha_0}\Gamma, A\vee B,A}
        \infer[left label=$(\vee)$]1{\H\dash{}{\alpha}\Gamma, A\vee B}
    \end{prooftree}
& $\alpha_0<\alpha$\\
\ \\
\ \\

    \begin{prooftree}
        \hypo{\H\dash{}{\alpha_0}\Gamma, A\vee B,B}
        \infer[left label=$(\vee)$]1{\H\dash{}{\alpha}\Gamma,A\vee B}
    \end{prooftree}
& $\alpha_0<\alpha$\\
\ \\
\ \\

\begin{prooftree}
        \hypo{\H\dash{}{\alpha_0}\Gamma, r\in t,s\ \dot \in\  t\wedge r=s}
        \infer[left label=$(\in)$]1{\H\dash{}\alpha\Gamma,r\in t}
    \end{prooftree}
    & 
    \ \newline
    $\alpha_0<\alpha$,\newline
    $|s|<|t|$,\newline
    $|s|<\Gamma_{\theta+1}+\alpha$,\newline
    $r\in t$ not basic.
\\
\ \\
\ \\
    \begin{prooftree}
        \hypo{\H[s]\dash{}{\alpha_s}\Gamma, r\notin t,s\ \dot \in\  t\rightarrow r\neq s\text{ for all $|s|<|t|$}}
        \infer[left label=$(\notin)$]1{\H\dash{}\alpha\Gamma, r\notin t}
    \end{prooftree}
& $\alpha_s<\alpha$,\newline
$r\in t$ not basic.\\
\ \\
\ \\
\begin{prooftree}
    \hypo{\H\dash{}{\alpha_0}\Gamma, \exists x\in t\ B(x),s\ \dot \in\  t\wedge B(s)}
    \infer[left label=$(b\exists)$]1{\H\dash{}\alpha\Gamma,\exists x\in t\ B(x)}
\end{prooftree}
& $\alpha_0<\alpha$,\newline
$|s|<|t|$,\newline
$|s|<\Gamma_{\theta+1}+\alpha$.\\
\ \\
\ \\
\begin{prooftree}
    \hypo{\H[s]\dash{}{\alpha_s}\Gamma, \forall x\in t\ B(x), s\ \dot \in\  t\rightarrow B(s)\text{ for all $|s|<|t|$}}
    \infer[left label=$(b\forall)$]1{\H\dash{}\alpha\Gamma,\forall x\in t\ B(x)}
\end{prooftree}
& $\alpha_s<\alpha$
\end{tabular}

\newpage

\begin{tabular}{b{10.5cm} m{3.5cm}}
   
\begin{prooftree}
    \hypo{\H\dash{}{\alpha_0} \Gamma, Ad(t), t= \mathbb{L}_{\Omega_n}(X)}
    \infer[left label = $(Ad)$]1{\H\dash{}\alpha\Gamma, Ad(t)}
\end{prooftree}
&
$\alpha_0<\alpha$,\newline
$n<\omega$,\newline
$\Omega_n<|t|$.
\\ 
\ \\
\ \\
\begin{prooftree}
    \hypo{\H\dash{}{\alpha_n}\Gamma, \neg Ad(t), t\neq \mathbb{L}_{\Omega_n}(X)\text{ for all $n\leq\omega$}}
    \infer[left label=$(\neg Ad)$]1{\H\dash{}\alpha\Gamma,\neg Ad(t)}
\end{prooftree}
& $\alpha_n<\alpha$\\
\ \\
\ \\
\begin{prooftree}
    \hypo{\H\dash{}{\alpha_0}\Gamma, \exists z\in\mathbb{L}_{\Omega_{n}}(X)\ A^z, A^{\Ln}}
    \infer[left label = $({\sf Ref}_{n}(X))$]1{\H\dash{}\alpha \Gamma,\exists z\in\mathbb{L}_{\Omega_{n}}(X)\ A^z}
\end{prooftree}
& $\alpha_0,\Omega_{n}<\alpha$,\newline
$n<\omega$,\newline
$A$ is a $\Sigma$ formula.\\
\\
\ \\
\ \\
\begin{prooftree}
    \hypo{\H\dash{}{\alpha_0}\Gamma, A}
    \hypo{\H\dash{}{\alpha_0}\Gamma, \neg A}
    \infer[left label = $(Cut)$]2{\H\dash{}\alpha\Gamma}
\end{prooftree}
& $\alpha_0<\alpha$

\end{tabular}\\
\end{definition}
\noindent
\textit{Besides $(Cut)$, each rule supplies a new formula in the conclusion. This formula is called the principal formula of the inference. Likewise, each rule withholds a (some) formula(s) of the premise(s). Those kind of formulas are called the active formulas of the derivation. Any other formula is called passive.}\\
\ \\
We take the convention of writing the principal formula in the premises of the rule to have access to the Weakening principle, that we will show in Lemma $\ref{weak}$.\\
We observe that when we make a derivation with principal formula of $\bigvee$-type, we only need a single premise, and when the principal formula has $\bigwedge$-type we need many premises to be derived. Given a formula $A$, we define the set of premises $\C(A)$. This set contains all the premises that allow (or are needed for) a derivation with principal formula $A$. 
\begin{definition}\index[Symbols]{$\C(A)$} $\label{CA}$
We define $\C(A)$ for a non-basic formula $A$ of $\bigvee$-type.\\
    $\C(r\in t)=\{s\ \dot \in\  t\wedge r=s:|s|<|t|\}$;\\
    $\C(A\vee B)=\{A,B\}$;\\
    $\C\big(Ad(t)\big)=\{t=\mathbb{L}_{\Omega_{n}}(X):n<\omega\wedge \Omega_{n}\leq|t|\}$;\\
    $\C\big(\exists x\in t\ A(x)\big)=\{s\ \dot \in\  t\wedge A(s):|s|<|t|\}$.\\
    \ \\
    Now, given a non-basic formula $A$ of $\bigwedge$-type, we define $\C(A)=\{\neg B:B\in\C(\neg A)\}$.
\end{definition}
Comparing Definitions $\ref{rules}$ and $\ref{CA}$, we notice that if $(R)$ is a rule with principal formula $A$, the active formula in the premise(s) of $(R)$ is (are) some (each) $B\in \C(A)$. Furthermore, the rules where the symbol $\dot\in$ appears implicitly correspond to three different rules, as in the following example. 
\begin{ex}\label{exC}
    The rule $(b\exists)$ has three different forms.
    \end{ex}
    \noindent
\begin{tabular}{b{8.5cm} m{6.5cm}}
\begin{prooftree}
    \hypo{\H\dash{}{\alpha_0}\Gamma,\exists x\in \overline u\ B(x),s\ \in\  \overline u\wedge B(s)}
    \infer[left label=$(b\exists)$]1{\H\dash{}\alpha\Gamma,\exists x\in \overline u\ B(x)}
\end{prooftree}
& $\alpha_0<\alpha$\newline
$|s|<\Gamma_{\mathrm{rank}(u)}$,\newline
$|s|<\Gamma_{\theta+1}+\alpha$.\\
\ & \ \\
\begin{prooftree}
    \hypo{\H\dash{}{\alpha_0}\Gamma, \exists x\in \mathbb L_\gamma(X)\ B(x),B(s)}
    \infer[left label=$(b\exists)$]1{\H\dash{}\alpha\Gamma,\exists x\in \mathbb L_\gamma(X)\ B(x)}
\end{prooftree}
& $\alpha_0<\alpha$,\newline
$|s|<\Gamma_{\theta+1}+\gamma$,\newline
$|s|<\Gamma_{\theta+1}+\alpha$.\\
\ & \ \\
\begin{prooftree}
    \hypo{\H\dash{}{\alpha_0}\Gamma,\exists x\in t\ B(x),F(s)\wedge B(s)}
    \infer[left label=$(b\exists)$]1{\H\dash{}\alpha\Gamma,\exists x\in t\ B(x)}
\end{prooftree}
& $t\equiv[x\in\mathbb L_\gamma(X):F(x)]$,\newline
$\alpha_0<\alpha$,\newline
$|s|<|t|$,\newline
$|s|<\Gamma_{\theta+1}+\alpha$.\\
\end{tabular}
\ \\
Now, coming back to Definition $\ref{CA}$, we observe that if $A$ is not a disjunction or a conjunction, each $B\in\C(A)$ is of the form $F(t)$ for a so-called \textit{characteristic term} $t$, where $t$ does not occur in $A$.
\begin{definition}\index[Symbols]{$t_A(B)$}
     If $A$ is a formula with $\bigwedge$-type different from a conjunction, we define $t_A(B):=t$ for $B\in\C(A)$, where $t$ is the characteristic term of $B\equiv F(t)$.\\
     If $A$ has $\bigvee$-type or is a conjunction, then we define $t_A(B):=\mathbb L_0(X)$ for any $B{\in}\C(A)$. 
\end{definition}
We conveniently define $t_A(B)$ this way for conjunctions and $\bigvee$-type formulas because it allows us to uniformize the controlling operator in the premise of a derivation. Indeed, we observe that in the rules deriving a formula with $\bigwedge$-type, the operator that controls the premise with active formula $B\in\C(A)$ is exactly $\H[t_A(B)]$, no matter the form of $A$. As $\H[\L_0(X)]=\H$ for any operator $\H$, the operator that controls any derivation with principal formula $A$ is always $\H[t_A(B)]$, no matter the form of the non-basic formula $A$. 
\section{Cut-elimination for \texorpdfstring{$\RS$}{RSl(X)}}\label{SectCutElimination}
Tipically, an $\RS$-proof will use the $(Cut)$ rule many times. In this section, we show that we can transform $\RS$-proofs of $\Sigma^m$-formulas into $\RS$-proofs without cuts.
\subsection{The rank of a formula}
We want to eliminate cuts from $\RS$-derivations. We introduce a complexity measure for formulas, the rank, by recursion.
\begin{definition}[Rank]$\label{rank}$\index[Symbols]{$\mathrm{rk}(\cdot$)}\index[General]{Rank}We define the rank of a term or formula by recursion.
    \begin{itemize}
        \item $\mathrm{rk}(\overline{u})=\Gamma_{\mathrm{rank}(u)}$,
        \item $\mathrm{rk}\big(\mathbb{L}_\alpha(X)\big)=\Gamma_{\theta+1}+\omega\cdot\alpha$,
        \item $\mathrm{rk}\big([x\in\mathbb{L}_\alpha(X):B(x)]\big)=\mathrm{max}\Big(\Gamma_{\theta+1}+\omega\cdot\alpha+1,\mathrm{rk}\big(B(\overline{\emptyset})\big)+2\Big)$,\\
        \item $\mathrm{rk}(s\in t)=\mathrm{rk}(s\notin t)=\mathrm{max}\big(\mathrm{rk}(s)+6,\mathrm{rk}(t)+1\big)$,
        \item $\mathrm{rk}\big(Ad(t)\big)=\mathrm{rk}\big(\neg Ad(t)\big)=\mathrm{rk}(t)+5$,
        \item $\mathrm{rk}(A\vee B)=\mathrm{rk}(A\wedge B)=\mathrm{max}\big(\mathrm{rk}(A),\mathrm{rk}(B)\big)+1$,
        \item $\mathrm{rk}\big(\exists x\in t\ A(x)\big)=\mathrm{rk}\big(\forall x\in t\ A(x)\big)=\mathrm{max}\Big(\mathrm{rk}(t),\mathrm{rk}\big(A(\overline{\emptyset})\big)+2\Big)$.
    \end{itemize}
\end{definition}

As a useful example, we compute the rank of $s=t$. Since $s=t$ is an abbreviation for $\forall x\in s(x\in t)\wedge\forall x\in t(x\in s)$, we have that 
\begin{equation*}
    \begin{split}
        \mathrm{rk}(s=t) & = \mathrm{max}\Big(\mathrm{rk}\big(\forall x\in s(x\in t)\big),\mathrm{rk}\big(\forall x\in t(x\in s)\big)\Big)+1\\
        & = \mathrm{max}\Big(\mathrm{max}\big(\mathrm{rk}(s),\mathrm{rk}(t)+3\big),\mathrm{max}\big(\mathrm{rk}(t),\mathrm{rk}(s)+3\big)\Big)+1\\
        & = \mathrm{max}\big(\mathrm{rk}(s), \mathrm{rk}(t)\big)+4.
    \end{split}
\end{equation*}
We will prove in Lemma $\ref{rk}$ that, given a formula $A$, the complexity of a characteristic formula $B\in \C(A)$ is always below the complexity of $A$. First, we need the following technical results (proved in \cite{tfm}).
\begin{Lemma}$\label{rklemma}$
    \begin{enumerate}
        \item Let $t$ be any term. Then there is $n<\omega$ such that $\mathrm{rk}(t)=\omega\cdot|t|+n$.
        \item Let $A$ be any formula. Then there is $n<\omega$ such that $\mathrm{rk}(A)=\omega\cdot \mathrm{max}\big(k(A)\big)+n$.
        \item Let $A$ be any formula and $s$ be any term. If $|s|<\mathrm{max}\Big(k\big(A(s)\big)\Big)$ then $\mathrm{rk}\big(A(s)\big)=\mathrm{rk}\big(A(\overline\emptyset)\big)$.
    \end{enumerate}
\end{Lemma}
We use this result to prove the following lemma.
\begin{Lemma}
$\label{rk}$
    Let $A$ be any non-basic formula and let $B\in\C(A)$. Then, we have
    \begin{equation*}
        \mathrm{rk}(B)<\mathrm{rk}(A).
    \end{equation*}
\end{Lemma}
\begin{proof}
    We suppose that $A$ has $\bigvee$-type and we proceed by induction on the construction of $A$.\\
    \ \\
    Case 1. We suppose $A\equiv s\in t$. We split the cases depending on the form of $t$.\\
    Subcase 1.1. We assume $t\equiv\overline{u}$. Then $s$ is not basic since $A$ is not basic by assumption. Therefore, $\mathrm{rk}(A)=\mathrm{rk}(s)+6$ and if $B\in\C(A)$ then $B$ is of the form $r\in\overline{u}\wedge r=s$ for some $r$ with $|r|<|\overline{u}|$. It follows from this last condition that $r$ is basic, say $r\equiv\overline{v}$, and then
    \begin{equation*}
        \begin{split}
            \mathrm{rk}(B) & = \mathrm{max}\big(\mathrm{rk}(\overline{v}\in\overline{u}),\mathrm{rk}(\overline{v}=s)\big)+1\\
            & =\mathrm{rk}(\overline{v}=s)+1\\
            & =\mathrm{rk}(s)+5\\
            & <\mathrm{rk}(s)+6=\mathrm{rk}(A). 
        \end{split}
    \end{equation*}
    Subcase 1.2. We assume $t\equiv\La$. Then $\mathrm{rk}(A)=\mathrm{max}\big(\mathrm{rk}(s)+6,\mathrm{rk}(t)+1\big)$. If $B\in\C(A)$, then $B$ is of the form $r=s$ for some $r$ with $|r|<|t|$. Therefore, 
        \begin{equation*}
            \mathrm{rk}(B) = \mathrm{max}\big(\mathrm{rk}(r),\mathrm{rk}(s)\big)+4< \mathrm{max}\big(\mathrm{rk}(t)+1,\mathrm{rk}(s)+6\big)=\mathrm{rk}(A).
        \end{equation*}
    Subcase 1.3. We assume $t\equiv [x\in\mathbb{L}_\alpha(X):F(x)]$. Then $\mathrm{rk}(A)=\mathrm{max}\big(\mathrm{rk}(s)+6,\mathrm{rk}(t)+1\big)$. If $B\in\C(A)$, then $B$ is of the form $F(r)\wedge s=r$ for some $r$ with $|r|<|t|$. Therefore,
        \begin{equation*}
        \begin{split}
            \mathrm{rk}(B) & = \mathrm{max}\Big(\mathrm{rk}\big(F(r)\big),\mathrm{rk}(s=r)\Big)+1\\
            & = \mathrm{max}\Big(\mathrm{rk}\big(F(r)\big),\mathrm{rk}(s)+4,\mathrm{rk}(r)+4\Big)+1.
        \end{split}
    \end{equation*}
    But $\mathrm{rk}(s)+5,\mathrm{rk}(r)+5<\mathrm{max}\big(\mathrm{rk}(s)+6,\mathrm{rk}(t)+1\big)=\mathrm{rk}(A)$. Also, $\mathrm{rk}\big(F(r)\big)<\mathrm{rk}(t)$. Indeed, if $\mathrm{max}\Big(k\big(F(r)\big)\Big)\leq |s|$, then $\mathrm{rk}\big(F(r)\big)+1<\omega\cdot|s|+\omega\leq \mathrm{rk}(t)$ by Lemma $\ref{rklemma}.2$; if $|s|<\mathrm{max}\Big(k\big(F(r)\big)\Big)$, then by Lemma $\ref{rklemma}.3$ we have
    \begin{equation*}
    \mathrm{rk}\big(F(r)\big)+1=\mathrm{rk}\big(F(\overline{\emptyset})\big)+1<\mathrm{max}\Big(\Gamma_{\theta+1}+\omega\cdot\alpha+1,\mathrm{rk}\big(F(\overline{\emptyset})\big)+2\Big)=\mathrm{rk}(t).
    \end{equation*}
    Hence, from $\mathrm{rk}\big(F(r)\big)<\mathrm{rk}(t)$ we get $\mathrm{rk}\big(F(r)\big)+1<\mathrm{rk}(A)$. Gathering everything, we obtain 
    \begin{equation*}
        \mathrm{rk}(B)=\mathrm{max}\Big(\mathrm{rk}\big(F(r)\big),\mathrm{rk}(s)+4,\mathrm{rk}(r)+4\Big)+1<\mathrm{rk}(A).
    \end{equation*}
    Case 2. We suppose $A\equiv B\vee C$. Then $\mathrm{rk}(A)=\mathrm{max}\big(\mathrm{rk}(B),\mathrm{rk}(C)\big)+1$ and so $\mathrm{rk}(B),\mathrm{rk}(C)<\mathrm{rk}(A)$.\\
    \ \\
    Case 3. We suppose $A\equiv\exists x\in t\ F(x)$. We omit the details, but, as in Case 1, subcases are splitted based on the form of $t$.\\ 
    \ \\
    An analogous proof shows that the result holds if $A$ has $\bigwedge$-type.
\end{proof}

We will write the complexity of the formulas that have been removed by an application of the $(Cut)$ rule in some derivation as a subscript, as follows. 
\begin{definition}\index[Symbols]{$\H\dash{\rho}{\alpha}\Gamma$}
    We will write $\H\dash{\rho}\alpha\Gamma$ whenever $\H\dash{}\alpha\Gamma$ and all the active formulas of the premise of an inference using $(Cut)$ occurring in the derivation have rank strictly less than $\rho$. In this case, we will say that the cut complexity of the derivation is bounded by $\rho$.
\end{definition}
\subsection{Predicative results of \texorpdfstring{$\RS$}{RSl(X)}}
Our objective is to eliminate cuts from derivations. First, we show some immediate results about operator controlled derivations that will be useful to prove the cut-elimination lemmas.\\
Henceforth, we will use $y$ as the set of indices of the possibly many premises used to derive some conclusion. The next lemma states that an operator appearing in the conclusion of an inference is always contained in the operator(s) appearing in the premise(s).

\begin{Lemma}
    Let $\H$ be any operator. We suppose that $\H\dash{\rho}{\alpha} \Gamma$ can be obtained by the application of an inference rule with premises $\H_i\dash{\rho_i}{\alpha_i}\Gamma_i$ with $\alpha_i<\alpha$ and $\H_i$ operators, for $i\in y$. Then, $\H\subseteq \H_i$.
\end{Lemma} 
\begin{proof}
    First, we suppose that the principal formula is $A\in\Gamma$. If $A$ has $\bigvee$-type, then there is only one premise, with control operator $\H_i=\H$. If $A$ has $\bigwedge$-type, then $\H_i=\H[t_A(B)]\supseteq \H$ for some $B\in\C(A)$.\\
Now, if $\H\dash\rho\alpha\Gamma$ has been obtained by (Cut) then $\H_i=\H$ in both premises. If the derivation has been obtain by applying (${\sf Ref}_{n}(X)$) for some $n<\omega$ then $\H_i=\H$ in the unique premise.
\end{proof}
Now, we prove a version of the weakening principle.
\begin{Lemma}
$\label{weak}$
    Let $\H$ be an operator. Let $\alpha,\alpha',\rho$ and $\rho'$ be ordinals. Let $\Delta$ and $\Gamma$ be finite sets of formulas. If $\alpha\leq\alpha'\in\H$, $\rho\leq\rho'$, $k(\Delta)\subseteq\H$ and $\H\dash\rho\alpha\Gamma$ then $\H\dash{\rho'}{\alpha'}\Gamma,\Delta$.
\end{Lemma}
\begin{proof}
    We proceed by induction on $\alpha$. If $\Gamma$ is an axiom, then $\Gamma,\Delta$ is an axiom too and, since $\{\alpha'\}\cup k(\Gamma\cup\Delta)\subseteq \H$ and an axiom has no cuts (which means that the complexity of the cut formulas is bounded by any ordinal), we have that $\H\dash{\rho'}{\alpha'}\Gamma,\Delta$.\\
    We suppose that $\Gamma$ has been derived by the application of a rule $(R)$ and consider cases based on $(R)$.\\
\ \\
    Case 1. We assume that $(R)$ is $(Cut)$. Then we have $\H\dash{\rho}{\alpha}\Gamma$ and the premises are
    \begin{equation}\label{A}
        \H\dash{\rho}{\alpha_0}\Gamma, A\text{, and}
    \end{equation}
    \begin{equation}\label{noA}
        \H\dash{\rho}{\alpha_0}\Gamma,\neg A
    \end{equation}
    with $\alpha_0<\alpha<\alpha'$ for some formula $A$ with $\mathrm{rank}(A)<\rho$.\\
    First, we assume that $A$ or $\neg A$ are in $\Delta$. Then, we use the induction hypothesis on $\eqref{A}$ or $\eqref{noA}$ respectively to obtain
    \begin{equation*}
        \H\dash{\rho'}{\alpha'}\Gamma,\Delta.
    \end{equation*}
    We assume now that $A$ and $\neg A$ are not in $\Delta$.
    Then, since $\alpha_0<\alpha$, by the induction hypothesis, we have $\H\dash{\rho'}{\alpha}\Gamma,\Delta, A$ and $\H\dash{\rho'}{\alpha}\Gamma,\Delta,\neg A$.
    We apply $(Cut)$, observing that the cut complexity of the derivation does not increase since $\mathrm{rank}(A)<\rho<\rho'$, and we obtain $\H\dash{\rho'}{\alpha'}\Gamma, \Delta,\neg A$.\\
    \ \\
    Case 2. Suppose that $A\in\Gamma$ is the principal formula of the last derivation obtained by an application of the rule $(R)$ and $A$ has $\bigvee$-type. Then, we have $\H\dash{\rho}{\alpha_0}\Gamma, A, B$ for some $B\in \C(A)$ (we write $\Gamma, A$, which is the same as $\Gamma$, to make explicit the principal formula $A$). If $B\in\Delta$, by the induction hypothesis, we get
    \begin{equation*}
        \H\dash{\rho'}{\alpha'}\Gamma, \Delta, A.
    \end{equation*}
     We assume that $B\notin \Delta$. Then, the induction hypothesis gives
     \begin{equation*}
         \H\dash{\rho'}{\alpha}\Gamma, \Delta, A,B.
     \end{equation*}
     We use the rule $(R)$ to obtain
    \begin{equation*}
        \H\dash{\rho'}{\alpha'}\Gamma,\Delta,A.
    \end{equation*}

    The cases where the principal formula of the last derivation has $\bigwedge$-type and the case where the last rule is some $({\sf Ref_n})$ are analogous.
\end{proof}
$\setcounter{equation}{0}$
Now that we have proved that the Weakening rule is admissible in $\RS$, we will drop the repetition of the principal formula of a rule in the premises to simplify proofs following the next reading convention.
\begin{readingconvention}$\label{read2}$
    We suppose that we have the following inference with principal formula $A$ by applying the rule $(R)$.
    \begin{equation*}
        \begin{prooftree}
            \hypo{\H\dash{}{\alpha_i}\Gamma_i, A}
            \infer[left label = $(R)$]1{\H\dash{}{\alpha}\Gamma, A}
        \end{prooftree}
    \end{equation*}
    We will often omit the repeated formula $A$ in the premises and, instead, write
    \begin{equation*}
        \begin{prooftree}
            \hypo{\H\dash{}{\alpha_i}\Gamma_i}
            \infer[left label = $(R)$]1{\H\dash{}{\alpha}\Gamma, A}
        \end{prooftree}
    \end{equation*}
\end{readingconvention}

We can prove some kind of Inversion for $\bigwedge$-type formulas.
\begin{Lemma}[Inversion]
$\label{inversion}$
    Let $\H$ be any operator. Let $A$ be a $\bigwedge$-type formula and let $\Gamma$ be a finite set of formulas. Let $\alpha$ and $\rho$ be ordinals. If $\H\dash\rho\alpha\Gamma, A$ then $\H[t_A(B)]\dash\rho{\alpha}\Gamma, B$ for every $B\in\C(A)$.
\end{Lemma}
\begin{proof}
    We proceed by induction on $\alpha$. If $\Gamma, A$ is an axiom, then $\Gamma$ is an axiom since $A$ is non-basic. Therefore, $\Gamma, B$ is an axiom for each $B\in\C(A)$ and so $\H[t_B(A)]\dash\rho\alpha \Gamma, B$.\\
    We suppose that $\Gamma,A$ has been obtained by the application of a rule with $A$ not principal. Then, we can apply the induction hypothesis to the premises of this inference and use the rule again to obtain the result.\\
    We suppose that $\Gamma,A$ has been obtain by the application of a rule $(R)$ with principal formula $A$. This means that, for each $B\in\C(A)$, we have 
    \begin{equation*}
        \H[t_A(B)]\dash\rho{\alpha_B} \Gamma,A, B
    \end{equation*}
    with $\alpha_B<\alpha$. By the induction hypothesis, for every $B\in\C(A)$ we have
    \begin{equation}\label{BB}
        \H[t_A(B)][t_A(B)]\dash\rho{\alpha_B}\Gamma, B, B.
    \end{equation}
    But $\H[t_A(B)][t_A(B)]=\H[t_A(B)]$ since $t_A(B)\in\H[t_A(B)]$, and so $\eqref{BB}$ is exactly
    \begin{equation*}
        \H[t_A(B)]\dash\rho{\alpha_B}\Gamma, B.
    \end{equation*}
    By Lemma $\ref{weak}$, we obtain
    \begin{equation*}
        \H[t_A(B)]\dash\rho\alpha \Gamma, B
    \end{equation*}
    for every $B\in\C(A)$.
\end{proof}
$\setcounter{equation}{0}$
We are now ready to start the proof of the Predicative Cut Elimination Theorem. To prepare for the proof, we first show the Reduction Lemma. This result will only be used to simplify the proof of the upcoming Predicative Cut Elimination Theorem and will not appear further in the paper.
\begin{Lemma}[Reduction]
$\label{reduction}$
    Let $\H$ be any operator. Let $\alpha$ be an ordinal. Let $\Gamma$ and $\Delta$ be finite sets of formulas. Let $A$ be a formula with $\mathrm{rk}(A)=\rho\notin \{\Omega_{n}:n<\omega\}$. If both $\H\dash\rho\alpha\Gamma, \neg A$ and $\H\dash\rho\beta \Delta, A$ hold, then $\H\dash\rho{\alpha+\beta} \Gamma,\Delta$ also holds.    
\end{Lemma}
\begin{proof}
    We will consider cases based on the form of $A$. Before considering the case where $A$ is a basic formula, we show the following claim:
    \begin{claim}$\label{claim 1}$
    \begin{enumerate}
        \item If $u\in v$ is true in $TC(\{X\})$ and $\H\dash\rho\alpha \Gamma,\neg \overline u\in\overline v$ then $\H\dash\rho\alpha\Gamma$.
        \item If $u\notin v$ is true in $TC(\{X\})$ and $\H\dash\rho\alpha\Gamma,\overline{u}\in\overline{v}$, then $\H\dash\rho\alpha\Gamma$.
    \end{enumerate}
    \end{claim}
    \renewcommand\qedsymbol{$\blacksquare$}
    \begin{proof}
        We prove Item 1. of Claim $\ref{claim 1}$ by induction on $\alpha$. If $\Gamma,\neg \overline u\in \overline v$ is an axiom, then $\Gamma$ is an axiom since $u\notin v$ is false in $TC(\{X\})$, and so $\H\dash\rho\alpha \Gamma$ holds.\\
        We suppose that $\Gamma,\neg \overline u\in \overline v$ was derived by an application of a rule $(R)$. Then $\neg \overline u\in \overline v$ is a passive formula since it is basic. Therefore, we have (possibly many) premises of the form $\H_i\dash\rho{\alpha_i}\Gamma_i,\neg \overline u\in \overline v$ with $\alpha_i<\alpha$ for every $i\in y$. By the induction hypothesis, we obtain $\H_i\dash\rho{\alpha_i}\Gamma_i$, with $\alpha_i<\alpha$, for every $i\in y$. Finally, we apply the rule $(R)$ to get $\H\dash\rho\alpha\Gamma$.\\
    The proof of Item $2.$ is analogous, and so Claim $\ref{claim 1}$ is shown.        
    \end{proof}
    \renewcommand\qedsymbol{$\square$}
    We now start the proof of the Reduction Lemma.\\
    We suppose that $A$ is a literal. Without loss of generality, we can suppose that $A\equiv \overline u\in\overline v$. Then, by Claim $\ref{claim 1}$, we have either $\H\dash\rho\alpha\Gamma$ or $\H\dash\rho\beta\Delta$ depending on whether $u\in v$ holds in $TC(\{X\})$ or not. Therefore, by Weakening (Lemma $\ref{weak}$) we obtain $\H\dash\rho{\alpha+\beta}\Gamma,\Delta$.\\
    \ \\
    We suppose now that $A$ has $\bigvee$-type. We have
    \begin{equation}\label{VA}
        \H\dash\rho\alpha\Gamma,\neg A,
    \end{equation}
    \begin{equation}\label{VnoA}
        \H\dash\rho\beta\Delta, A.
    \end{equation}
    We proceed by induction on $\beta$. If $\Delta, A$ is an axiom, then $\Delta$ is an axiom since $A$ is non-basic and so $\Gamma,\Delta$ is also an axiom, showing that $\H\dash\rho{\alpha+\beta} \Gamma,\Delta$ holds.\\
    \ \\
    Now, we assume that $\Delta, A$ has been obtained by the application of a rule $(R)$. We suppose that $A$ is not the principal formula in this last derivation. Then, we can apply the induction hypothesis to the premises and apply again the rule $(R)$ to obtain the result.\\
    We suppose now that $A$ is the principal formula of the last derivation. Without loss of generality, we can assume that $A$ has $\bigvee$-type. We notice that $(R)$ cannot be any $({\sf Ref}_{n})$ rule since, if it were, we would have $\mathrm{rk}(A)=\Omega_{n}$, against the hypothesis. Therefore, the rule $(R)$ is either $(\vee)$, $(\in)$, $(b\exists)$ or $(Ad)$, and we have the only premise 
    \begin{equation}\label{prem}
       \H\dash\rho{\beta_0}\Delta, A,B,
    \end{equation}
    with $\beta_0<\beta$. We apply the induction hypothesis to $\eqref{VA}$ and $\eqref{prem}$ to obtain
     \begin{equation}\label{noB}
         \H\dash\rho{\alpha+\beta_0}\Gamma,\Delta, B.
     \end{equation}
     On the other hand, from $\eqref{VA}$, we use Lemma $\ref{inversion}$ (Inversion) and we get 
     \begin{equation}\label{invnoB}
         \H[t_A(B)]\dash\rho\alpha\Gamma,\neg B.
     \end{equation}
     But $t_A(B)\in\H$ by $\eqref{noB}$, and so $\H[t_A(B)]=\H$. Thus, Lemma $\ref{weak}$ on $\eqref{invnoB}$ gives
     \begin{equation}\label{weaknoB}
         \H\dash\rho{\alpha+\beta_0}\Gamma,\Delta,\neg B.
     \end{equation}
     Finally, we apply (Cut) to $\eqref{noB}$ and $\eqref{weaknoB}$ to obtain $\H\dash\rho{\alpha+\beta}\Gamma,\Delta$. We note that, by Lemma $\ref{rk}$, we have $\mathrm{rk}(B)<\mathrm{rk}(A)=\rho$ and so the complexity of the cuts done in the derivation is still bounded by $\rho$.
\end{proof}
$\setcounter{equation}{0}$
Now, we state and prove the Predicative Cut Elimination Theorem. If we have a derivation $\H\dash{\rho}{\alpha}\Gamma$ with $\rho<\Omega_k$ where $k=\min(n<\omega:\rho<\Omega_n)$, this results allows us to lower the bound of the complexity of the cuts up to $\Omega_{k-1}+1$, or to $0$ if $k=0$. 
\begin{Theorem}[Predicative Cut Elimination]
$\label{cut elim}$
    Let $\H$ be any operator closed under $\phi$. Let $\alpha\in\H$ and $\rho$ be ordinals such that $\Omega_{n}\notin[\rho,\rho+\omega^\alpha)$ for any $n<\omega$. We have that if $\H\dash{\rho+\omega^\alpha}\beta\Gamma$ then $\H\dash\rho{\phi\alpha\beta}\Gamma$.   
\end{Theorem}
\begin{proof}
    First of all, we observe that $\phi\alpha\beta\in\H$ since $\alpha,\beta\in \H$ and so the ordinal bound of the derivation appearing in the conclusion of the theorem is coherent. We proceed by induction on $\alpha$ with subsidiary induction on $\beta$. If $\Gamma$ is an axiom then trivially $\H\dash\rho{\phi\alpha\beta}\Gamma$.\\
    We suppose that $\Gamma$ has been obtained by the application of a rule $(R)$. We distinguish cases according to whether $(R)$ is the rule (Cut) or $(R)$ is any other rule.\\
    \ \\
    Case 1. We suppose that $(R)$ is not (Cut). Then we have the premise(s) $\H_i\dash{\rho+\omega^\alpha}{\beta_i}\Gamma_i$ with $\beta_i<\beta$ for each $i\in y$. By the subsidiary induction hypothesis, we get $\H_i\dash\rho{\phi\alpha\beta_i}\Gamma_i$ and, since $\phi\alpha\beta_i<\phi\alpha\beta$ for all $i\in y$ we obtain $\H\dash\rho{\phi\alpha\beta}\Gamma$ by an application of $(R)$.\\
    \ \\
    Case 2. We suppose now that $(R)$ is (Cut). This means that the premises are 
    \begin{equation}\label{cutelim1}
        \H\dash{\rho+\omega^\alpha}{\beta_0} \Gamma, B\text{ with $\beta_0<\beta$}
    \end{equation}

    \begin{equation}\label{cutelim2}
        \H\dash{\rho+\omega^\alpha}{\beta_0}\Gamma,\neg B\text{ with $\beta_0<\beta$}
    \end{equation}
        for some formula $B$. By the subsidiary induction hypothesis applied to $\eqref{cutelim1}$ and \eqref{cutelim2}, we have
    \begin{equation}\label{cutelim3}
        \H\dash\rho{\phi\alpha\beta_0}\Gamma,B
    \end{equation}
    \begin{equation}\label{cutelim4}
        \H\dash\rho{\phi\alpha\beta_0}\Gamma,\neg B.
    \end{equation}
    We observe that $\mathrm{rk}(B)<\rho+\omega^\alpha$ (if the rank of $B$ was greater than $\rho+\omega^\alpha$ we could not have derived $\Gamma$ with cuts of complexity bounded by $\rho+\omega^\alpha$). If $\mathrm{rk}(B)<\rho$, then we can apply (Cut) to \eqref{cutelim3} and \eqref{cutelim4} and get $\H\dash\rho{\phi\alpha\beta}\Gamma$ since $\phi\alpha\beta_0<\phi\alpha\beta$.\\
    If $\mathrm{rk}(B)\in[\rho,\rho+\omega^\alpha)$ we cannot apply (Cut) efficiently because it would increase the complexity of the cuts. In this case, we write $\mathrm{rk}(B)=\rho+\omega^{\alpha_1}+\cdots+\omega^{\alpha_n}$ for some $\alpha>\alpha_1\geq\dots\geq\alpha_n$.  By Lemma $\ref{weak}$ on $\eqref{cutelim3}$ and \eqref{cutelim4} we are able to have the cut complexity bound equal to the rank of $B$:
    \begin{equation}\label{cutelim5}
        \H\dash{\rho+\omega^{\alpha_1}+\cdots+\omega^{\alpha_n}}{\phi\alpha\beta_0} \Gamma,B
    \end{equation}
    \begin{equation}\label{cutelim6}
        \H\dash{\rho+\omega^{\alpha_1}+\cdots+\omega^{\alpha_n}}{\phi\alpha\beta_0}\Gamma,\neg B.
    \end{equation}
    We observe that either $B$ or $\neg B$ has $\bigvee$-type, and so we can apply Lemma $\ref{reduction}$ (Reduction) to \eqref{cutelim5} and \eqref{cutelim6} and obtain
    \begin{equation*}
        \H\dash{\rho+\omega^{\alpha_1}+\cdots+\omega^{\alpha_n}}{\phi\alpha\beta_0+\phi\alpha\beta_0}\Gamma.
    \end{equation*}
    By Lemma $\ref{weak}$, since $\phi\alpha\beta_0+\phi\alpha\beta_0<\phi\alpha\beta$ and $\phi\alpha\beta\in\H$, we get 
    \begin{equation}\label{cutelim7}
        \H\dash{\rho+\omega^{\alpha_1}+\cdots+\omega^{\alpha_n}}{\phi\alpha\beta}\Gamma.
    \end{equation}
    Now, since $\rho+\omega^{\alpha_1}+\cdots+\omega^{\alpha_n}=\mathrm{rk}(B)\in \H$ (this comes from the fact that $k(B)\subseteq\H$ by \eqref{cutelim3} and $\mathrm{rk}(B)=\omega\cdot \mathrm{max}(k(B))+n'$ for some $n'<\omega$), we have that $\alpha_1,\dots,\alpha_n\in\H$. Therefore, as $\alpha_n<\alpha$, we can apply the principal induction hypothesis to \eqref{cutelim7} and obtain 
    \begin{equation}\label{cutelim8}
        \H\dash{\rho+\omega^{\alpha_1}+\cdots+\omega^{\alpha_{n-1}}}{\phi_{\alpha_n}(\phi\alpha\beta)}\Gamma.
    \end{equation}
    But $\alpha_n<\alpha$ shows that $\phi_{\alpha_n}(\phi\alpha\beta)=\phi\alpha\beta$ by Veblen functions' properties, and so \eqref{cutelim8} is exactly
    \begin{equation*}
        \H\dash{\rho+\omega^{\alpha_1}+\cdots+\omega^{\alpha_{n-1}}}{\phi\alpha\beta}\Gamma.
    \end{equation*}
    Repeating this $n-1$ times, we finally obtain
    \begin{equation*}
        \H\dash{\rho}{\phi\alpha\beta}\Gamma.
    \end{equation*}
\end{proof}
$\setcounter{equation}{0}$
We prove the last predicative result for $\RS$, the Boundedness Lemma.
\begin{Lemma}[Boundedness]
$\label{boundedness}$
    Let $\H$ be any operator. Let $n$ be any natural number. Let $\rho$ be an ordinal. Let $A^{\mathbb{L}_{\Omega_n}(X)}$ be a $\Sigma^{\Omega_n}$ formula and let $\alpha,\beta$ be ordinals such that $\beta\in \H$ and $\alpha\leq\beta<\Omega_n$. If $\H\dash\rho\alpha\Gamma, A^{\mathbb{L}_{\Omega_n}(X)}$ then $\H\dash\rho\alpha\Gamma, A^{{\mathbb{L}_{\beta}(X)}}$.    
\end{Lemma}
\begin{proof}
    We proceed by induction on $\alpha$. If $\Gamma,A^{\mathbb{L}_{\Omega_n}(X)}$ is an axiom then $\Gamma, A^{\mathbb{L}_{\beta}(X)}$ is also an axiom (if $\Gamma$ is an axiom this is clear; if $A^{\mathbb{L}_{\Omega_n}(X)}$ is an axiom then $A$ has no quantifiers and therefore $A^{\mathbb{L}_{\Omega_n}(X)}\equiv A\equiv A^{\mathbb{L}_{\beta}(X)}$).\\
    We suppose that the derivation has been obtained by the application of a rule $(R)$. If $A^{\mathbb{L}_{\Omega_n}(X)}$ is not the principal formula, then we can apply the induction hypothesis to the premise(s) and use the rule $(R)$ again to obtain the result. So assume that $A^{\mathbb{L}_{\Omega_n}(X)}$ is the principal formula. We distinguish 3 cases: the rule $(R)$ is not $({\sf Ref}_n)$ and the formula $A^{\Ln}$ has $\bigwedge$-type, the rule $(R)$ is not $({\sf Ref}_n)$ and the formula $A^{\Ln}$ has $\bigvee$-type, or the rule $(R)$ is $({\sf Ref}_n)$. We observe that the last rule can not be $({\sf Ref}_m)$ with $m>n$ since all the terms that appear in $A^{\Ln}$ have level less than $\Omega_n$. The case where $(R)$ is $({\sf Ref}_m)$ with $m<n$ falls within the Case 2, where $A^{\Ln}$ has $\bigvee$-type.\\
    \ \\
    Case 1. We suppose that $A^{\Ln}$ has $\bigwedge$-type. We notice that, since $A^{\Ln}$ is $\Sigma^{\Ln}$, every formula in $\C(A^{\Ln})$ is also $\Sigma^{\Omega_n}$. This means that formulas in $\C(A^{\Ln})$ are of the form $B^{\Ln}$ and $B^{\mathbb L_\beta(X)}\in\C(A^{\mathbb L_\beta(X)})$ for every $B^{\Ln}\in\C(A^{\Ln})$. So, we have
    \begin{equation*}
        \H[t_{A^{\Ln}}(B^{\Ln})]\dash{\rho}{\alpha_B}\Gamma, A^{\Ln}, B^{\Ln}.
    \end{equation*}
    with $\alpha_B<\alpha$, for every $B^{\Ln}\in \C(A^{\Ln})$. Since $\alpha_B<\alpha$, we can use the induction hypothesis, that gives
    \begin{equation*}
        \H[t_{A^{\Ln}}(B^{\Ln})]\dash{\rho}{\alpha_B}\Gamma, A^{\mathbb L_\beta(X)}, B^{\mathbb L_\beta(X)}.
    \end{equation*}
    Applying the rule $(R)$, we obtain as desired
    \begin{equation*}
        \H\dash{\rho}{\alpha}\Gamma,A^{\mathbb L_\beta(X)}.
    \end{equation*}
    Case 2. We suppose that $A^{\Ln}$ has $\bigvee$-type and is not of the form\\
    $\exists x \in \Ln\ C(x)^{\Ln}$. Then, again each formula in $\C(A^{\Ln})$ is $\Sigma^{\Ln}$. So, there is $B^{\Ln}\in\C(A^{\Ln})$ such that
    \begin{equation*}
        \H\dash{\rho}{\alpha_0}\Gamma, A^{\Ln}, B^{\Ln},
    \end{equation*}
    with $\alpha_0<\alpha$. By the induction hypothesis, we have
    \begin{equation*}
        \H\dash{\rho}{\alpha_0}\Gamma, A^{\mathbb L_\beta(X)}, B^{\mathbb L_\beta(X)}.
    \end{equation*}
    We apply the rule $(R)$ and we obtain
    \begin{equation*}
        \H\dash{\rho}{\alpha}\Gamma,A^{\mathbb L_\beta(X)}.
    \end{equation*}
    We notice that we can do this even in the case that $A^{\Ln}\equiv \exists x\in\Ln\ C(x)^{\Ln}$. If this is the case, we would have
    \begin{equation*}
        \H\dash{\rho}{\alpha_0}\Gamma, A^{\Ln}, C(t)^{\Ln},
    \end{equation*}
    for some term $t$ with $|t|<\Ln$ and $|t|<\Gamma_{\theta+1}+\alpha$. By the previous reasoning, we get
    \begin{equation*}
        \H\dash{\rho}{\alpha}\Gamma, A^{\mathbb L_\beta(X)}, C(t)^{\mathbb L_\beta(X)}.
    \end{equation*}
    The thing is that $\alpha<\beta$ shows that $\Gamma_{\theta+1}+\alpha<\Gamma_{\theta+1}+\beta$, and so $|t|<\Gamma_{\theta+1}+\beta$, meaning that we can apply $(b\exists)$ to get
    \begin{equation*}
        \H\dash{\rho}{\alpha}\Gamma,\exists x\in\mathbb L_\beta(X)\ C(x)^{\mathbb L_\beta(X)},
    \end{equation*}
    as desired.\\
    \ \\    
    Case 3. We suppose $A^{\Ln}\equiv \exists x \in \Ln\ C^{x}$ for some formula $C$ and $(R)$ is $({\sf Ref}_n)$. This means that we have 
    \begin{equation*}
        \H\dash{\rho}{\alpha_0}\Gamma, A^{\Ln},C^{\Ln},
    \end{equation*}
    with $\alpha_0<\alpha$. We will use the induction hypothesis on $A^{\Ln}$ and on $C^{\Ln}$ separately: on $A^{\Ln}$ to get $A^{\mathbb L_\beta(X)}$ and on $C$ to get $C^{\mathbb L_{\alpha_0}(X)}$. We obtain the following.
    \begin{equation*}
        \H\dash{\rho}{\alpha_0}\Gamma, A^{\mathbb L_{\beta}(X)}, C^{\mathbb L_{\alpha_0}(X)}.
    \end{equation*}
    We obtain by an application of $(b\exists)$ the desired result:
    \begin{equation*}
        \H\dash{\rho}{\alpha}\Gamma,\exists x\in\mathbb L_\beta(X)\ C^x.
    \end{equation*}
\end{proof}
\subsection{The Collapsing Theorem}
In this subsection, we prove the Collapsing Theorem, according to which we can collapse the ordinal bound of derivations controlled by some specific operators which we define now.
\begin{definition}
    We define $\H_\beta$\index[Symbols]{$\H_\beta$} as follows. For any set of ordinals $Y \subseteq T(\theta)$ we let
    \begin{equation*}
        \H_\beta(Y)=\bigcap\Big\{B_n(\alpha):Y\subseteq B_n(\alpha)\text{ with $\beta<\alpha$ and $n<\omega$}\Big\}.
    \end{equation*}
\end{definition}
We notice that if $\beta<\alpha$ then $\beta+1\leq\alpha$ and so $B_n(\beta+1)\subseteq B_n(\alpha)$ for any $n<\omega$. Moreover, $B_0(\beta+1)\subseteq B_n(\beta+1)$ for any $n<\omega$.
Therefore, we have that
\begin{equation*}
    \H_\beta=\bigcap\Big\{B_n(\alpha):\beta<\alpha\wedge n<\omega\Big\}=B_0(\beta+1)
\end{equation*}
It is straightforward to show that $\H_\beta$ is an operator closed under Veblen functions for every ordinal $\beta$.\\
We will use the following result in the proof of the Collapsing Theorem. The proof proceeds by induction on $\alpha$ and can be found in \cite{tfm}.
\begin{Lemma}
$\label{downwards}$
    Let $\H$ be any operator. Let $\alpha,\beta,\gamma$ and $\rho$ be ordinals. Let $\Gamma\cup\{A\}$ be a finite set of formulas. If $\beta>\gamma\in\H$ and $\H\dash{\rho}{\alpha}\Gamma,\forall x\in \mathbb L_\beta(X) A(x)$ then 
    \[\H\dash{\rho}{\alpha}\Gamma, \forall x\in \mathbb L_\gamma(X) A(x).
    \]
\end{Lemma}
$\setcounter{equation}{0}$
We recall that, given some $m\leq\omega$, a $\Sigma^{\Omega_m}$ formula is an $\RS$-formula that has been obtained from a {\sf KPl} $\Sigma$-formula by restricting all the unbounded quantifiers to $\Lm$ and by replacing free variables by terms of level strictly less than $\Omega_m$ (see Definition \ref{sigman}).
\begin{Theorem}[Collapsing Theorem]
$\label{collapsing}$
    Let $n\leq\omega$ and let $m<\omega$. Let $\Gamma$ be a set of $\Sigma^{\Omega_m}$-formulas and let $\alpha$ and $\beta$ be ordinals with $\beta\in \H_\beta$.\\
    If $\H_\beta\dash{\Omega_n+1}\alpha\Gamma$ then $\H_{\beta+\omega^{\Omega_n+1+\alpha}}\dash{\psi_m(\beta+\omega^{\Omega_n+1+\alpha})}{\psi_m(\beta+\omega^{\Omega_n+1+\alpha})}\Gamma$.

\end{Theorem}
\begin{proof}
To simplify notation we define for every ordinal $\alpha$
    \begin{equation*}
        \hat\alpha=\beta+\omega^{\Omega_n+1+\alpha}.
    \end{equation*}
    We are going to prove a more general claim to deal with cases where some terms might be added to the operator:
    \begin{claim}$\label{claim collapse}$
    Let $\Gamma,\alpha,\beta,n$ and $m$ as in the assumption of the theorem. Let $\Delta$ be any finite set of formulas such that $k(\Delta)\subseteq B_m(\beta+1)$. 
    \begin{center}
        If $\H_\beta[\Delta]\dash{\Omega_n+1}{\alpha}\Gamma$ then $\H_{\hat\alpha}[\Delta]\dash{\psi_m\hat{\alpha}}{\psi_m\hat{\alpha}}\Gamma$.
    \end{center}
    \end{claim}
    With Claim $\ref{claim collapse}$, taking $\Delta=\emptyset$ we obtain the theorem.\\ 
    First, we observe that, from
    \begin{equation*}
    0,1,\alpha,\Omega_n,\beta\in\H_\beta[\Delta]=\bigcap\Big\{B_k(\gamma):k(\Delta)\subseteq B_k(\gamma)\wedge k<\omega\Big\}
    \end{equation*}
    we get that
    \begin{equation*}
        \hat\alpha=\beta+\phi_0(\Omega_n+1+\alpha)\in \H_\beta[\Delta]
    \end{equation*}
    since each $B_k(\alpha)$ is closed under addition and Veblen functions. Therefore, also $\hat\alpha\in \H_{\hat\alpha}[\Delta]$, and so $\psi_m(\hat\alpha)\in\H_{\hat\alpha}[\Delta]$.\\
    Now, we prove Claim $\ref{claim collapse}$ by induction on $n$ with a subsidiary induction on $\alpha$.
    \ \\    
    If $\Gamma$ is an axiom, then the claim is trivial.\\
    \ \\
    We suppose that $\Gamma$ has been obtained by applying a rule. We run through the cases based on the last inference rule.\\
    Cases 1. and 2. correspond to the cases where the last rule applied has a principal formula with $\bigvee$-type (and the rule is different from ${\sf Ref}_k$) and $\bigwedge$-type, respectively. We refer to \cite{tfm} for the proof of these two cases.\\
    \ \\
    Case 3. We suppose that the last rule applied is $({\sf Ref}_{k})$ for some $k\leq m$. Then, we have $\H_\beta[\Delta]\dash{\Omega_n+1}{\alpha}\Gamma',\exists z\in\mathbb{L}_{\Omega_k}(X)F^z$, where $F$ is a $\Sigma$-formula and $\H_\beta[\Delta]\dash{\Omega_n+1}{\alpha_0}\Gamma', F^{\mathbb{L}_{\Omega_k}(X)}$. By the induction hypothesis on $\alpha$, $\H_{\hat{\alpha}_0}[\Delta]\dash{\psi_m\hat\alpha_0}{\psi_m\hat\alpha_0}\Gamma', F^{\mathbb{L}_{\Omega_k}(X)}$. We cannot use again the rule since maybe $\psi_m\hat\alpha\geq\Omega_k$. Instead, we use Lemma $\ref{boundedness}$ (Boundedness) to obtain $\H_{\hat\alpha_0}[\Delta]\dash{\psi_m\hat\alpha_0}{\psi_m\hat\alpha_0}\Gamma',F^{\mathbb{L}_{\psi_k\hat\alpha_0}(X)}$. Moreover, since $\H_{\hat\alpha_0}[\Delta]\subseteq \H_{\hat\alpha}[\Delta]$, we apply Lemma $\ref{weak}$ to change the controlling operator and increase the bound of the complexity of the cuts and get $\H_{\hat\alpha}[\Delta]\dash{\psi_m\hat\alpha}{\psi_m\hat\alpha_0}\Gamma',F^{\mathbb{L}_{\psi_k\hat\alpha_0}(X)}$. Now, an application of $(b\exists)$ yields $\H_{\hat\alpha}[\Delta]\dash{\psi_m\hat\alpha}{\psi_m\hat\alpha}\Gamma',\exists z\in\mathbb{L}_{\Omega_k}(X)F^z$.\\
    \ \\
    Case 4. We suppose that the last rule applied is $(Cut)$. Then, we have $\H_\beta[\Delta]\dash{\Omega_n+1}{\alpha}\Gamma$. We also have the premises $\H_\beta[\Delta]\dash{\Omega_n+1}{\alpha_0}\Gamma,A$ and $\H_\beta[\Delta]\dash{\Omega_n+1}{\alpha_0}\Gamma,\neg A$ with $\alpha_0<\alpha$ and $\mathrm{rk}(A)<\Omega_n+1$. We will run through cases based upon the ordering relation between $\mathrm{rk}(A)$ and $\Omega_m$.\\
    \ \\
    Subcase 4.1. We assume $\mathrm{rk}(A)<\Omega_m$. First of all, we observe that 
    \begin{equation*}
        \H_\beta[\Delta]=\bigcap\Big\{B_l(\delta):k(\Delta)\subseteq B_l(\delta)\wedge\beta<\delta\wedge l<\omega\Big\}\subseteq B_m(\beta+1)
    \end{equation*}
    since $k(\Delta)\subseteq B_m(\beta+1)$ by assumption. Thus, since $\mathrm{rk}(A)\in k(A)\subseteq \H_\beta[\Delta]$, we have
    \begin{equation*}
        \mathrm{rk}(A)\in \H_\beta[\Delta]\cap\Omega_m\subseteq B_m(\beta+1)\cap\Omega_m=\psi_m(\beta+1)\leq\psi_m\hat\alpha.
    \end{equation*}
    By the induction hypothesis on $\alpha$, we have $\H_{\hat\alpha_0}[\Delta]\dash{\psi_m\hat\alpha_0}{\psi_m\hat\alpha_0}\Gamma, A$ and $\H_{\hat\alpha_0}[\Delta]\dash{\psi_m\hat\alpha_0}{\psi_m\hat\alpha_0}\Gamma, \neg A$. Now, taking $\H_{\hat\alpha}[\Delta]$ as the control operator by means of Lemma $\ref{weak}$, an application of $(Cut)$ yields $\H_{\hat\alpha}[\Delta]\dash{\psi_m\hat\alpha}{\psi_m\hat\alpha}\Gamma$ as desired. The cut complexity has not increased since $\mathrm{rk}(A)<\psi_m\hat\alpha$.\\
    \ \\
    Subcase 4.2. We assume $\Omega_m\leq \mathrm{rk}(A)<\Omega_n+1$. We notice that we are not able to proceed as in Subcase 4.1 because the complexity of the cuts in the last derivation go beyond $\psi_m\hat\alpha$. We prove the following claim.\\ 
    \begin{claim}$\label{cl coll 2}$  
    Let $\beta\leq\eta<\hat\alpha$ such that $\eta\in \H_\eta$, let $k=\min\{l<\omega: \mathrm{rk}(A)<\Omega_l\}$. If $\H_\eta[\Delta]\dash{\delta}{\delta}\Gamma, A$ and $\H_\eta[\Delta]\dash{\delta}{\delta}\Gamma, \neg A$ for some $\delta<\Omega_{k}$ then
    \begin{equation*}
        \H_{\hat\alpha}[\Delta]\dash{\psi_m\hat\alpha}{\psi_m\hat\alpha}\Gamma.
    \end{equation*}
    \end{claim} 
    \renewcommand\qedsymbol{$\blacksquare$}
    \begin{proof}
    We let $\mu=\mathrm{max}(\mathrm{rk}(A),\delta)+1$. We notice that $\mu\leq\omega^\mu<\Omega_{k}$. 
    Let $\rho=\Omega_{k-1}+1$ (we know $k>0$ since $\Omega_m\leq rk(A)$). Then, we have $\Omega_{k-1}<\rho<\rho+\omega^\mu<\Omega_k\leq \Omega_n$.\\
    From the hypothesis of Claim $\ref{cl coll 2}$, we have
    \begin{equation}\label{coll1}
        \H_\eta[\Delta]\dash{\delta}{\delta}\Gamma, A
    \end{equation}
    and 
    \begin{equation}\label{coll2}
        \H_\eta[\Delta]\dash{\delta}{\delta}\Gamma,\neg A.
    \end{equation}
Note in particular that this implies $\delta \in \H_\eta[\Delta]$.
    By an application of $(Cut)$ on \eqref{coll1} and \eqref{coll2}, we obtain
    \begin{equation*}
        \H_\eta[\Delta]\dash{\rho+\omega^\mu}{\delta+1}\Gamma.
    \end{equation*}
   We observe that the complexity of the cuts is in fact bounded by $\mu$, and so it is also bounded by $\rho+\omega^\mu$.\\
   Obviously, there is no $\Omega_l$ in the interval $[\rho,\rho+\omega^\mu)$. Moreover, $\omega^\mu\in \H_\eta[\Delta]$ since $\delta,\mathrm{rk}(A)\in\H_\eta[\Delta]$. We can use Theorem $\ref{cut elim}$ (Predicative Cut Elimination) and we get
    \begin{equation*}
        \H_\eta[\Delta]\dash{\rho}{\phi_\mu(\delta+1)}\Gamma.
    \end{equation*}
    Now, since $\beta\leq \eta$, we have that $k(\Delta)\subseteq B_m(\beta+1)\subseteq B_m(\eta+1)$. Also, $\eta\in\H_\eta[\Delta]$ by assumption. Thus, the conditions of the Collapsing Theorem (in fact, the conditions of the general Claim $\ref{claim collapse}$ we are proving) are met and so, since $\rho<\Omega_n$, we use the main induction hypothesis (on $n$) to obtain
    \begin{equation}\label{coll3}
        \H_{\eta+\omega^{\rho+\phi_\mu(\delta+1)}}[\Delta]\dash{\psi_m(\eta+\omega^{\rho+\phi_\mu(\delta+1)})}{\psi_m(\eta+\omega^{\rho+\phi_\mu(\delta+1)})}\Gamma.
    \end{equation}
    It remains to show that $\psi_m(\eta+\omega^{\rho+\phi_\mu(\delta+1)})\leq \psi_m(\beta+\omega^{\Omega_n+1+\alpha})$ in order to use Lemma $\ref{weak}$ and obtain the conclusion of Claim $\ref{cl coll 2}$.\\
    From $\rho+\phi_\mu(\delta+1)<\Omega_{k+1}\leq\Omega_n$ we get that
    \begin{equation*}
        \omega^{\rho+\phi_\mu(\delta+1)}<\omega^{\Omega_n}.
    \end{equation*}
    Now, we observe that, since $\beta\leq\eta<\beta+\omega^{\Omega_n+1+\alpha}$ we can write
    \begin{equation*}
        \eta=\beta+\zeta\text{ for some $\zeta<\omega^{\Omega_n+1+\alpha}$}.
    \end{equation*}
    It follows that 
    \begin{equation*}
    \begin{split}
        \eta+\omega^{\Omega_k+\phi_\mu(\delta+1)} & = \beta+\zeta+\omega^{\Omega_k+\phi_\mu(\delta+1)}\\
        & < \beta+\zeta+\omega^{\Omega_n}\\
        & \leq \beta+ \omega^{\Omega_n+1+\alpha}\text{ since $\zeta,\omega^{\Omega_n}<\omega^{\Omega_n+1+\alpha}$ and $\omega^{\Omega_n+1+\alpha}$ is additive principal}\\
        & =\hat{\alpha}.
    \end{split}
    \end{equation*}
    From this inequality we derive that 
    \begin{equation*}
        \H_{\eta+\omega^{\rho+\phi_\mu(\delta+1)}}[\Delta]\subseteq \H_{\hat{\alpha}}[\Delta]\text{ and }\psi_m(\eta+\omega^{\rho+\phi_\mu(\delta+1)})\leq\psi_m(\hat\alpha).
    \end{equation*} 
    Lemma $\ref{weak}$ applied to \eqref{coll3} yields
    \begin{equation*}
        \H_{\hat\alpha}[\Delta]\dash{\psi_m(\hat\alpha)}{\psi_m(\hat\alpha)}\Gamma,
    \end{equation*}
    as desired.
    \end{proof}
    \renewcommand\qedsymbol{$\square$}
    Hence, Claim $\ref{cl coll 2}$ is verified and we can continue analyzing Subcase 4.2 as follows.\\
    \ \\
    Subsubcase 4.2.1. We assume $\mathrm{rk}(A)\neq\Omega_j$ for any $j<\omega$. Let $\Omega_k=\min\{\Omega_i:\mathrm{rk}(A)<\Omega_i\}$. We notice that $\Gamma\cup\{A,\neg A\}$ is in particular a set of $\Sigma^{\Omega_k}$ formulas because $\Omega_m<\mathrm{rk}(A)<\Omega_k$. By the induction hypothesis on $\alpha$, we have $\H_{\hat\alpha_0}[\Delta]\dash{\psi_k\hat\alpha_0}{\psi_k\hat\alpha_0}\Gamma, A$ and $\H_{\hat\alpha_0}[\Delta]\dash{\psi_k\hat\alpha_0}{\psi_k\hat\alpha_0}\Gamma, \neg A$. We use Claim $\ref{cl coll 2}$ with $\eta=\hat\alpha_0$, $\delta=\psi_k\hat\alpha_0$ and $k=k$ to get the result.\\
    \ \\
    Subsubcase 4.2.2. We assume $\Omega_m\leq \mathrm{rk}(A)=\Omega_k\leq\Omega_n$. In this case, $A$ and $\neg A$ are of the form $\exists x\in\mathbb{L}_{\Omega_k}(X) B(x)$ and $\forall x\in\mathbb L_{\Omega_k}(X) \neg B(x)$ (respectively or alternatively), and we have
    \begin{equation}\label{coll4}
        \H_\beta[\Delta]\dash{\Omega_n+1}{\alpha_0}\Gamma,\exists x\in\mathbb{L}_{\Omega_k}(X) B(x)
    \end{equation}
    and
    \begin{equation}\label{coll5}
        \H_\beta[\Delta]\dash{\Omega_n+1}{\alpha_0}\Gamma,\forall x\in\mathbb{L}_{\Omega_k}(X) \neg B(x).
    \end{equation}
    By the induction hypothesis on $\alpha$ applied to \eqref{coll4}, we have $\H_{\hat\alpha_0}[\Delta]\dash{\psi_k\hat\alpha_0}{\psi_k\hat\alpha_0}\Gamma,\exists x\in \mathbb L_{\Omega_k}(X) B(x)$. Let $\xi=\psi_k(\Omega_n+\omega^{\beta+\alpha_0})$. We observe that $\xi\in\H_{\hat\alpha_0}[\Delta]\cap\Omega_k$. By Lemma $\ref{boundedness}$ (Boundedness), we get 
   \begin{equation}\label{coll6}
   \H_{\hat\alpha_0}[\Delta]\dash{\psi_k\hat\alpha_0}{\psi_k\hat\alpha_0}\Gamma,\exists x\in \mathbb L_\xi(X) B(x).
   \end{equation}
   On the other hand, we apply Lemma $\ref{downwards}$ to \eqref{coll5} to obtain 
   \begin{equation*}
       \H_\beta[\Delta]\dash{\Omega_n+1}{\alpha_0}\Gamma,\forall x\in \mathbb L_\xi(X)\neg B(x).
   \end{equation*}
   Since $\beta<\hat\alpha_0$, by Lemma $\ref{weak}$ we get
   \begin{equation*}
       \H_{\hat\alpha_0}[\Delta]\dash{\Omega_n+1}{\alpha_0}\Gamma,\forall x\in \mathbb L_\xi(X)\neg B(x).
   \end{equation*}
   By the induction hypothesis on $\alpha$,
   \begin{equation}\label{coll7}
       \H_{\hat\alpha_0+\omega^{\Omega_n+1+\alpha_0}}[\Delta]\dash{\psi_k(\hat\alpha_0+\omega^{\Omega_n+1+\alpha_0})}{\psi_k(\hat\alpha_0+\omega^{\Omega_n+1+\alpha_0})}\Gamma,\forall x\in \mathbb L_\xi(X)\neg B(x).
   \end{equation}
   Now, we apply Claim $\ref{cl coll 2}$ to \eqref{coll6} and \eqref{coll7} with $\delta=\psi_k(\hat\alpha_0+\omega^{\Omega_n+1+\alpha_0})$, $\eta=\hat\alpha_0+\omega^{\Omega_n+1+\alpha_0}$ and $k=k$ and obtain the result.
\end{proof}
$\setcounter{equation}{0}$
We combine all the cut-elimination results for $\Sigma^{\Omega_0}$-formulas to give a general overview of how to use them together.
\begin{corollary}
    Let $\Gamma$ be any set of $\Sigma^{\Omega_0}$-formulas and let $n\leq\omega$. Let $\alpha$ be any ordinal and let $\beta_k<\dots<\beta_1<\Omega_n$. Then, there is $\delta<\Omega_0$ such that
    \begin{equation*}
        \text{if }\H_0\dash{\Omega_n+1+\omega^{\beta_1}+\dots+\omega^{\beta_k}}{\alpha}\Gamma\text{ then } \H_{\omega^{\Omega_n+1+\gamma}}\dash{0}{\delta}\Gamma,
    \end{equation*}
    where $\gamma=\phi(\beta_1)\Big(\phi(\beta_2)\big(\dots\phi(\beta_k)(\alpha)\dots\big)\Big)$\footnote{Using the notations of hyperations given in \cite{hyperations}, we can succinctly write this expression as $\gamma=e^\beta(\alpha)$ where $\beta=\omega^{\beta_1}+\cdots+\omega^{\beta_k}$.}. 
\end{corollary}
\begin{proof}
    If $k\geq1$, by repeatedly applying Predicative Cut Elimination (Theorem {$\ref{cut elim}$}), we 
    arrive at
    \begin{equation*}
        \H_0\dash{\Omega_n+1}{\gamma}\Gamma.
    \end{equation*}
    By the Collapsing Theorem \ref{collapsing},
    \begin{equation*}
    \H_{\omega^{\Omega_n+1+\gamma}}\dash{\psi_0(\omega^{\Omega_n+1+\gamma})}{\psi_0(\omega^{\Omega_n+1+\gamma})}\Gamma.
    \end{equation*}
    We use Theorem $\ref{cut elim}$ again to obtain
    \begin{equation*}
        \H_{\omega^{\Omega_n+1+\gamma}}\dash{0}{\phi\big(\psi_0(\omega^{\Omega_n+1+\gamma})\big)\big(\psi_0(\omega^{\Omega_n+1+\gamma})\big)}\Gamma.
    \end{equation*}
\end{proof}
\section{Embedding {\sf KPl} into \texorpdfstring{$\RS$}{RSl(X)}}\label{SectEmbedding}
The objective of this section is to prove Theorem $\ref{Embed}$, that shows that if ${\sf KPl}\vdash \Gamma(a_1,\dots,a_n)$ then there is $m<\omega$ such that for arbitrary $X$ we have $\H[s_1,\dots,s_n]\dash{\Omega_\omega+m}{\Omega_\omega\cdot\omega^m}\Gamma(s_1,\dots,s_n)^{\mathbb L_{\Omega_\omega}(X)}$ for any operator $\H$ and any terms $s_1,\dots,s_n$ of level below $\Omega_\omega$. We will see that this $m$ can be retrieved constructively from the {\sf KPl}-proof of $\Gamma(a_1,\dots,a_n)$.
\subsection{The \texorpdfstring{$\Vdash$}{forces} relation}

We start by introducing the relation $\Vdash \Gamma$ as in \cite{cookrathj}, which will mean that a set of formulas $\Gamma$ is derivable with the control of any operator with a reasonable depth depending on the rank of those formulas. We will use the following operation.
\begin{definition}
    Let $\alpha_1,\dots,\alpha_n$ be ordinals. Let $\pi:\{1,\dots,n\}\twoheadrightarrow\{1,\dots,n\}$ be an onto function such that $\alpha_{\pi(1)}\geq\dots\geq\alpha_{\pi(n)}$. We define
    \begin{equation*}
        \alpha_1\#\dots\#\alpha_n=\alpha_{\pi(1)}+\dots+\alpha_{\pi(n)}.
    \end{equation*}
    Given a set of formulas $\Gamma=\{A_1,\dots,A_n\}$, we define
    \begin{equation*}
        \#\Gamma=\omega^{\mathrm{rk}(A_1)}\#\cdots\#\omega^{\mathrm{rk}(A_n)}.
    \end{equation*}
\end{definition}
With the operation $\#$, we can now define $\Vdash \Gamma$.
\begin{definition}\index[Symbols]{$\#$}\index[Symbols]{$\Vdash$}\index[Symbols]{$\Vstile{\rho}{\alpha}$}
    We write $\Vdash \Gamma$ whenever for any operator $\H$ we have $\H[\Gamma]\dash{0}{\#\Gamma}\Gamma$.\\
    We write $\dststile{\rho}{\alpha}\Gamma$ whenever for any operator $\H$ with $\alpha\in\H$ we have $\H[\Gamma]\dash{\rho}{\#\Gamma\#\alpha}\Gamma$.
\end{definition}
Along this section, we follow Reading Convention $\ref{read2}$ and omit the repetition of principal formulas in the premises. The next lemma shows that we can treat the $\Vdash$ relation as a logic in the sense that given the premise(s) $\Gamma_i$, for $i\in y$, of the conclusion $\Gamma$ of some instance of an $\RS$-rule, if $\Vdash \Gamma_i$ for all $i\in y$ then $\Vdash \Gamma$. 
\begin{Lemma} $\label{derive}$
    Let $\Gamma\cup\{A,B\}$ be a finite set of formulas. Let $\alpha$ and $\rho$ be ordinals. If $B\in\C(A)$ then $\#(\Gamma, B)\#\alpha<\#(\Gamma,A)\#\alpha$.\\
    Moreover, if $\Gamma, A$ follows from the premise(s) $\Gamma, B_i$, for $i\in y$ by a rule other than $(Cut)$ and $({\sf Ref}_{n})$, with $n<\omega$, with principal formula $A$ and active formulas $B_i$, then $\Vstile{\rho}{\alpha}\Gamma,A$ whenever $\Vstile{\rho}{\alpha}\Gamma, B_i$ for all $i\in y$.
\end{Lemma}
\begin{proof}
    First, by Lemma $\ref{rk}$ we have $\mathrm{rk}(B)<\mathrm{rk}(A)$ whenever $B\in \C(A)$. It follows that $\omega^{\mathrm{rk}(B)}<\omega^{\mathrm{rk}(A)}$ for any $B\in \C(A)$ and so $\#(\Gamma,B)<\#(\Gamma, A)$ for every $B\in\C(A)$.\\
    We suppose now that $\Vstile{\rho}{\alpha}\Gamma, B_i$ for all $i\in y$ and fix an operator $\H$. Then, 
    \begin{equation*}
        \H[\Gamma, B_i]\dash{\rho}{\#(\Gamma,B_i)\#\alpha}\Gamma, B_i
    \end{equation*}
    for all $i\in y$. By an application of the rule, and since $\#(\Gamma, B_i)\#\alpha<\#(\Gamma,A)\#\alpha$ for all $i\in y$, we get $\H[\Gamma,A]\dash{\rho}{\#(\Gamma,A)\#\alpha}\Gamma,A$. Hence, $\Vstile{\rho}{\alpha}\Gamma, A$.    
\end{proof}
By Lemma $\ref{derive}$, we can write derivations with the $\Vdash$ relation. For example, if for any operator $\H$ we have $\H[A]\dash{0}{\omega^{\mathrm{rk}(A)}}A$ then we also have $\H[A\vee B]\dash{0}{\omega^{\mathrm{rk}(A\vee B)}}A\vee B$. In this case, we will write
\begin{equation*}
    \begin{prooftree}
        \hypo{&\Vdash A}
        \infer[left label=$(\vee)$]1{&\Vdash  A\vee B}
    \end{prooftree}
\end{equation*}
The first part of following lemma gives a nice way to derive formulas of the form $A\rightarrow B$. For instance, a general reasoning that we will use to prove that $\Vdash A\rightarrow B$ will be to first derive $\Vdash \neg A,B$ and, by means of the following lemma, get $\Vdash \neg A\vee B$, which is equivalent to $\Vdash A\rightarrow B$.
\begin{Lemma}$\label{derive2}$
\begin{enumerate}
    \item If $\Vstile{\rho}{\alpha}\Gamma, A, B$ then $\Vstile{\rho}{\alpha}\Gamma, A\vee B$.
    \item For any formula $A$ we have $\Vdash A,\neg A$.
\end{enumerate}
    
\end{Lemma}
\begin{proof}
    1. This follows by applying Lemma \ref{derive} twice.\\
    2. We proceed by induction on the rank of $A$. We note that given any non-basic formula $A$, since $\mathrm{rk}(B)<\mathrm{rk}(A)$ for any $B\in\C(A)$, when proving $\Vdash A,\neg A$ we can suppose $\Vdash B, \neg B$ by the induction hypothesis. We consider cases based on the form of $A$.\\
    \ \\
    Case 1. We suppose that $A\equiv \overline{u}\in\overline{v}$. This means that $A,\neg A$ is an axiom (since either $u\in v$ or $u\notin v$ holds).\\
    \ \\
    Case 2. We suppose that $A\equiv r\in t$ is not a basic formula. This means that either $r$ or $t$ is not a basic term, and so $|r|>\Gamma_{\theta+1}$ or $|t|>\Gamma_{\theta+1}$. By the induction hypothesis, we have $\Vdash s\dot\in t\wedge r=s,\neg(s\dot\in t\wedge r=s)$, which is exactly $\Vdash s\dot\in t\wedge r=s,s\dot\in t\rightarrow r\neq s$, for all terms $s$ with $|s|<|t|$. Therefore, we have the following derivation for every term $s$ with $|s|<|t|$, where the first inference is applied to the first formula and the second inference is applied to the second formula:\\
    \begin{center}

    \begin{prooftree}
        \hypo{&\Vdash s\dot\in t\wedge r=s,s\dot\in t\rightarrow r\neq s}
        \infer[left label = $(\in)$]1{&\Vdash r\in t, s\dot\in t\rightarrow r\neq s}
        \infer[left label = $(\notin)$]1{&\Vdash r\in t, r\notin t}
    \end{prooftree}\\
    \end{center}    
    Hence, we obtain $\Vdash A,\neg A$.\\
    \ \\
    Case 3. We suppose that $A\equiv \exists x\in t\ B(x)$. By the inductive hypothesis, we have $\Vdash s\dot\in t\wedge B(s),\neg\big(s\dot\in t\wedge B(s)\big)$, which is exactly $\Vdash s\dot\in t\wedge B(s),s\dot\in t\rightarrow\neg B(s)$, for all terms $s$ with $|s|<|t|$. Therefore, we have the following derivation  for every term $s$ with $|s|<|t|$, where the first inference is applied to the first formula and the second inference is applied to the second formula:
    \begin{center}
        \begin{prooftree}
            \hypo{&\Vdash s\dot\in t\wedge B(s),s\dot\in t\rightarrow\neg B(s)}
            \infer[left label = $(b\exists)$]1{&\Vdash \exists x\in t\ B(x),s\dot\in t\rightarrow \neg B(s)}
            \infer[left label= $(b\forall)$]1{&\Vdash \exists x\in t\ B(x), \forall x\in t\ \neg B(x)}
        \end{prooftree}
    \end{center}
    Hence, we obtain $\Vdash A,\neg A$.\\
    The other cases follow analogously.
\end{proof}
At some points we will need to write derived formulas in some equivalent expression, e.g. write $A\rightarrow B$ instead of $\neg A\vee B$. We will simply use the symbol $\equiv$ as the label of the derivation when this happens.\\
The next lemma states some results that will be helpful to embed the {\sf KPl} axioms and rules into the $\RS$-system. 
\begin{Lemma}
$\label{Vdash}$
    Let $s$ be any term. Then, we have
    \begin{enumerate}
        \item $\Vdash s\notin s$,
        \item Given any term $t$, if $|s|<|t|$ then $\Vdash s\dot\in t\rightarrow s\in t$,
        \item $\Vdash s\subseteq s$,
        \item $\Vdash s=s$,
        \item Let $\alpha$ be an ordinal. If $|s|<\Gamma_{\theta+1}+\alpha$ then $\Vdash s\in\mathbb L_\alpha(X)$.
    \end{enumerate}
\end{Lemma}
\begin{proof}
    1. The proof proceeds by induction on $\mathrm{rk}(s)$, considering cases based on the form of $s$. We refer to \cite{tfm}.\\
    \ \\
    We prove 2. and 3. simultaneously. Actually, we show $\Vdash \forall x\in s(x\in s)$ by induction on $\mathrm{rk}(s)$ and considering cases based on the form of $s$, and 2. will be shown along the way.\\
    \ \\
    Case 1. We suppose $s\equiv \overline u$. Then, given the basic term $\overline v$, either $\overline v\in\overline u$ or $\overline v\notin \overline{u}$ is an axiom, and so $\overline v\in\overline u,\overline v\notin\overline u$ is an axiom. We have the following derivation for any basic term $\overline v$:
    \begin{equation*}
        \begin{prooftree}
            \hypo{&\Vdash \overline v\in\overline u,\overline v\notin\overline u}
            \infer[left label = $\text{Lemma }\ref{derive2}.1$]1{&\Vdash\neg\overline v\in\overline{u}\vee\overline{v}\in\overline{u}}
            \infer[left label = $\equiv$]1{&\Vdash \overline{v}\in\overline{u}\rightarrow\overline v\in\overline{u}}
            \infer[left label = $\equiv$]1{&\Vdash \overline{v}\dot\in\overline{u}\rightarrow\overline v\in\overline{u}}
            \infer[left label = $(b\forall)$]1{&\Vdash \forall x\in s(x\in s)}
        \end{prooftree}
    \end{equation*}
        We notice that the second to last line is exactly $\Vdash \overline v\dot\in\overline{u}\rightarrow \overline{v}\in\overline{u}$, as in Item 2.\\
    \ \\
    Case 2. We suppose $s\equiv \mathbb L_\alpha(X)$. By the induction hypothesis $\Vdash \forall x\in r(x\in r)$ for any $|r|<|s|$. Therefore, we have the following derivation for all $|r|<|s|$:
    \begin{equation*}
        \begin{prooftree}
            \hypo{\Vdash \forall x\in r(x\in r)}
            \hypo{\Vdash\forall x\in r(x\in r)}
            \infer[left label = $(\wedge)$]2{&\Vdash\forall x\in r(x\in r)\wedge\forall x\in r(x\in r)}
            \infer[left label = $\equiv$]1{&\Vdash r=r}
            \infer[left label = $\equiv$]1{&\Vdash r\dot\in s\wedge r=r}
            \infer[left label= $(\in)$]1{&\Vdash r\in s}
            \infer[left label = $\equiv$]1{&\Vdash r\dot\in s\rightarrow r\in s}
            \infer[left label = $(b\forall)$]1{&\Vdash \forall x\in s(x\in s)}
        \end{prooftree}
    \end{equation*}
    Case 3. We suppose $s\equiv [x\in\mathbb L_\alpha(X):B(x)]$. By the induction hypothesis, we get $\Vdash \forall x\in r(x\in r)$ for all $|r|<|s|$. So, by Weakening (Lemma $\ref{weak}$), we obtain $\Vdash \forall x\in r(x\in r),\neg B(r)$. Applying $(\wedge)$, we have $\Vdash r=r,\neg B(r)$. Moreover, we have $\Vdash B(r),\neg B(r)$ by Item 1. We get the following derivation for all $|r|<|s|$:
    \begin{equation*}
        \begin{prooftree}
            \hypo{\Vdash r=r,\neg B(r)}
            \hypo{\Vdash B(r),\neg B(r)}
            \infer[left label = $(\wedge)$]2{&\Vdash B(r)\wedge r=r,\neg B(r)}
            \infer[left label = $\equiv$]1{&\Vdash r\dot\in s\wedge r=r, \neg B(r)}
            \infer[left label = $(\in)$]1{&\Vdash r\in s, \neg B(r)}
            \infer[left label = Lemma $\ref{derive2}.1$]1{&\Vdash \neg B(r)\vee r\in s}
            \infer[left label = $\equiv$]1{&\Vdash B(r)\rightarrow r\in s}
            \infer[left label = $\equiv$]1{&\Vdash r\dot\in s\rightarrow r\in s}
            \infer[left label = $(b\forall)$]1{&\Vdash \forall x\in s(x\in s)}
        \end{prooftree}
    \end{equation*}
    4. The result follows from Item 3. by an application of the $(\wedge)$ rule.\\
    \ \\
    5. First, for all $|s|<\Gamma_{\theta+1}+\alpha$ we have $\Vdash s=s$ by Item 3. Using Definition $\ref{dotin}$, this is equivalent to $\Vdash s\dot\in\mathbb L_\alpha(X)\wedge s=s$. We apply $(\in)$ to obtain the result.
\end{proof}

\subsection{Embedding the axioms of {\sf KPl}}
In this subsection, we will embed the axioms of {\sf KPl} into $\RS$ and we give ordinal bounds to the length and cut-complexity of the derivations of the axioms in $\RS$. It seems that, in order to embed {\sf KPl}-derivations into $\RS$, it is sufficient to prove that we can find ordinal bounds $\alpha$ and $\beta$ such that, for any {\sf KPl} axiom $\Ax$ and any operator $\H$, we have
\begin{equation*}
    \H\dash{\beta}{\alpha}(\Ax)^{\mathbb L_{\Omega_\omega}(X)}.
\end{equation*}
Nonetheless, we need a stronger result: the third Admissibility axiom $(Ad3)$ states that any admissible set has to satisfy basic axioms. This means that we have to prove that, given an axiom $\Ax$ among Leibniz Principle, Pair, Union, $\Delta_0$-Separation and $\Delta_0$-Collection,
\begin{equation*}
    \H\dash{\beta}{\alpha}(\Ax)^{\mathbb L_{\Omega_{n}}(X)}
\end{equation*}
holds for all $n\leq\omega$ and for any operator $\H$. Actually, for $\Delta_0$-Collection we only need $n<\omega$ since this axiom is not an axiom of {\sf KPl}.\\
First, we have the following result via a single excluded middle on $s\in t$.
\begin{Lemma}$\label{stAB}$
    Let $s, t$ be terms such that $|s|<|t|$. Let $\Gamma$ be any finite set of formulas and let $A$ and $B$ be formulas. If $\Vdash\Gamma,A,B$ then
    \begin{center}
        $\Vdash \Gamma, s\dot\in t\rightarrow A,s\dot\in t\wedge B$.
    \end{center}
    
\end{Lemma}
The embedding of the set-theoretic axioms of {\sf KPl} follow standard methods.
\begin{Lemma}[Set-theoretic axioms]\ $\label{axiomsembed}$Let $n<\omega$.
\begin{enumerate}
    \item Leibiniz Principle. Let $s$ and $t$ be terms. For any formula {\sf KPl}-formula $A(x)$, we have
    \begin{equation*}
        \Vdash s\neq t,\neg A(s)^{\mathbb L_{\Omega_n}(X)},A(t)^{\mathbb L_{\Omega_n}(X)}.
    \end{equation*}
    \item Class Induction. For any {\sf KPl}-formula $A$, we have
    \begin{equation*}
        \Vstile{}{\omega^{\mathrm{rk}(B)}}B\rightarrow \forall x\in\mathbb L_{\Omega_n}(X)\ A(x)^{\Ln},
    \end{equation*}
    where $B\equiv \forall x\in\mathbb L_{\Omega_n}(X)\big(\forall y\in x\ A(y)^{\mathbb L_{\Omega_n}(X)}\rightarrow A(x)^{\mathbb L_{\Omega_n}(X)}\big)$.
    \item $\Delta_0$-Separation. Let $A(a,b_1,\dots,b_k)$ be a {\sf KPl} $\Delta_0$-formula. Let $s,t_1,\dots,t_k$ be terms such that $|s|,|t_1|,\dots,|t_k|<\Omega_n$. We will use the abbreviation $\vec t=t_0,\dots,t_k$. Then,
    \begin{equation*}
        \Vdash \exists y\in\mathbb L_{\Omega_n}(X)\Big[\forall x\in y\big(x\in s\wedge A(x,\vec t)^{\mathbb L_{\Omega_n}(X)}\big)\wedge\forall x\in s\big(A(x,\vec t)^{\mathbb L_{\Omega_n}(X)}\rightarrow x\in y\big)\Big].
    \end{equation*}
    \item Infinity. $\Vdash\exists x\in\mathbb{L}_{\Omega_n}(X) [\exists z\in x(z\in x)\wedge \forall y\in x\exists z\in x(y\in z)]$.\\
    \ \\
    \item Pairing. Let $s$ and $t$ be terms such that $|s|,|t|<\Omega_n$. Then
    \begin{equation*}
        \Vdash \exists z\in\mathbb L_{\Omega_n}(X)(s\in z\wedge t\in z).
    \end{equation*}
    \item Union. Let $s$ be a term such that $|s|<\Omega_n$. Then
    \begin{equation*}
        \Vdash \exists z\in\mathbb L_{\Omega_n}(X)\forall y\in s\forall x\in y(x\in z).
    \end{equation*}
    \end{enumerate}
\end{Lemma}
\begin{proof}
    The proof follows the methods of \cite{cookrathj}. More details can be found in \cite{tfm}.
\end{proof}
$\setcounter{equation}{0}$
Now, we focus on the axioms ruling the $Ad$ predicate. To embed the first axiom, we need to know what $(Ad1)^{\Lo}$ will look like. In particular, we need to identify in $\RS$ the term to which the set $\omega$ translates and the predicate to which the set-theoretic predicate $\mathrm{Tran}$ translates. We start with the latter. The set-theotetic predicate $\mathrm{Tran(x)}$ for a free variable $x$ would translate to the following $\RS$-predicate for a term $t$.
\begin{definition}
    Let $t$ be any term. We define the $\RS$-predicate
    \begin{equation*}
        \mathrm{Tran}(t)\equiv \forall x\in t\forall y\in x(y\in t).
    \end{equation*}
\end{definition}
We note that the {\sf KPl} $\Delta_0$-formula $\mathrm{Tran}(x)$ with free variable $x$ is translated to $\mathrm{Tran}(t)$ in the $\RS$-language, for some chosen term $t$. This is why we write both predicates the same way.\\ 
Next, we can define $\omega$ as the unique ordinal only containing infinitely many successor ordinals, and so the term
\begin{equation*}
    \underline\omega=\Big[x\in\mathbb L_{\omega+1}(X):\mathrm{Ord}(x)\wedge\big( \exists y\in x\rightarrow \exists y\in x\forall z\in x(z\neq y\rightarrow z\in y)\big) \Big].
\end{equation*}
fully captures the set $\omega$. This means that $(Ad1)^{\Lo}$ can be written as
\begin{equation*}
    \forall x\in\Lo[Ad(x)\rightarrow \underline\omega\in x\wedge \mathrm{Tran}(x)].
\end{equation*}

But our admissibles in $\RS$ are all the $\Ln$, $n<\omega$. So we show that they are transitive and contain $\omega$.
\begin{Lemma} Let $n<\omega$. Then,
    $\label{tran}$
    \begin{enumerate}
        \item $
        \Vdash \mathrm{Tran}\big(\mathbb L_{\Omega_{n}}(X)\big)$.
    \item $
        \Vdash\underline \omega\in\mathbb L_{\Omega_{n}}(X)$.
    \end{enumerate}
    
\end{Lemma}
\begin{proof}
    1. Let $t$ be any term such that $|t|<\Omega_n$. We show that for every term $s$ with $|s|<|t|$
    \begin{equation}\label{tran1}
        \Vdash s\dot\in t\rightarrow s\in \Ln
    \end{equation}
    to use the rule $(b\forall)$. We prove \eqref{tran1} by fixing such an $s$ and splitting cases based on the form of $t$. First, we observe that $|s|<\Omega_n$ and so, by Lemma $\ref{Vdash}.5$, we have
    \begin{equation}\label{tran2}
        \Vdash s\in \Ln.
    \end{equation}
    Case 1. We suppose $t\equiv \overline u$. Then $s\equiv \overline v$. Therefore, using Weakening (Lemma $\ref{weak}$) on \eqref{tran2}, we get $\Vdash \neg\overline v\in\overline u, \overline v\in\Ln$. By Lemma $\ref{derive2}.1$, we obtain $\Vdash \neg \overline v\in\overline u\vee \overline v\in\Ln$, which is equivalent to $\Vdash \overline v\in\overline u\rightarrow \overline v\in\Ln$.
    By Definition $\ref{dotin}$, this is
    \begin{equation*}
        \Vdash \overline v\dot\in\overline u\rightarrow \overline v\in\Ln.
    \end{equation*}
    Case 2. We suppose $t\equiv\mathbb L_\alpha(X)$. Then \eqref{tran2} is exactly $\Vdash s\dot\in t\rightarrow s\in\Ln$ by Definition $\ref{dotin}$.\\
    \ \\
    Case 3. We suppose $t\equiv [x\in\mathbb L_\alpha(X):B(x)]$. Then, using Weakening (Lemma $\ref{weak}$) on \eqref{tran2} we get $\Vdash \neg B(s),s\in\Ln$. By Lemma $\ref{derive2}.1$, we obtain $\Vdash \neg B(s)\vee s\in \Ln$, which is exactly $\Vdash B(s)\rightarrow s\in \Ln$. By Definition $\ref{dotin}$, this is equivalent to $\Vdash s\dot\in t\rightarrow s\in\Ln$.\\
    \ \\
    Hence, we have shown \eqref{tran1} for any $|s|<|t|$. By an application of $(b\forall)$, we obtain
    \begin{equation*}
        \Vdash \forall y\in t\big(y\in\Ln\big).
    \end{equation*}
    Again by Definition $\ref{dotin}$, this is the same as
    \begin{equation*}
        \Vdash t\dot\in\Ln\rightarrow \forall y\in t\big(y\in\Ln\big).
    \end{equation*}
    Since this holds for any $|t|<\Omega_n$, another application of $(b\forall)$ yields
    \begin{equation*}
        \Vdash\forall x\in\Ln\forall y\in x\big(y\in \Ln\big).
    \end{equation*}
    $\setcounter{equation}{0}$
    \ \\
    2. We observe that $|\underline\omega|=\Gamma_{\theta+1}+\omega+1$. Let $n<\omega$. Then, since $\Gamma_{\theta+1}+\omega+1<\Omega_{n}$, by Lemma $\ref{Vdash}.5$ we have $\Vdash \underline\omega\in L_{\Omega_{n}}(X)$.
\end{proof}
$\setcounter{equation}{0}$
We can now proceed to the proof of the embedding of $(Ad1)$ into $\RS$.
\begin{Lemma}[Ad1] Let $\H$ be any operator. We have
    $\label{Ad1}$    
     \begin{equation*}
         \H\dash{\Omega_\omega}{\Omega_\omega+1}\forall x\in\mathbb L_{\Omega_\omega}(X)\big(Ad(x)\rightarrow\underline\omega\in x\wedge \mathrm{Tran}(x)\big)
     \end{equation*}
\end{Lemma}

\begin{proof}
    We need to derive
    \begin{equation*}
        Ad(t)\rightarrow\underline\omega\in t\wedge \mathrm{Tran}(t)
    \end{equation*}
    for every term $t$ with $|t|<\Omega_\omega$ to apply $(b\forall)$. So we fix a term $t$ such that $|t|<\Omega_\omega$. In the following derivations, we fix a natural number $n$. By Lemma $\ref{tran}$ and applying $(\wedge)$, we have $\Vdash \underline\omega\in\Ln\wedge \mathrm{Tran}\big(\Ln\big)$. By Weakening (Lemma $\ref{weak}$), we get
    \begin{equation}\label{ad11}
        \Vdash t\neq\Ln,\underline\omega\in\Ln\wedge \mathrm{Tran}\big(\Ln\big),\underline\omega\in t\wedge \mathrm{Tran}(t).
    \end{equation}
    On the other hand, by Lemma $\ref{axiomsembed}.1$, we have
    \begin{equation}\label{ad12}
        \Vdash t\neq\Ln,\neg\Big(\underline\omega\in\Ln\wedge \mathrm{Tran}\big(\Ln\big)\Big),\underline\omega\in t\wedge \mathrm{Tran}(t).
    \end{equation}
    Fix an operator $\H$. An application of $(Cut)$ to \eqref{ad11} and \eqref{ad12} gives
    \begin{equation*}
        \H[t]\dash{\Omega_{n+1}}{\alpha}t\neq\Ln, \underline\omega\in t\wedge \mathrm{Tran}(t),
    \end{equation*}
    where $\alpha=\#\{t\neq\Ln, \underline\omega\in t\wedge \mathrm{Tran}(t)\}<\Omega_\omega$. Since this derivation holds for every $n<\omega$, we apply $(\neg Ad)$ to get $\H[t]\dash{\Omega_\omega}{\Omega_\omega}\neg Ad(t),\underline\omega\in t\wedge \mathrm{Tran}(t)$.
    By Lemma $\ref{derive2}.1$, we obtain $\H[t]\dash{\Omega_\omega}{\Omega_\omega}\neg Ad(t)\vee\underline\omega\in t\wedge \mathrm{Tran}(t)$, which is the same as $\H[t]\dash{\Omega_\omega}{\Omega_\omega} Ad(t)\rightarrow\underline\omega\in t\wedge \mathrm{Tran}(t)$.
    This holds for any $|t|<\Omega_\omega$. An application of $(b\forall)$ yields $\H\dash{\Omega_\omega}{\Omega_\omega+1}\forall x\in\Lo\big(Ad(x)\rightarrow \underline\omega\in x\wedge \mathrm{Tran}(x)\big)$.
\end{proof}
$\setcounter{equation}{0}$
We continue with the second axiom about the $Ad$ predicate.
\begin{Lemma}[Ad2] Let $\H$ be any operator. We have
$\label{Ad2}$
    \begin{equation*}
        \H\dash{\Omega_\omega}{\Omega_\omega+3}\forall x\in\mathbb L_{\Omega_\omega}(X)\forall y\in \mathbb L_{\Omega_\omega}(X)\big(Ad(x)\wedge Ad(y)\rightarrow x\in y\vee x= y\vee y\in x\big)
    \end{equation*}
\end{Lemma}
\begin{proof}
    We need to derive
    \begin{equation*}
        Ad(s)\wedge Ad(t)\rightarrow s\in t\vee s=t\vee t\in s
    \end{equation*}
    for all the terms $s,t$ with $|s|,|t|<\Omega_\omega$ in order to apply $(b\forall)$ twice. In the following derivations, we fix two natural numbers $n$ and $m$.\\
    If $n=m$, then by Lemma $\ref{Vdash}.4$ we have $\Vdash \Ln=\Lm$.\\
    If $n<m$, by Lemma $\ref{Vdash}.5$ we have $\Vdash \Ln\in\Lm$.\\
    If $m<n$, by Lemma $\ref{Vdash}.5$ we have $\Vdash \Lm\in \Ln$.\\
    \ \\
    In any case, two applications of $(\vee)$ give
    \begin{equation*}
        \Vdash \mathrm{Trichotomy}(\Ln,\Lm),
    \end{equation*}
    where $\mathrm{Trichotomy}(a,b):\equiv a\in b\vee a=b\vee b\in a$.\\
    We use Weakening (Lemma $\ref{weak}$) to get
    \begin{equation}\label{ad21}
    \begin{split}
        \Vdash & s\neq\Ln,t\neq \Lm,\mathrm{Trichotomy}\big(\Ln,\Lm\big),\\
        & s\in t\vee s=t\vee t\in s.
        \end{split}
    \end{equation}
    On the other hand, by Lemma $\ref{axiomsembed}.1$ we have
    \begin{equation}\label{ad22}
    \begin{split}
        \Vdash & s\neq\Ln,t\neq \Lm,\neg\mathrm{Trichotomy}\big(\Ln,\Lm\big),\\
        & s\in t\vee s=t\vee t\in s.
    \end{split}
    \end{equation}
    We fix an operator $\H$. An application of $(Cut)$ on \eqref{ad21} and \eqref{ad22} yields
    \begin{equation*}
        \H[s,t]\dash{\Omega_{n+m}}{\alpha}s\neq\Ln,t\neq \Lm,s\in t\vee s=t\vee t\in s,
    \end{equation*}
    where $\alpha=\#\{s\neq\Ln,t\neq \Lm,s\in t\vee s=t\vee t\in s\}$. Since this derivation holds for any $n,m<\omega$, we use $(\neg Ad)$ twice and we obtain
    \begin{equation*}
        \H[s,t]\dash{\Omega_\omega}{\Omega_\omega+1} \neg Ad(s),\neg Ad(t),s\in t\vee s=t\vee t\in s.
    \end{equation*}
    We use Lemma $\ref{derive2}.1$ to get $\H[s,t]\dash{\Omega_\omega}{\Omega_\omega+1} \neg Ad(s)\vee\neg Ad(t),s\in t\vee s=t\vee t\in s$, which is the same as $\H[s,t]\dash{\Omega_\omega}{\Omega_\omega+1} \neg \big(Ad(s)\wedge Ad(t)\big),s\in t\vee s=t\vee t\in s$.
    Again by Lemma $\ref{derive2}.1$, we get $\H[s,t]\dash{\Omega_\omega}{\Omega_\omega+1} \neg \big(Ad(s)\wedge Ad(t)\big)\vee s\in t\vee s=t\vee t\in s$, which is equivalent to $\H[s,t]\dash{\Omega_\omega}{\Omega_\omega+1} Ad(s)\wedge Ad(t)\rightarrow s\in t\vee s=t\vee t\in s$.
    Therefore, since this last derivation holds for any $t,s$ with $|t|,|s|<\Omega_\omega$, we can apply $(b\forall)$ twice to obtain
    \begin{equation*}
        \H\dash{\Omega_\omega}{\Omega_\omega+3}\forall x\in\Lo\forall y\in \Lo\big(Ad(x)\wedge Ad(y)\rightarrow x\in y\vee x=y\vee y\in x\big).
    \end{equation*}
\end{proof}
$\setcounter{equation}{0}$
We continue with the third admissibility axiom. By Lemma $\ref{axiomsembed}$, we can derive in $\RS$ that any $\Ln$ satisfies all the axioms appearing in $(Ad3)$, except $\Delta_0$-Collection. This is given by the following Lemma, proved in \cite{tfm}. The embedding of this axiom uses the $\RS$-rule $(\mathrm{Ref_n})$. This is the only place where this rule is used.
\begin{Lemma}$\label{Coll}$
    Let $A(x,y)$ be any $\Delta_0$-formula of {\sf KPl} with free variables $x$ and $y$. Let $n<\omega$. Let $s$ be a term such that $|s|<\Omega_n$. Then,
    \begin{equation*}
        \Vdash \forall x\in s\exists y\in \mathbb L_{\Omega_n}(X)\ A(x,y)\rightarrow \exists z\in\mathbb{L}_{\Omega_n}(X)\forall x\in s\exists y\in z\ A(x,y).
    \end{equation*}
\end{Lemma}
$\setcounter{equation}{0}$
We have everything we needed to embed $(Ad3)$.
\begin{Lemma}[Ad3]
$\label{Ad3}$
     $\H\dash{\Omega_\omega}{\Omega_\omega+1}\forall x\in \mathbb L_{\Omega_\omega}(X)[Ad(x)\rightarrow (Pair)^x\wedge (Union)^x\wedge (\Delta_0-Separation)^x\wedge(\Delta_0-Collection)^x]$.
\end{Lemma}
\begin{proof}
    We need to derive 
    \begin{equation*}
        Ad(t)\rightarrow (Pair)^t\wedge (Union)^t\wedge (\Delta_0-Separation)^t\wedge(\Delta_0-Collection)^t
    \end{equation*}
    for every $t$ with $|t|<\Omega_\omega$ in order to apply $(b\forall)$ and obtain the desired derivation. So let $t$ be any term with $|t|<\Omega_\omega$. Fix a natural number $n$. 
    By Lemma $\ref{axiomsembed}$, we have
    \begin{equation}\label{ad31}
        \Vdash ({\sf Ax})^{\Ln}
    \end{equation}
    for every axiom ${\sf Ax}$ among Pair, Union, $\Delta_0$-Collection and $\Delta_0$-Separation. We will use the following abbreviation: given any term (or variable) $s$, we will write $\bigwedge ({\sf Axioms})^s$ to denote
    \begin{equation*}
        (Pair)^s\wedge (Union)^s\wedge (\Delta_0-Separation)^s\wedge(\Delta_0-Collection)^s
    \end{equation*}
    By \eqref{ad31} and applying $(\wedge)$ thrice, we get $\Vdash \bigwedge ({\sf Axioms})^{\Ln}$. We use Weakening (Lemma $\ref{weak}$) to obtain
    \begin{equation}\label{ad32}
        \Vdash t\neq\Ln, \bigwedge ({\sf Axioms})^t,\bigwedge ({\sf Axioms})^{\Ln}.
    \end{equation}
    On the other hand, by Lemma $\ref{axiomsembed}.1$ we have
    \begin{equation}\label{ad33}
        \Vdash t\neq\Ln, \bigwedge ({\sf Axioms})^t,\neg \bigwedge ({\sf Axioms})^{\Ln}.
    \end{equation}
    Therefore, for any operator $\H$, an application of $(Cut)$ on $\eqref{ad32}$ and \eqref{ad33} yields
    \begin{equation*}
        \H[t]\dash{\Omega_{n+1}}{\alpha}t\neq\Ln, \bigwedge ({\sf Axioms})^t,
    \end{equation*}
    where $\alpha=\#\{t\neq\Ln, \bigwedge ({\sf Axioms})^t\}$. Since this derivation holds for any $n<\omega$, an application of $(\neg Ad)$ yields $\H[t]\dash{\Omega_\omega}{\Omega_\omega}\neg Ad(t), \bigwedge ({\sf Axioms})^t$. By Lemma $\ref{derive2}.1$, we get for any operator $\H$ that $\H[t]\dash{\Omega_\omega}{\Omega_\omega}\neg Ad(t)\vee \bigwedge ({\sf Axioms})^t$, which is equivalent to $\H[t]\dash{\Omega_\omega}{\Omega_\omega}Ad(t)\rightarrow \bigwedge ({\sf Axioms})^t$. Therefore, we obtain $\H\dash{\Omega_\omega}{\Omega_\omega+1}\forall x\in\Lo[Ad(x)\rightarrow \bigwedge ({\sf Axioms})^x]$ by an application of $(b\forall)$.
\end{proof}
$\setcounter{equation}{0}$
Finally, we embed the limit axiom. We state a preliminary lemma.
\begin{Lemma}
    $\label{Ad}$
    For any natural number $n$ we have
    \begin{equation*}
        \Vdash Ad\big(\mathbb L_{\Omega_n}(X)\big).
    \end{equation*}
\end{Lemma}
\begin{proof}
    Let $n<\omega$. By Lemma $\ref{Vdash}.4$, we have $\Vdash \mathbb L_{\Omega_n}(X)=\mathbb L_{\Omega_n}(X)$. We apply $(Ad)$ to obtain the result.
\end{proof}
We can embed the last axiom of {\sf KPl}.
\begin{Lemma}[Lim]
    $\label{Lim}$
    Let $\H$ be any operator. Then 
    \begin{equation*}
        \H\dash{0}{\Omega_\omega\cdot\omega^2}\forall x\in \mathbb L_{\Omega_\omega}(X)\exists y\in \mathbb L_{\Omega_\omega}(X)\big(Ad(y)\wedge x\in y\big).
    \end{equation*}
\end{Lemma}
\begin{proof}
    Let $s$ be a term such that $|s|<\Omega_\omega$. Then there is $m<\omega$ such that $|s|<\Omega_{m}$. So let $n:=\min(m<\omega: |s|<\Omega_{n})$. It follows that $\Vdash s\in\mathbb L_{\Omega_{n}}(X)$ by Lemma $\ref{Vdash}.5$.\\
    On the other hand, by Lemma $\ref{Ad}$ we have 
    \begin{equation*}
        \Vdash Ad\big(\mathbb L_{\Omega_{n}}(X)\big).
    \end{equation*}
    We have the following derivation for any $|s|<\Omega_\omega$:
    \begin{center}
        \begin{prooftree}
            \hypo{\Vdash Ad\big(\mathbb L_{\Omega_{n}}(X)\big)}
            \hypo{\Vdash s\in\mathbb L_{\Omega_{n}}(X)}
            \infer[left label = $(\wedge)$]2{&\Vdash Ad\big(\mathbb L_{\Omega_{n}}(X)\big)\wedge s\in\mathbb L_{\Omega_{n}}(X)}
            \infer[left label = $(b\exists)$]1{&\Vdash\exists y\in\mathbb L_{\Omega_\omega}(X)\big(Ad(y)\wedge s\in y\big)}
            \infer[left label = $(b\forall)$]1{&\Vdash\forall x\in\mathbb L_{\Omega_\omega}(X)\exists y\in\mathbb L_{\Omega_\omega}(X)\big(Ad(y)\wedge x\in y\big)}
        \end{prooftree}
    \end{center}
    Since $\#\{(Lim)^{\mathbb L_{\Omega_\omega}(X)}\}=\omega^{\mathrm{rk}\big(\mathbb L_{\Omega_\omega}(X)\big)}=\omega^{\Omega_\omega+2}=\omega^{\Omega_\omega}\cdot\omega^2=\Omega_\omega\cdot \omega^2$, we eventually obtain that for any operator $\H$
    \begin{equation*}
        \H\dash{0}{\Omega_\omega\cdot \omega^2}(Lim)^{\mathbb L_{\Omega_\omega}(X)}.
    \end{equation*}
\end{proof}
$\setcounter{equation}{0}$
\subsection{The embedding Theorem}
We have successfully embedded all of the axioms of {\sf KPl} into the $\RS$-system. It is now time to state and show the full embedding theorem.

\begin{Theorem}
    $\label{Embed}$
    Let $\Gamma(a_1,\dots,a_n)$ be a finite set of formulas with all the free variables displayed such that ${\sf KPl}\vdash \Gamma(a_1,\dots,a_n)$. Then, there is some $m<\omega$ such that for any set $X$, for any operator $\H$ and any $\RS$-terms $s_1,\dots,s_n$ we have
    \begin{center}
        $\H[s_1,\dots,s_n]\dash{\Omega_\omega+m}{\Omega_\omega\cdot \omega^m}\Gamma(s_1,\dots,s_n)^{\mathbb L_{\Omega_\omega}(X)}$.
    \end{center}
\end{Theorem}
\begin{proof}
    We fix a set $X$. We fix an operator $\H$ and terms $s_1,\dots,s_n$. We proceed by induction on the {\sf KPl} proof.\\
    If $\Gamma(a_1,\dots,a_n)$ is an axiom of {\sf KPl}, then we obtain the result by Lemmas $\ref{axiomsembed}$, \ref{Ad1}, \ref{Ad2}, \ref{Ad3} and \ref{Lim}.\\
    Now, we assume that $\Gamma(a_1,\dots,a_n)$ is obtained by a {\sf KPl} rule (R). So, if $\Delta(a_1,\dots,a_n)$ is the premise, or one of the premises, of this inference, then {\sf KPl} proves $\Delta(a_1,\dots,a_n)$. Therefore, by the induction hypothesis, there is $m<\omega$ such that
    \begin{equation*}
        \H[s_1,\dots,s_n]\dash{\Omega_\omega+m}{\Omega_\omega\cdot\omega^m}\Delta(s_1,\dots,s_n)^{\mathbb L_{\Omega_\omega}(X)}.
    \end{equation*}
    for any operator $\H$ and any terms $s_1,\dots,s_n$ of level below $\Omega_\omega$. This is the way we will reason for each case (cases correspond to {\sf KPl} rules).\\
    So, we fix an arbitrary operator $\H$ and arbitrary terms $s_1,\dots,s_n$ of level less than $\Omega_\omega$. To simply notation, we will write $\vec a=a_1,\dots,a_n$ and $\vec s=s_1,\dots, s_n$ (even though $\vec a$ is a vector of variables of {\sf KPl} and $\vec s$ is a vector of $\RS$-terms). As a reminder, when we write $A(\vec s)^{\mathbb L_{\Omega_\omega}(X)}$ we are meaning the formula $A$ replacing free variables by the terms in $\vec s$ and bounding all unrestricted quantifiers to $\mathbb L_{\Omega_\omega}(X)$.\\
    \ \\
    Case 1. We suppose that the last {\sf KPl} rule applied is $(\wedge)$. This means that 
    \begin{equation*}
        \Gamma(\vec a)=\Gamma'(\vec a), A(\vec a)\wedge B(\vec a),
    \end{equation*}
    for some {\sf KPl} formulas $A$ and $B$. Therefore, we have 
    \begin{equation}\label{emb1}
        {\sf KPl}\vdash\Gamma'(\vec a), A(\vec a)
    \end{equation}
    and
    \begin{equation}\label{emb2}
        {\sf KPl}\vdash\Gamma'(\vec a), B(\vec a)
    \end{equation}
    We apply the induction hypothesis to \eqref{emb1} to find some $m_0<\omega$ independent of $\vec s$ such that
    \begin{equation}\label{emb3}
        \H[\vec s]\dash{\Omega_\omega+m_0}{\Omega_\omega\cdot\omega^{m_0}}\Gamma'(\vec s)^{\mathbb L_{\Omega_\omega}(X)}, A(\vec s)^{\mathbb L_{\Omega_\omega}(X)}.
    \end{equation}
    We also apply the induction hypothesis to \eqref{emb2} to find some $m_1<\omega$ such that
    \begin{equation}\label{emb4}
        \H[\vec s]\dash{\Omega_\omega+m_1}{\Omega_\omega\cdot\omega^{m_1}}\Gamma'(\vec s)^{\mathbb L_{\Omega_\omega}(X)}, B(\vec s)^{\mathbb L_{\Omega_\omega}(X)}.
    \end{equation}
    We apply the $\RS$ rule $(\wedge)$ to \eqref{emb3} and \eqref{emb4} to obtain
    \begin{equation*}
        \H[\vec s]\dash{\Omega_\omega+\mathrm{max}(m_0,m_1)+1}{\Omega_\omega\cdot\omega^{\mathrm{max}(m_0+m_1)+1}}\Gamma'(\vec s)^{\mathbb L_{\Omega_\omega}(X)}, A(\vec s)^{\mathbb L_{\Omega_\omega}(X)}\wedge B(\vec s)^{\mathbb L_{\Omega_\omega}(X)},
    \end{equation*}
    which is exactly
    \begin{equation*}
        \H[\vec s]\dash{\Omega_\omega+\mathrm{max}(m_0,m_1)+1}{\Omega_\omega\cdot\omega^{\mathrm{max}(m_0+m_1)+1}}\Gamma'(\vec s)^{\mathbb L_{\Omega_\omega}(X)}, (A\wedge B)(\vec s)^{\mathbb L_{\Omega_\omega}(X)}.
    \end{equation*}
    Case 2. We suppose that the last {\sf KPl} rule applied is $(\vee)$. This means that
    \begin{equation*}
        \Gamma(\vec a)=\Gamma'(\vec a), A(\vec a)\vee B(\vec a),
    \end{equation*}
    for some {\sf KPl} formulas $A$ and $B$. By a reasoning similar to the one in Case 1, we get some $m<\omega$ independent of $\vec s$ such that
    \begin{equation*}
        \H[\vec s]\dash{\Omega_\omega+m}{\Omega_\omega\cdot\omega^m}\Gamma'(\vec s)^{\mathbb L_{\Omega_\omega}(X)},(A\vee B)(\vec s)^{\mathbb L_{\Omega_\omega(X)}}.
    \end{equation*}
    Case 3. We suppose that the last {\sf KPl} rule applied is $(b\forall)$. This means that 
    \begin{equation*}
        \Gamma(\vec a)=\Gamma'(\vec a), \forall x\in a_i\ A(x,\vec a)
    \end{equation*}
    for some {\sf KPl}-formula $A$ and some free variable $a_i$ with $i\in\{1,\dots,n\}$. Therefore, we have
    \begin{equation*}
        {\sf KPl}\vdash \Gamma'(\vec a), b\in a_i\rightarrow A(b,\vec a)
    \end{equation*}
    where $b\neq a_j$ for all $j\in\{1,\dots,n\}$. By the induction hypothesis we can find some $m<\omega$ independent of $\vec s$ such that 
    \begin{equation}\label{emb5}
        \H[r,\vec s]\dash{\Omega_\omega+m}{\Omega_\omega\cdot\omega^m}\Gamma'(\vec s)^{\mathbb L_{\Omega_\omega}(X)},r\in s_i\rightarrow A(r,\vec s)^{\mathbb L_{\Omega_\omega}(X)},
    \end{equation}
    for all terms $r$ with $|r|<|s_i|$. We would like to apply $(b\forall)$ to end this case. Nonetheless, we need to have $r\dot\in s\rightarrow A(r,\vec s)^{\mathbb L_{\Omega_\omega}(X)}$ instead of $r\in s\rightarrow A(r,\vec s)^{\mathbb L_{\Omega_\omega}(X)}$ in the premise to be able to do this. We split subcases based on the form of $s_i$. We omit the details of Subcases 3.1 and 3.2 and fully show Subcase 3.3.\\
    Subcase 3.1. We assume that $s_i\equiv\overline u$. Then, we have
    \begin{equation*}
        \H[\vec s]\dash{\Omega_\omega+m+1}{\Omega_\omega\cdot\omega^{m+1}}\Gamma'(\vec s)^{\mathbb L_{\Omega_\omega}(X)},\forall x\in \overline u\ A(x,\vec s)^{\mathbb L_{\Omega_\omega}(X)}.
    \end{equation*}
    Subcase 3.2. We assume that $s_i\equiv \mathbb L_\alpha(X)$. Then, we have
    \begin{equation*}
        \H[\vec s]\dash{\Omega_\omega+m+1}{\Omega_\omega\cdot\omega^{m+1}}\Gamma'(\vec s)^{\mathbb L_{\Omega_\omega}(X)}, \forall x\in \mathbb L_\alpha(X)\ A(x,\vec s)^{\mathbb L_{\Omega_\omega}(X)}.
    \end{equation*}
    Subcase 3.3. We assume that $s_i\equiv [x\in\mathbb L_\alpha(X)]$. We apply Lemma $\ref{vee}$ to \eqref{emb5} to get
    \begin{equation*}
        \H[r,\vec s]\dash{\Omega_\omega+m}{\Omega_\omega\cdot\omega^m}\Gamma'(\vec s)^{\Lo},r\notin s_i,A(r,\vec s)^{\Lo}.
    \end{equation*}
    Using Weakening (Lemma $\ref{weak}$), we obtain
    \begin{equation}\label{emb8}
        \H[r,\vec s]\dash{\Omega_\omega+m}{\Omega_\omega\cdot\omega^m}\Gamma'(\vec s)^{\Lo},r\notin s_i,A(r,\vec s)^{\Lo}, \neg B(r).
    \end{equation}
    On the other hand, by Lemma $\ref{derive2}.2$ we have
    \begin{equation*}\label{emb9}
        \Vdash \neg B(r),B(r).
    \end{equation*}
    and by Lemma $\ref{Vdash}.4$ we have
    \begin{equation*}\label{emb10}
        \Vdash r=r.
    \end{equation*}
    Since $|r|<|s_i|$, we have the following derivation
    \begin{equation*}
        \begin{prooftree}
            \hypo{\Vdash \neg B(r), B(r)}
            \hypo{r=r}
            \infer[left label = $(\wedge)$]2{&\Vdash \neg B(r), B(r)\wedge r=r}
            \infer[left label = $(\in)$]1{&\Vdash \neg B(r),r\in s_i}.
        \end{prooftree}
    \end{equation*}
    By Weakening (Lemma $\ref{weak}$), we get 
    \begin{equation}\label{emb11}
        \H[r,\vec s]\dash{\Omega_\omega+m}{\Omega_\omega\cdot\omega^m}\Gamma'(\vec s)^{\Lo},\neg B(r),r\in s_i,A(r,\vec s)^{\Lo}.
    \end{equation}
    Applying $(Cut)$ to \eqref{emb8} and \eqref{emb11} yields
    \begin{equation*}
        \H[r,\vec s]\dash{\Omega_\omega+m}{\Omega_\omega\cdot\omega^m+1}\Gamma'(\vec s)^{\Lo},\neg B(r),A(r,\vec s)^{\Lo}.
    \end{equation*}
    We have the following derivation
     \begin{equation*}
         \begin{prooftree}
             \hypo{&\H[r,\vec s]\dash{\Omega_\omega+m}{\Omega_\omega\cdot\omega^m+1}\Gamma'(\vec s)^{\Lo},\neg B(r),A(r,\vec s)^{\Lo}}
             \infer[left label = $(\vee)$]1{&\H[r,\vec s]\dash{\Omega_\omega+m}{\Omega_\omega\cdot\omega^m+2}\Gamma'(\vec s)^{\Lo},\neg B(r)\vee A(r,\vec s)^{\Lo},A(r,\vec s)^{\Lo}}
             \infer[left label = $(\vee)$]1{&\H[r,\vec s]\dash{\Omega_\omega+m}{\Omega_\omega\cdot\omega^m+3}\Gamma'(\vec s)^{\Lo},\neg B(r)\vee A(r,\vec s)^{\Lo},\neg B(r)\vee A(r,\vec s)^{\Lo}}
             \infer[left label = $\equiv$]1{&\H[r,\vec s]\dash{\Omega_\omega+m}{\Omega_\omega\cdot\omega^m+4}\Gamma'(\vec s)^{\Lo}, B(r)\rightarrow A(r,\vec s)^{\Lo}}
             \infer[left label = $(b\forall)$]1{&\H[r,\vec s]\dash{\Omega_\omega+m}{\Omega_\omega\cdot\omega^m+5}\Gamma'(\vec s)^{\Lo},\forall r\in s_i\ A(r,\vec s)^{\Lo}}
         \end{prooftree}
     \end{equation*}
     We use Lemma $\ref{weak}$ to obtain the final bounds:
     \begin{equation*}
         \H[r,\vec s]\dash{\Omega_\omega+m+5}{\Omega_\omega\cdot\omega^{m+5}}\Gamma'(\vec s)^{\Lo},\forall r\in s_i\ A(r,\vec s)^{\Lo}.
     \end{equation*}
     \ \\
     Case 4. We suppose that the last {\sf KPl} rule applied is $(\forall)$. This means that
     \begin{equation*}
         \Gamma(\vec a)=\Gamma'(\vec a),\forall x A(x,\vec a)
     \end{equation*}
    for some {\sf KPl} formula $A$. Therefore,
    \begin{equation*}
        {\sf KPl}\vdash \Gamma'(\vec a),A(c,\vec a),
    \end{equation*}
    with $c\neq a_i$ for all $i\in\{1,\dots, n\}$. By the induction hypothesis, we can find some $m<\omega$ independent of $\vec s$ such that
    \begin{equation*}
        \H[r,\vec s]\dash{\Omega_\omega+m}{\Omega_\omega\cdot\omega^m}\Gamma'(\vec s),A(r,\vec s)^{\Lo}\text{ for all terms $r$ with $|r|<\Omega_\omega$.}
    \end{equation*}
    We apply the $\RS$ rule $(b\forall)$ to get
    \begin{equation*}
        \H[\vec s]\dash{\Omega_\omega+m}{\Omega_\omega\cdot\omega^m+1}\Gamma'(\vec s),\forall x\in\Lo\ A(x,\vec s)^{\Lo}.
    \end{equation*}
     Using Lemma $\ref{weak}$ gives the desired bounds, as we obtain
     \begin{equation*}
         \H[\vec s]\dash{\Omega_\omega+m+1}{\Omega_\omega\cdot\omega^{m+1}}\Gamma'(\vec s),\forall x\in\Lo\ A(x,\vec s)^{\Lo}.
     \end{equation*}
     Case 5. We suppose that the last {\sf KPl} rule applied is $(b\exists)$. This means that 
     \begin{equation*}
         \Gamma(\vec a)=\Gamma'(\vec a), \exists x\in a_i A(x,\vec a)
     \end{equation*}
     for some {\sf KPl} formula $A$ and some free variable $a_i$, with $i\in\{1,\dots,n\}$. Therefore, we have
     \begin{equation*}
         {\sf KPl}\vdash \Gamma'(\vec a)^{\Lo},c\in a_i\wedge A(c,\vec a)
     \end{equation*}
     for some free variable $c$.
    We prove this case by splitting two subcases based on whether $c$ is $a_j$ for some $j\in\{1,\dots,n\}$ or not. We divide subcases based on the form of $s_i$. We refer the reader interested in the details to \cite{cookrathj}, Theorem 4.10 (case 3 of the proof, page 36).\\
    \ \\
    Case 6. We suppose that the last {\sf KPl} rule applied is $(\exists)$. This means that
    \begin{equation*}
        \Gamma(\vec a)=\Gamma'(\vec a),\exists x A(x,\vec a).
    \end{equation*}
    for some {\sf KPl} formula $A$. Therefore,
    \begin{equation}\label{6c1}
        {\sf KPl}\vdash \Gamma'(\vec a),A(c,\vec a).
    \end{equation}
    We distinguish two cases depending on whether $c=a_i$ for some $i\in\{1,\dots,n\}$.\\
    We suppose that $c=a_i$ with $i\in\{1,\dots,n\}$. Then, the derivation \eqref{6c1} is
    \begin{equation*}
        {\sf KPl}\vdash \Gamma'(\vec a),A(a_i,\vec a).
    \end{equation*}
    By the induction hypothesis, we can find some $m<\omega$ independent of $\vec s$ such that
    \begin{equation*}
        \H[\vec s]\dash{\Omega_\omega+m}{\Omega_\omega\cdot\omega^m}\Gamma'(\vec s),A(s_i,\vec s)^{\Lo}.
    \end{equation*}
    Since $|s_i|<\Omega_\omega$, we use the $\RS$ rule $(b\exists)$ to get
    \begin{equation*}
        \H[\vec s]\dash{\Omega_\omega+m}{\Omega_\omega\cdot\omega^m+1}\Gamma'(\vec s),\exists xA(x,\vec s)^{\Lo}.
    \end{equation*}
    Finally, by means of Lemma $\ref{weak}$, we are able to obtain the desired ordinal bounds, as follows
    \begin{equation*}
        \H[\vec s]\dash{\Omega_\omega+m+1}{\Omega_\omega\cdot\omega^{m+1}}\Gamma'(\vec s),\exists xA(x,\vec s)^{\Lo}.
    \end{equation*}
    We suppose now that $c\neq a_i$ for any $i\in\{1,\dots, n\}$. Since in the translation of {\sf KPl} formulas to $\RS$-formulas, free variables become terms in, we can assign to $c$ any term we want with level below $\Omega_\omega$, and so we choose $\overline\emptyset$ while applying the induction hypothesis, that produces an $m<\omega$ such that
    \begin{equation*}
        \H[\vec s]\dash{\Omega_\omega+m}{\Omega_\omega\cdot\omega^m}\Gamma'(\vec s),A(\overline{\emptyset},\vec s)^{\Lo}.
    \end{equation*}
    Here, we apply the $\RS$ rule $(\exists)$ to obtain
    \begin{equation*}
        \H[\vec s]\dash{\Omega_\omega+m}{\Omega_\omega\cdot\omega^m+1}\Gamma'(\vec s),\exists x\in\Lo A(x,\vec s)^{\Lo}.
    \end{equation*}
    Finally, by means of Lemma $\ref{weak}$, we obtain the desired ordinal bounds, as follows
    \begin{equation*}
        \H[\vec s]\dash{\Omega_\omega+m+1}{\Omega_\omega\cdot\omega^{m+1}}\Gamma'(\vec s),\exists x\in\Lo A(x,\vec s)^{\Lo}.
    \end{equation*}
    Case 7. We suppose that the last {\sf KPl} rule applied is $(Cut)$. This means that there is some {\sf KPl} formula $A(\vec a,b_1,\dots,b_k)$, where $b_1,\dots,b_k$ are all the free variables occuring in $A$ different from the free variables in $\vec a$, such that
    \begin{equation*}
        {\sf KPl}\vdash \Gamma(\vec a),A(\vec a,b_1,\dots,b_k)\text{ and }{\sf KPl}\vdash \Gamma(\vec a), 
        \neg A(\vec a,b_1,\dots,b_k).
    \end{equation*}
    Since the level of $\overline\emptyset$ is below $\Omega_\omega$, we can choose $\overline\emptyset$ as the term replacing $b_j$ in the $\RS$-formula $A^{\Lo}$ for all $j\in\{1,\dots,k\}$, and so by the induction hypothesis we can find $m_0,m_1<\Omega_\omega$ such that 
    \begin{equation}\label{cutt1}
        \H[\vec s]\dash{\Omega_\omega+m_0}{\Omega_\omega\cdot\omega^{m_0}}\Gamma(\vec s)^{\Lo},A(\vec s,\overline{\emptyset},\dots,\overline{\emptyset})^{\Lo}.
    \end{equation}
    and
    \begin{equation}\label{cutt2}
        \H[\vec s]\dash{\Omega_\omega+m_1}{\Omega_\omega\cdot\omega^{m_1}}\Gamma(\vec s)^{\Lo},\neg A(\vec s,\overline{\emptyset},\dots,\overline{\emptyset})^{\Lo}.
    \end{equation}
    We apply the $\RS$ rule $(Cut)$ to \eqref{cutt1} and \eqref{cutt2} and obtain
    \begin{equation*}
        \H[\vec s]\dash{\Omega_\omega+\mathrm{max}(m_0,m_1)}{\Omega_\omega\cdot\omega^{\mathrm{max}(m_0,m_1)+1}}\Gamma(\vec s)^{\Lo}.
    \end{equation*}
    \end{proof} 
 $\setcounter{equation}{0}$   
 By inspecting the proof of Theorem \ref{Embed}, we see that the $m$ appearing in the ordinal bounds of the theorem can be constructively retreived from the {\sf KPl}-proof of $\Gamma(a_1,\dots,a_n)$.
\section{The total set-recursive-from-\texorpdfstring{$\omega$}{omega} functions of {\sf KPl}} \label{SectTheoremMainKPl}
We are now ready to classify the total set-recursive-from-$\omega$ functions of {\sf KPl}. We will give a bound for $f(x)$, where $x$ is any set and $f$ is a provably total set-recursive-from-$\omega$ function in {\sf KPl}. By Slaman's theorem (\cite{slaman}), this means that {\sf KPl} proves that $f$ is total and uniformly $\Sigma$-definable with parameter $\omega$ in every admissible set. This bound will be $G_n(x)$, defined as follows. This definition depends on a fixed set $X$, following Reading Convention $\ref{read}$. This means that the ordinals $\Omega_n$ for $n\leq \omega$ and the functions $\psi_n$ for $n<\omega$ are also fixed and depend on $X$ and the set-theoretic rank $\theta$ of $X$.
\begin{definition}\index[Symbols]{$e_n$} \index[Symbols]{$G_n(X)$}$\label{en}$
    We define the ordinal $e_n$ by recursion on $n\in\omega$ as follows:
    \begin{itemize}
        \item $e_0=\Omega_\omega+1$,
        \item $e_{n+1}=\omega^{e_n}$.
    \end{itemize}
    Now, we define for each $n<\omega$ the set $G_n(x)=L_{\psi_0(e_{n+3})}(x)$.
\end{definition}
We observe that the definition of $e_n$ can be written in a $\Sigma_1$ way. The following  result follows the idea that conclusions of derivation can be seen as disjunctions.
\begin{Lemma}
     $\label{vee}$
     Let $\H$ be any operator. Let $\Gamma\cup\{A,B\}$ be a finite set of formulas. If $\H\dash{\rho}{\alpha}\Gamma, A\vee B$ then $\H\dash{\rho}{\alpha}\Gamma, A,B$.
\end{Lemma}
This lemma is easily shown by induction on $\alpha$. The last preliminary result states that $B_0(e_{n+1})$ always contains $e_n$, and so $\psi_0(e_n)<\psi_0(e_{n+1})$ by Lemma \ref{311}.
\begin{Lemma}$\label{enlemma}$
    For every natural number $n$ we have
    \begin{equation*}
        e_n\in B_0(e_{n+1}).
    \end{equation*}
\end{Lemma}
\begin{proof}
    We proceed by induction on $n$.\\
    For $n=0$, we have $e_0=\Omega_\omega+1\in B_0(e_1)$ since, by definition, we have $1,\Omega_\omega\in B_0(e_1)$.\\
    We suppose that $e_n\in B_0(e_{n+1})$. We have to show that $e_{n+1}\in B_0(e_{n+2})$. We have $e_{n+1}=\omega^{e_n}=\phi0e_n$. By the induction hypothesis,
    $e_n\in B_0(e_{n+1})\subseteq B_0(e_{n+2})$.
    It follows that $0,e_n\in B_0(e_{n+2})$, and therefore $e_{n+1}=\phi0e_n\in B_0(e_{n+2})$.
 
\end{proof}
We finally state and prove our main theorem. 
\begin{Theorem}[Main Theorem]
$\label{main}$
    Let $f$ be a provably total set-recursive-from-$\omega$ function in {\sf KPl}. Then, there is some $n<\omega$ such that
    \begin{center}
        $V\vDash \forall x\big(f(x)\in G_n(x)\big)$.
    \end{center}
\end{Theorem}
\begin{proof}
     By Slaman's theorem (\cite{slaman}), $f$ is, provably in {\sf KPl}, uniformly $\Sigma_1$ definable with parameter $\omega$ in any admissible set.
    Let $A_f(\cdot,\cdot)$ be the $\Sigma$-formula that defines $f$ in any admissible set, so that
\begin{equation}\label{equation:MainAss}
{\sf KPl}\vdash Ad(b)\rightarrow[\forall x\in b\exists!y\in b\ A_f(x,y)^b].    
\end{equation}
We fix a set $X$.
Let $\theta$ be the set-theoretic rank of $X$, as in Reading Convention $\ref{read}$. 
In particular, from \eqref{equation:MainAss} we see that {\sf KPl} proves that $f$ is total and $A_f\big(X,f(X)\big)$ is satisfied in $L_{\Omega_0}(X)$. 
     By Theorem $\ref{Embed}$, we have
    \begin{equation*}
        \H_0\dash{\Omega_\omega+m}{\Omega_\omega\cdot\omega^m}Ad\big(\mathbb L_{\Omega_0}(X)\big)\rightarrow \forall x\in \mathbb L_{\Omega_0}(X)\exists!y\in \mathbb L_{\Omega_0}(X) A_f(x,y)^{\mathbb L_{\Omega_0}(X)},
    \end{equation*}
    which is exactly
    \begin{equation*}
        \H_0\dash{\Omega_\omega+m}{\Omega_\omega\cdot\omega^m}\neg Ad\big(\mathbb L_{\Omega_0}(X)\big)\vee\forall x\in \mathbb L_{\Omega_0}(X)\exists! y\in \mathbb L_{\Omega_0}(X) A_f(x,y)^{\mathbb L_{\Omega_0}(X)}.
    \end{equation*}
    Therefore, by Lemma $\ref{vee}$ we obtain
    \begin{equation}\label{main1}
        \H_0\dash{\Omega_\omega+m}{\Omega_\omega\cdot\omega^m}\neg Ad\big(\mathbb L_{\Omega_0}(X)\big),\forall x\in \mathbb L_{\Omega_0}(X)\exists!y\in \mathbb L_{\Omega_0}(X) A_f(x,y)^{\mathbb L_{\Omega_0}(X)}.
    \end{equation}
    On the other hand, we have by Lemma $\ref{Ad}$ and Weakening (Lemma $\ref{weak}$)
    \begin{equation}\label{main2}
        \H_0\dash{\Omega_\omega+m}{\Omega_\omega\cdot\omega^m}Ad\big(\mathbb L_{\Omega_0}(X)\big),\forall x\in \mathbb L_{\Omega_0}(X)\exists!y\in \mathbb L_{\Omega_0}(X) A_f(x,y)^{\mathbb L_{\Omega_0}(X)}.
    \end{equation}
    We notice that Lemma $\ref{Ad}$ actually concerns the relation $\Vdash$, but since this relation means that the derivation can be controlled by any operator with bounds that are below the ones displayed in \eqref{main2}, we can write the derivation controlled by $\H_0$ with those ordinal bounds. We apply $(Cut)$ to \eqref{main1} and \eqref{main2} to obtain
    \begin{equation*}
        \H_0\dash{\Omega_\omega+m}{\Omega_\omega\cdot\omega^m}\forall x\in \mathbb L_{\Omega_0}(X)\exists!y\in \mathbb L_{\Omega_0}(X) A_f(x,y)^{\mathbb L_{\Omega_0}(X)}.
    \end{equation*}
    We notice that $\mathrm{rk}\Big(Ad\big(\mathbb L_{\Omega_0}(X)\big)\Big)=\Omega_0+5$ and so the complexity of the cuts has not been increased. We now apply Lemma $\ref{inversion}$ (Inversion) two times. First, since 
    \begin{equation*}
        \exists! y\in \mathbb L_{\Omega_0}(X) A_f(\overline X,y)^{\mathbb L_{\Omega_0}(X)}\in\C\big(\forall x\in\mathbb L_{\Omega_0}(X)\exists!y\in\mathbb L_{\Omega_0}(X)A_f(x,y)^{\mathbb L_{\Omega_0}(X)}\big),
    \end{equation*}
    we get
    \begin{equation*}
        \H_0\dash{\Omega_\omega+m}{\Omega_\omega\cdot\omega^m}\exists! y\in \mathbb L_{\Omega_0}(X) A_f(\overline X,y)^{\mathbb L_{\Omega_0}(X)}.
    \end{equation*}
    Now, since the symbol ``!'' abbreviates a conjunction, we finally obtain
    \begin{equation}\label{main3}
        \H_0\dash{\Omega_\omega+m}{\Omega_\omega\cdot\omega^m}\exists y\in \mathbb L_{\Omega_0}(X) A_f(\overline X,y)^{\mathbb L_{\Omega_0}(X)}.
    \end{equation}
    We are now going to eliminate cuts from the derivation \eqref{main3}. By Predicative Cut Elimination (Theorem $\ref{cut elim}$), we have
    \begin{equation*}
        \H_0\dash{\Omega_\omega+1}{e_{m+1}}\exists y\in\mathbb L_{\Omega_0}(X)A_f(\overline X,y)^{\mathbb L_{\Omega_0}(X)}.
    \end{equation*}
    By the Collapsing Theorem \ref{collapsing}, we get
    \begin{equation*}
        \H_{e_{m+2}}\dash{\psi_0(e_{m+2})}{\psi_0(e_{m+2})}\exists y\in\mathbb L_{\Omega_0}(X) A_f(\overline X,y)^{\mathbb L_{\Omega_0}(X)}.
    \end{equation*}
    Again by Theorem $\ref{cut elim}$ (Predicative Cut Elimination), we have
    \begin{equation*}
        \H_{e_{m+2}}\dash 0{\alpha}\exists y\in\mathbb L_{\Omega_0}(X)A_f(\overline X,y)^{\mathbb L_{\Omega_0}(X)}.
    \end{equation*}
    where $\alpha=\phi\big(\psi_0(e_{m+2})\big)\big(\psi_0(e_{m+2})\big)<\Omega_0$. By the Boundedness Lemma \ref{boundedness} we get 
    \begin{center}
        $\H_{e_{m+2}}\dash0\alpha \exists y\in\mathbb{L}_\alpha (X) A_f(\overline X,y)^{\mathbb{L}_\alpha (X)}$.
    \end{center}
    We shall see that this implies
    \begin{center}
        $L_\alpha(X)\vDash \exists y A_f(X,y)$.
    \end{center}
    Actually, we prove a more general result.
    \begin{claim}$\label{fin cl}$
    Let $A^{\mathbb L_\delta(X)}$ be a $\Sigma^{\mathbb L_\delta(X)}$-formula. Let $\beta,\delta<\Omega_0$. If $\H_\gamma\dash{0}{\beta}A^{\mathbb L_\delta(X)}$ then $L_\delta(X)\vDash A$.
    \end{claim}
    \renewcommand\qedsymbol{$\blacksquare$}
    \begin{proof}
    We fix $\delta$. We prove Claim $\ref{fin cl}$ by induction on $\beta$. We observe that, since $\beta<\Omega_0$, no reflection rule has been applied. Moreover, since the proof is cut-free, by the subformula property the predicate $Ad$ does not appear in the proof, and so the rules $(Ad)$ and $(\neg Ad)$ have not been applied. If $A^{\mathbb L_\delta(X)}$ is an axiom, then $A^{\mathbb L_\delta(X)}$ is a basic formula. We suppose that this formula is $\overline{u}\in\overline{v}$. This means that $u,v\in TC(\{X\})$ satisfy $u\in v$. But $TC(\{X\})\subseteq L_\delta(X)$. It follows that $L_\delta(X)\vDash A$.\\
    We assume that $A^{\mathbb L_\delta(X)}$ has been derived using a rule (R) different from $(Cut)$, $({\sf Ref}_{n})$, $(Ad)$ and $(\neg Ad)$. Then, we have $\H_\gamma[t_A(B)]\dash{0}{\beta_B}B^{\mathbb L_\delta(X)}$ for some/any premise $B^{\mathbb L_\delta(X)}\in \C\big(A^{\mathbb L_\delta(X)}\big)$. But all those premises are also $\Sigma^{\mathbb L_\delta(X)}$-formulas. By the induction hypothesis, $L_\delta(X)\vDash B$ for some/any $B$. Finally, apply the same rule $(R)$ but in $L_\delta(X)$ to obtain $L_\delta(X)\vDash A$.
    \end{proof}
    \renewcommand\qedsymbol{$\square$}
    Hence, using Claim $\ref{fin cl}$, we get $L_\alpha(X)\vDash \exists y\ A_f(X,y)$. This means that $f(x)\in L_\alpha(X)$.    
    Finally, $L_\alpha(X)\subseteq G_{m+3}(X)$. By Lemma $\ref{enlemma}$, we have
    \begin{equation*}
        e_{m+2}\in B_0(e_{m+3}).
    \end{equation*}
    It follows that $\psi_0(e_{m+2})<\psi_0(e_{m+3})$. Thus,
    \begin{equation*}
        \alpha=\phi\big(\psi_0(e_{m+2})\big)\big(\psi_0(e_{m+2})\big)<\psi_0(e_{m+3}).
    \end{equation*}
    Hence,
    \begin{equation*}
        L_\alpha(X)\subseteq L_{\psi_0(e_{m+3})}(X)=G_{m+3}(X),
    \end{equation*}
as desired.
\end{proof}
$\setcounter{equation}{0}$

A slight strengthening of Theorem \ref{main} will be obtained by combining this argument and the results of \S\ref{SectWellOrderingProofs}.

\section{The provably total set-recursive-from-\texorpdfstring{$\omega$}{omega} functions of \texorpdfstring{${\sf KPl^r}$}{KPlr} and \texorpdfstring{${\sf W-KPl}$}{W-KPl}} 
\label{SectTheoremMainWKPl}
In this section, we show two results similar to Theorem \ref{main} for ${\sf KPl^r}$ and {\sf W-KPl}.
\subsection{The provably total functions of \texorpdfstring{${\sf KPl^r}$}{KPlr}}\label{SectKPlr}
We classify the total set-recursive-from-$\omega$ functions of ${\sf KPl^r}$, which is ${\sf KPl}$ with the axiom of Induction restricted to $\Delta_0$-formulas, following the same reasoning as for {\sf KPl}. We will embed ${\sf KPl^r}$ into the same infinitary system introduced in \S$\ref{rs}$. Nonetheless, we will proceed in a slightly different way. In Theorem $\ref{totalr}$, we will show that if $f$ is provably total in ${\sf KPl^r}$ then $f(x)\in L_{\psi_0(\Omega_n)}(X)$ for some $n$, as the (relativized) proof-theoretic ordinal of ${\sf KPl^r}$ is $\psi_0(\Omega_\omega)$ (cf. \cite{pohlerspaper}). This means that we want to embed ${\sf KPl^r}$-proofs into the infinitary system with ordinal bounds $\Omega_n+k$ for some $n,k<\omega$. Thus, if ${\sf Ax}$ is an axiom of ${\sf KPl^r}$, we would like to have
\begin{equation}\label{comment}
    \H\dash{\Omega_n+1}{\Omega_n+k}({\sf Ax})^{\mathbb L_{\Omega_n}(X)}.
\end{equation}
for some $n,k<\omega$. Nonetheless, \eqref{comment} does not hold for the limit axiom (see Lemma $\ref{Lim}$). To avoid this problem, we will use another presentation of first-order logic.\\
$\setcounter{equation}{0}$
We introduce the one-sided sequent calculus $\dash{}{k}\Gamma$.
\begin{definition}$\label{fol}$
    For every $k<\omega$ and every finite set of formulas $\Gamma$ we define the relation $\dash{}{k}\Gamma$ as follows.\\
    \ \\
    (AxL) $\dash{}{k}\Gamma,A,\neg A$ for any $k<\omega$ and any finite set of formulas $\Gamma,A$.\\
    $(\vee)$ If $\dash{}{k_0}\Gamma, A\vee B, A$ or $\dash{}{k_0}\Gamma, A\vee B, B$ for some $k_0<k$ then $\dash{}{k}\Gamma, A\vee B$.\\
    $(\wedge)$ If $\dash{}{k_0}\Gamma, A\wedge B, A$ and $\dash{}{k_1} \Gamma, A\wedge B, B$ with $k_0,k_1<k$ then $\dash{}{k}\Gamma, A\wedge B$.\\
    $(\exists)$ If $\dash{}{k_0}\Gamma, \exists x\, A(x), A(t)$ with $k_0<k$ then $\dash{}{k}\Gamma, \exists x\ A(x)$.\\
    $(\forall)$ If $\dash{}{k_0}\Gamma, \forall x\, A(x), A(u)$ with $k_0<k$ and $u$ is not free in $\Gamma,\forall x\, A(x)$ then we have $\dash{}{k}\Gamma,\forall x\, A(x)$.
\end{definition}
By routine methods one establishes that this proof system is sound and complete for first-order logic, and so by the Deduction Theorem we have that
\begin{equation*}
    {\sf KPl^r}\vdash\Gamma\text{ if and only if }\dash{}{k}\neg A_1,\dots,\neg A_n,\Gamma
\end{equation*}
for some $k<\omega$ and some instances $A_1,\dots,A_n$ of axioms of ${\sf KPl^r}$. The main idea is that we want to ignore any instance of the axiom $(Lim)$ while embedding ${\sf KPl^r}$-derivations into the infinitary system. To do that, we go through the calculus introduced in Definition $\ref{fol}$. In particular, our first step is to show that we can embed first-order logic derivations of $\Sigma$-formulas into the infinitary system eliminating any instance of $\neg Lim$.
\begin{Lemma}$\label{neglim}$
    Let $\Gamma(\vec a)$ be a finite set of $\Sigma$-formulas. We suppose $\dash{}{k}\neg Lim,\Gamma(\vec a)$ for some $k<\omega$. Then, given any operator $\H$, for any $n<\omega$ there is $0<m<\omega$ such that, for any $\vec s$ with $|s_i|<\Omega_{n}$, we have
    \begin{equation*}
        \H[\vec s]\dash{\Omega_{n+m}+1}{\Omega_{n+m}+k}\Gamma(\vec s)^{\L_{\Omega_{n+m}}(X)}.
    \end{equation*}
\end{Lemma}
\begin{proof}
    We proceed by induction on $k$. If $\Gamma(\vec a)$ is a logical axiom, then the result holds.\\
    We first assume that the principal formula of the last rule applied is $A(\vec a)$, with $\Gamma(\vec a)=\Gamma'(\vec a), A(\vec a)$. We distinguish cases based on the rule.\\
    \ \\
    Case 1. We assume the last rule applied is $(\vee)$. This means that $\Gamma(\vec a)$ is $\Gamma'(\vec a), A(\vec a)\vee B(\vec a)$ and we have the premise, say, 
    \begin{equation*}
        \dash{}{k_0}\neg Lim, \Gamma'(\vec a),A(\vec a).
    \end{equation*}
    By the induction hypothesis, we get
    \begin{equation*}
        \H[\vec s]\dash{\Omega_{n+m}+1}{\Omega_{n+m}+k_0}\Gamma'(\vec s)^{\L_{\Omega_{n+m}}(X)},A(\vec s)^{\L_{\Omega_{n+m}}(X)}.
    \end{equation*}
    An application of $(\vee)$ gives
    \begin{equation*}
        \H[\vec s]\dash{\Omega_{n+m}+1}{\Omega_{n+m}+k}\Gamma'(\vec s)^{\L_{\Omega_{n+m}}(X)},A(\vec s)^{\L_{\Omega_{n+m}}(X)}\vee B(\vec s)^{\L_{\Omega_{n+m}}(X)}.
    \end{equation*}
    \ \\
    Case 2. We assume the last rule applied is $(\wedge)$. Then the argument is analogous to Case 1.\\
    \ \\
    Case 3. We assume the last rule applied is $(\exists)$. We distinguish two subcases based on whether the existential quantifier is restricted or not.\\
    Subcase 3.1. We suppose that the quantifier is unrestricted, so we have 
    \begin{equation*}
        \dash{}{k}\neg Lim,\Gamma'(\vec a),\exists x\ A(x,\vec a).
    \end{equation*}
    Then, we have the premise
    \begin{equation*}
        \dash{}{k_0}\neg Lim,\Gamma'(\vec a),A(b,\vec a).
    \end{equation*}
    By the induction hypothesis, we obtain
    \begin{equation*}
        \H[\vec s]\dash{\Omega_{n+m}+1}{\Omega_{n+m}+k_0}\Gamma'(\vec s)^{\L_{\Omega_{n+m}}(X)},A(\vec s)^{\L_{\Omega_{n+m}}(X)}.
    \end{equation*}
    for any $\vec s\in \L_n(X)$. Therefore, by applying $(b\exists)$, we get
    \begin{equation*}
        \H[\vec s]\dash{\Omega_{n+m}+1}{\Omega_{n+m}+k}\Gamma'(\vec s)^{\L_{\Omega_{n+m}}(X)},\exists x\in\L_{\Omega_{n+m}}(X)\ A(\vec s)^{\L_{\Omega_{n+m}}(X)}.
    \end{equation*}
    Subcase 3.2. We suppose that the quantifier is restricted, so we have 
    \begin{equation*}
        \dash{}{k}\neg Lim,\Gamma'(a),\exists x(x\in a_i\wedge A(x,\vec a)).
    \end{equation*}
    Therefore, we have the premise
    \begin{equation*}
        \dash{}{k_0}\neg Lim,\Gamma'(\vec a),x\in a_i\wedge A(x,\vec a).
    \end{equation*}
    By the induction hypothesis, we have
    \begin{equation}\label{kplrlem1}
        \H[s,\vec s]\dash{\Omega_{n+m}+1}{\Omega_{n+m}+k_0}\Gamma(\vec s)^{\L_{\Omega_{n+m}}(X)},s\in s_i\wedge A(s,\vec s).
    \end{equation}
    At this point, we want to use the rule $(b\exists)$. We recall that this rule has three different forms, according to Definition $\ref{rules}$ and Example $\ref{exC}$. So we have to distinguish cases based on the form of $s_i$.\\
    If $s_i\equiv \overline u$, we get the result from \eqref{kplrlem1} applying $(b\exists)$. Otherwise, by Lemma $\ref{inversion}$ on \eqref{kplrlem1}, we get
    \begin{equation}\label{rlem2}
        \H[s,\vec s]\dash{\Omega_{n+m}+1}{\Omega_{n+m}+k_0}\Gamma'(\vec s)^{\L_{\Omega_{n+m}}(X)},s\in s_i
    \end{equation}
    and
    \begin{equation}\label{rlem3}
        \H[s,\vec s]\dash{\Omega_{n+m}+1}{\Omega_{n+m}+k_0}\Gamma'(\vec s)^{\L_{\Omega_{n+m}}(X)},A(s,\vec s).
    \end{equation}
    If $s_i\equiv \L_\beta(X)$ for some $\beta$, then by $(b\exists)$ on \eqref{rlem2} we get the result.\\
    We assume now that $s_i\equiv[z\in \L_\beta(X):B(z)]$. To apply $(b\exists)$, we need to have $\H[s,\vec s]\dash{}{}\Gamma'(\vec s)^{\L_{\Omega_{n+m}}},B(s)\wedge A(s,\vec s)$. We need the following claim.
    \begin{claim}$\label{claimyb}$
        Let $\H$ be any operator. Let $\Delta, C(a)$ be any finite set of formulas. Let $\alpha$ be any ordinal and let $\rho>\max\{\mathrm{rk}(C(\mathbb L_\delta(X)))\}$.
        \begin{equation*}
            \text{If }\H\dash{\rho}{\alpha}\Delta, t\in\{x\in\L_\delta(X):C(x)\}\text{ then }\H\dash{\rho}{\alpha}\Delta,C(t).
        \end{equation*}
    \end{claim}
    \renewcommand\qedsymbol{$\blacksquare$}
    \begin{proof}
    Claim $\ref{claimyb}$ can be proved by induction on $\alpha$. We consider the only interesting case where the principal formula in the last derivation of $\H\dash{\rho}{\alpha}\Delta, t\in\{x\in\L_\delta(X):C(x)\}$ is $t\in\{x\in\L_\delta(X):C(x)\}$. This means that we have for some $\alpha_0<\alpha$ and some term $r$ with $|r|<\mathrm{rk}(\mathbb L_\delta(X))$
    \begin{equation*}
        \H\dash{\rho}{\alpha_0}\Delta, t=r\wedge C(r).
    \end{equation*}
    By Inversion and Weakening, we get
    \begin{equation*}
        \H\dash{\rho}{\alpha_0}\Delta, C(r), t\neq r, C(t).
    \end{equation*}
    But using Leibniz principle, by Lemma $\ref{axiomsembed}$, we also have
    \begin{equation*}
        \H\dash{\rho}{\alpha_0}\Delta, t=r,C(t),\neg C(r).
    \end{equation*}
    So we use $(Cut)$ to obtain
    \begin{equation*}
        \H\dash{\rho}{\alpha_0}\Delta, C(t).
    \end{equation*}
    The last application of $(Cut)$ shows why we need the condition on $\rho$ in the hypothesis of the claim.
    \end{proof}
    \renewcommand\qedsymbol{$\square$}
    So from \eqref{rlem2} we obtain using Claim $\ref{claimyb}$
    \begin{equation}\label{rlem4}
        \H[s,\vec s]\dash{\Omega_{n+m}+1}{\Omega_{n+m}+k_0}\Gamma'(\vec s)^{\L_{\Omega_{n+m}}(X)},B(s).
    \end{equation}
    An application of $(\wedge)$ on \eqref{rlem3} and \eqref{rlem4} yields
    \begin{equation*}
        \H[s,\vec s]\dash{\Omega_{n+m}+1}{\Omega_{n+m}+k}\Gamma'(\vec s)^{\L_{\Omega_{n+m}}(X)},B(s)\wedge A(s,\vec s).
    \end{equation*}
    Finally, we apply $(b\exists)$ to obtain
    \begin{equation*}
        \H[s,\vec s]\dash{\Omega_{n+m+1}+1}{\Omega_{n+m+1}+k}\Gamma'(\vec s)^{\L_{\Omega_{n+m}}(X)},\exists x\in s\ A(s,\vec s).
    \end{equation*}
    Case 4. We assume the last rule applied is $(\forall)$. Then, the quantifier introduced by this rule must be restricted since $\Gamma(\vec a)$ is a set of $\Sigma$-formulas. Therefore, we have $\Gamma(\vec a)=\Gamma'(\vec a),\forall y\in a_i\ A(y,\vec a)$, and we have
    \begin{equation*}
        \dash{}{k}\neg Lim,\Gamma'(\vec a),\forall y\in a_i\ A(y,\vec a)
    \end{equation*}
    with the premise
    \begin{equation*}
        \dash{}{k_0}\neg Lim,\Gamma'(\vec a),b\in a_i\rightarrow A(b,\vec a)
    \end{equation*}
    with $b\neq a_j$ for all $j$ and with $k_0<k$. Therefore, by the induction hypothesis, we get
    \begin{equation}\label{rlem5}
        \H[t,\vec s]\dash{\Omega_{n+m}+1}{\Omega_{m+n}+k_0}\Gamma'(\vec s)^{\L_{\Omega_{n+m}}(X)},t\in s_i\rightarrow A(t,\vec s)^{\L_{\Omega_{n+m}}(X)}
    \end{equation}
    for all $|t|<\Omega_n$. Now, we separate cases based on the form of $s_i$.\\
    Subcase 4.1. If $s_i\equiv \overline{u}$, then \eqref{rlem5} is exactly
    \begin{equation*}
        \H[t,\vec s]\dash{\Omega_{n+m}+1}{\Omega_{m+n}+k_0}\Gamma'(\vec s)^{\L_{\Omega_{n+m}}(X)},t\dot\in s_i\rightarrow A(t,\vec s)^{\L_{\Omega_{n+m}}(X)}.
    \end{equation*}
    By an application of $(b\forall)$, we obtain
    \begin{equation*}
        \H[t,\vec s]\dash{\Omega_{n+m}+1}{\Omega_{m+n}+k}\Gamma'(\vec s)^{\L_{\Omega_{n+m}}(X)},\forall y\in s_i\ A(y,\vec s)^{\L_{\Omega_{n+m}}(X)}.
    \end{equation*}
    Subcase 4.2. We suppose $s_i\equiv \L_\beta(X)$ for some $\beta<\Omega_n$. Then, by Lemma $\ref{vee}$ on \eqref{rlem5}, we get
    \begin{equation}\label{rlem6}
        \H[t,\vec s]\dash{\Omega_{n+m}+1}{\Omega_{m+n}+k_0}\Gamma'(\vec s)^{\L_{\Omega_{n+m}}(X)},\neg t\in\mathbb L_\beta(X),A(t,\vec s)^{\L_{\Omega_{n+m}}(X)}.
    \end{equation}
    This holds for all $|t|<\Omega_n$, and so, in particular, for all $|t|<\omega\cdot\beta$. Moreover, we also have for all $|t|<\omega\cdot\beta$
    \begin{equation}\label{rlem7}
        \H[t,\vec s]\dash{\Omega_{n+m}+1}{\Omega_{m+n}+k_0}\Gamma'(\vec s)^{\L_{\Omega_{n+m}}(X)}, t\in\mathbb L_\beta(X),A(t,\vec s)^{\L_{\Omega_{n+m}}(X)}.
    \end{equation}
    Hence, by $(Cut)$ on \eqref{rlem6} and \eqref{rlem7}, we obtain
    \begin{equation*}
        \H[t,\vec s]\dash{\Omega_{n+m}+1}{\Omega_{m+n}+k}\Gamma'(\vec s)^{\L_{\Omega_{n+m}}(X)},A(t,\vec s)^{\L_{\Omega_{n+m}}(X)}.
    \end{equation*}
    An application of $(b\forall)$ yields
    \begin{equation*}
        \H[\vec s]\dash{\Omega_{n+m}+1}{\Omega_{m+n}+k+1}\Gamma'(\vec s)^{\L_{\Omega_{n+m}}(X)}, \forall y\in\mathbb L_\beta(X)\ A(y,\vec s)^{\L_{\Omega_{n+m}}(X)}.
    \end{equation*}
    Subcase 4.3. We suppose $s_i\equiv [x\in\mathbb L_\beta(X):B(x)]$. Then, \eqref{rlem5} is exactly
    \begin{equation}\label{rlem8}
         \H[t,\vec s]\dash{\Omega_{n+m}+1}{\Omega_{m+n}+k_0}\Gamma'(\vec s)^{\L_{\Omega_{n+m}}(X)},\neg t\in[x\in \mathbb L_\beta(X):B(x)]\vee A(t,\vec s)^{\L_{\Omega_{n+m}}(X)}
    \end{equation}
    for all $|t|<\Omega_n$ and so, in particular, for all $|t|<\omega\cdot\beta$. We will use the following claim.
    \begin{claim}\label{negb}
        Let $|t|<\omega\cdot\beta$. Let $\H$ be any operator. Let $\alpha$ be any ordinal and let $\rho>\omega\cdot\beta$. If $\H\dash{\rho}{\alpha}\Delta,\neg t\in[x\in \mathbb L_\beta(X):B(x)]\vee A$ then $\H\dash{\rho}{\alpha}\Delta,\neg B(t)\vee A$.
    \end{claim}
    \renewcommand\qedsymbol{$\blacksquare$}
    \begin{proof}
    We prove Claim \ref{negb} by induction on $\alpha$. If $\Delta$ is an axiom, then the result holds. If the last rule applied has not principal formula $\neg t\in[x\in \mathbb L_\beta(X):B(x)]\vee A$, then we apply the induction hypothesis to the premise(s) and use the rule again.\\
    We suppose that $\neg t\in[x\in \mathbb L_\beta(X):B(x)]\vee A$ is the principal formula of the last inference.\\
    If the premise is $\H\dash{\rho}{\alpha_0}\Delta, A$, then we apply $(\vee)$ to obtain the result. So we assume the premise is $\H\dash{\rho}{\alpha_0}\Delta,\neg [t\in\mathbb L_\beta:B(x)]$. The premises of this derivation are $\H\dash{\rho}{\alpha_0'}\Delta,B(r)\rightarrow r\neq t$ for all $|r|<\omega\cdot\beta$. In particular, we have
    \begin{equation*}
        \H\dash{\rho}{\alpha_0'}\Delta,B(t)\rightarrow t\neq t,
    \end{equation*}
    from what we get by Lemma $\ref{vee}$
    \begin{equation}\label{rlem9}
        \H\dash{\rho}{\alpha_0'}\Delta,\neg B(t), t\neq t.
    \end{equation}
    But we also have
    \begin{equation}\label{rlem10}
        \H\dash{\rho}{\alpha_0'}\Delta,\neg B(t),t=t.
    \end{equation}
    Therefore, by $(Cut)$ on \eqref{rlem9} and \eqref{rlem10}, we obtain
    \begin{equation*}
        \H\dash{\rho}{\alpha_0}\Delta,\neg B(t).
    \end{equation*}
    An application of $(\vee)$ gives the result.
    \end{proof}
    \renewcommand\qedsymbol{$\square$}
    An application of Claim $\ref{negb}$ on \eqref{rlem8} yields
    \begin{equation*}
        \H[t,\vec s]\dash{\Omega_{n+m}+1}{\Omega_{m+n}+k_0}\Gamma'(\vec s)^{\L_{\Omega_{n+m}}(X)},\neg B(t)\vee A(t,\vec s)^{\L_{\Omega_{n+m}}(X)}.
    \end{equation*}
    We apply $(b\forall)$ and we obtain
    \begin{equation*}
        \H[\vec s]\dash{\Omega_{n+m}+1}{\Omega_{m+n}+k_0}\Gamma'(\vec s)^{\L_{\Omega_{n+m}}(X)},\forall y\in s_i\ A(y,\vec s)^{\L_{\Omega_{n+m}}(X)}.
    \end{equation*}
    \ \\
    Finally, we assume that the principal formula of the last rule applied is $\neg Lim$. Then, we have the premise
\begin{equation*}
    \dash{}{k_0}\neg Lim,\forall y(\neg Ad(y)\vee a_i\notin y),\Gamma(\vec a),
\end{equation*}
We apply Lemma $\ref{inversion}$ (Inversion) to obtain 
\begin{equation*}
    \dash{}{k_0}\neg Lim,\neg Ad(b)\vee a_i\notin b,\Gamma(\vec a)\text{ for some free variable $b$ not occurring in $\Gamma(\vec a)$}.
\end{equation*}
Let $n<\omega$. We apply the induction hypothesis to $n+1$, so we have
\begin{equation}\label{rlem11}
    \H[\vec s,t]\dash{\Omega_{n+1+m_0}+1}{\Omega_{n+1+m_0}+k_0}\neg Ad(t)\vee s_i\notin t,\Gamma(\vec s), \text{ for all $t,\vec s$ with $|t|,|s_1|,\dots,|s_m|<\Omega_{n+1}$}. 
\end{equation}
In particular, we consider \eqref{rlem11} for any $t,\vec a$ with $|t|<\Omega_{n+1}$ and $|s_1|,\dots,|s_m|<\Omega_n$. An application of $(b\forall)$ yields
\begin{equation}\label{rlem12}
    \H[\vec s]\dash{\Omega_{n+1+m_0}+1}{\Omega_{n+1+m_0}+k}\forall y\in\Lm(\neg Ad(y)\vee s_i\notin y),\Gamma(\vec s)^{\Lnm}.
\end{equation}
But one can prove
\begin{equation}\label{rlem13}
    \H[\vec s]\dash{\Omega_{n+1+m_0}+1}{\Omega_{n+1+m_0}+k}\exists y\in\Lm(Ad(y)\wedge s_i\in y),\Gamma(\vec s)^{\Lnm}.
\end{equation}
Therefore, using $(Cut)$ on \eqref{rlem12} and \eqref{rlem13} gives
\begin{equation}\label{rlem14}
    \H[\vec s]\dash{\Omega_{n+1+m_0}+1}{\Omega_{n+1+m_0}+k}\Gamma(\vec s)^{\Lnm}.
\end{equation}
We observe that the rank of the formula cut in \eqref{rlem14} is $\Omega_{n+1}$ and so the cut-complexity of the derivation does not increase. We obtain the desired result by taking $m=1+m_0$ in \eqref{rlem14}.
\end{proof}
$\setcounter{equation}{0}$
On the other hand, if ${\sf KPl^r}\vdash \Gamma$ for some finite set of $\Sigma$-formulas $\Gamma$, then there is $k<\omega$ such that for some axioms $A_1,\dots,A_l$ of ${\sf KPl^r}$ we have $\dash{}{k}\neg A_1,\dots,\neg A_l,\Gamma$. But then, by Lemma $\ref{neglim}$, we can embed this derivation into the infinitary system eliminating $\neg Lim$. Afterwards, we can get rid of the remaining negated axioms in the infinitary derivation with $(Cut)$. This is exactly what we prove in the next theorem.
\begin{Theorem}[Embedding Theorem for ${\sf KPl^r}$]$\label{Embedr}$
    Let $\Gamma(\vec a)$ be a finite set of $\Sigma$-formulas. If ${\sf KPl^r}\vdash \Gamma(\vec a)$ then, given any operator $\H$, for any $\vec s$ there are $n<\omega$ and $k<\omega$ such that
    \begin{equation*}
        \H[\vec s]\dash{\Omega_{n}+1}{\Omega_{n}+k}\Gamma(\vec s)^{\Ln}
    \end{equation*}
\end{Theorem}
\begin{proof}
    We suppose ${\sf KPl^r}\vdash \Gamma(\vec a)$. Then
    \begin{equation}\label{rem1}
        \dash{}{k}\neg Lim,\neg A_1,\dots, \neg A_l,\Gamma(\vec a)
    \end{equation}
    for some $k<\omega$ and some axioms $A_1,\dots, A_l$ of ${\sf KPl^r}$ different from $(Lim)$. We can assume that all the formulas $\neg A_1,\dots,\neg A_l$ are $\Sigma$ to use Lemma $\ref{neglim}$ because $Pair,Union,\Delta_0-Separation$ and $Infinity$ can be derived from $(Lim)$ and $(Ad3)$. Hence, all the formulas appearing in \eqref{rem1} are $\Sigma$ except for $\neg Lim$. Given terms $\vec s$, let $n_0<\omega$ be such that $|\vec s|<\Omega_{n_0}$. Then, by Lemma $\ref{neglim}$, we get
    \begin{equation*}
        \H[\vec s]\dash{\Omega_{n_0+m}+1}{\Omega_{n_0+m}+2k}\neg A_1^{\Lnpm},\dots,\neg A_l^{\Lnpm},\Gamma(\vec s)^{\Lnpm}.
    \end{equation*}
    Applying $(Cut)$ $l$-many times, we obtain
    \begin{equation}\label{rem2}
        \H[\vec s]\dash{\Omega_{n_0+m}+1}{\Omega_{n_0+m}+2k+l}\Gamma(\vec s)^{\Lnpm}.
    \end{equation}
    We obtain the desired result by taking $n=n_0+m$ in \eqref{rem2}.
\end{proof}
$\setcounter{equation}{0}$
The following operation on sets will be used to classify the total set-recursive-from-$\omega$ function of ${\sf KPl^r}$.
\begin{definition}
    We define for every $n<\omega$ the operation $\hat G_n(x)$ as
    \begin{equation*}
        \hat G_n(x)=L_{\psi_0(\Omega_{n+1})}(x)
    \end{equation*}
\end{definition}
We observe that, since the definition of $\Omega_{n+1}$ and $\psi_0$ depend on $x$ (see Reading Convention \ref{read}), the definition of $\hat G_n$ also depends on $x$.\\
Now, we classify the total set-recursive-from-$\omega$ function of ${\sf KPl^r}$. We note that this result is analogous to Theorem $\ref{main}$.
\begin{Theorem}$\label{totalr}$
    Let $f$ be a provably total set-recursive-from-$\omega$ function in ${\sf KPl^r}$. Then, there is $n<\omega$ such that
    \begin{equation*}
        V\vDash\forall x(f(x)\in \hat G_n(x)).
    \end{equation*}
\end{Theorem}
\begin{proof}
    Let $A_f(\cdot,\cdot)$ be the $\Sigma$ formula that defines $f$ in any admissible set. In particular, we have that ${\sf KPl^r}\vdash Ad(b)\rightarrow[\forall x\in b\exists!y\in b\ A_f(x,y)^b]$.\\
    Now, we fix a set $X$ and we let $\theta$ be the set-theoretic rank of $X$, as in Reading Convention $\ref{read}$. By Theorem $\ref{Embedr}$, we have
    \begin{equation*}
        \H_0\dash{\Omega_{n}+1}{\Omega_{n}+k}Ad(\mathbb L_{\Omega_0}(X))\rightarrow \forall x\in \mathbb L_{\Omega_0}(X)\exists!y\in \mathbb L_{\Omega_0}(X) A_f(x,y)^{\mathbb L_{\Omega_0}(X)}.
    \end{equation*}
    Following the same reasoning as in the proof of Theorem $\ref{main}$, we get
    \begin{equation}\label{rmain1}
        \H_0\dash{\Omega_{n}+1}{\Omega_{n}+k}\neg Ad(\mathbb L_{\Omega_0}(X)),\forall x\in \mathbb L_{\Omega_0}(X)\exists!y\in \mathbb L_{\Omega_0}(X) A_f(x,y)^{\mathbb L_{\Omega_0}(X)}.
    \end{equation}
    On the other hand, we have by Lemma $\ref{Ad}$ and Lemma $\ref{weak}$
    \begin{equation}\label{rmain2}
        \H_0\dash{\Omega_{n}+1}{\Omega_{n}+k}Ad(\mathbb L_{\Omega_0}(X)),\forall x\in \mathbb L_{\Omega_0}(X)\exists!y\in \mathbb L_{\Omega_0}(X) A_f(x,y)^{\mathbb L_{\Omega_0}(X)}.
    \end{equation}
    We apply $(Cut)$ to \eqref{rmain1} and \ref{rmain2} to obtain
    \begin{equation*}
        \H_0\dash{\Omega_{n}+1}{\Omega_{n}+k+1}\forall x\in \mathbb L_{\Omega_0}(X)\exists!y\in \mathbb L_{\Omega_0}(X) A_f(x,y)^{\mathbb L_{\Omega_0}(X)}.
    \end{equation*}
    We notice that $\mathrm{rk}(Ad(\mathbb L_{\Omega_0}(X)))=\Omega_0+5$ and so the complexity of the cuts has not been increased. We now apply Lemma $\ref{inversion}$ (Inversion) two times to obtain
    \begin{equation}\label{rmain3}
        \H_0\dash{\Omega_{n}+1}{\Omega_{n}+k+1}\exists y\in \mathbb L_{\Omega_0}(X) A_f(\overline X,y)^{\mathbb L_{\Omega_0}(X)}.
    \end{equation}
    We are now going to eliminate cuts from the Derivation \eqref{rmain3}.
    By the Collapsing Theorem \ref{collapsing}, we get
    \begin{equation*}
        \H_{\omega^{\Omega_{n}\cdot 2+k}}\dash{\psi_0(\omega^{\Omega_{n}\cdot 2+k})}{\psi_0(\omega^{\Omega_{n}\cdot 2+k})}\exists y\in\mathbb L_{\Omega_0}(X) A_f(\overline X,y)^{\mathbb L_{\Omega_0}(X)}.
    \end{equation*}
    By Theorem $\ref{cut elim}$ (Predicative Cut Elimination), we have
    \begin{equation*}
        \H_{\omega^{\Omega_{n}\cdot 2+k}}\dash 0{\alpha}\exists y\in\mathbb L_{\Omega_0}(X)A_f(\overline X,y)^{\mathbb L_{\Omega_0}(X)}.
    \end{equation*}
    where $\alpha=\psi_0(\Omega_{n+1})$. Technically, we should have $\alpha=\phi(\psi_0(\omega^{\Omega_{n}\cdot 2+k}))(\psi_0(\omega^{\Omega_{n}\cdot 2+k}))$, but we choose the looser bound $\alpha=\psi_0(\Omega_{n+1})$. By the Boundedness Lemma, we get 
    \begin{center}
        $\H_{\omega^{\Omega_{n}\cdot 2+k}}\dash0\alpha \exists y\in\mathbb{L}_\alpha (X) A_f(\overline X,y)^{\mathbb{L}_\alpha (X)}$.
    \end{center}
    It follows that
    \begin{center}
        $L_\alpha(X)\vDash \exists y A_f(X,y)$.
    \end{center}
    This means that $f(X)\in L_\alpha(X)=\hat G_n(X)$.
\end{proof}
$\setcounter{equation}{0}$
\subsection{The provably total functions of {\sf W-KPl}}
To classify the provably total functions of {\sf W-KPl}, we will argue as in the previous subsection, although we need to extend the calculus introduced in Definition $\ref{fol}$ with the $\omega$-rule and the $Cut$-rule to deal with Mathematical Induction. Moreover, to show cut-elimination, we introduce a measure of complexity of formulas. We opt for defining the complexity of a formula as the number of connectives the formula has.
\begin{definition}
    We define $\mathrm{com}(F)$ for every formula $F$ as follows.
    \begin{enumerate}
        \item $\mathrm{com}(a\in b)=\mathrm{com}(a\notin b)=1$,
        \item $\mathrm{com}(A\vee B)=\mathrm{com}(A\wedge B)=\max(\mathrm{com}(A),\mathrm{com}(B))+1$,
        \item $\mathrm{com}(Ad(A))=\mathrm{com}(\neg Ad(A))=\mathrm{com}(\exists x\ A(x))=\mathrm{com}(\forall x\ A(x))=\mathrm{com}(A)+1$.
    \end{enumerate}
\end{definition}
\begin{definition}
    For every $\alpha<\omega$ and every finite set of formulas $\Gamma$ we define the relation $[\omega]\dash{}{\alpha}\Gamma$ as follows.\\
    \ \\
    (AxL) $[\omega]\dash{r}{\alpha}\Gamma,A,\neg A$ for any $k<\omega$ and any finite set of formulas $\Gamma,A$.\\
    $(\vee)$ If $[\omega]\dash{r}{\alpha_0}\Gamma, A$ or $[\omega]\dash{r}{\alpha_0}\Gamma, B$ for some $\alpha_0<\alpha$ then $[\omega]\dash{r}{\alpha}\Gamma, A\vee B$.\\
    $(\wedge)$ If $[\omega]\dash{r}{\alpha_0}\Gamma,A$ and $[\omega]\dash{r}{\alpha_1} \Gamma, B$ with $\alpha_0,\alpha_1<\alpha$ then $[\omega]\dash{r}{\alpha}\Gamma, A\wedge B$.\\
    $(\exists)$ If $[\omega]\dash{r}{\alpha_0}\Gamma, A(t)$ with $\alpha_0<\alpha$ then $[\omega]\dash{r}{\alpha}\Gamma, \exists x\ A(x)$.\\
    $(\forall)$ If $[\omega]\dash{r}{\alpha_0}\Gamma,A(u)$ with $\alpha_0<\alpha$ and $u$ is not free in $\Gamma,\forall x\ A(x)$ then $[\omega]\dash{r}{\alpha}\Gamma,\forall x\ A(x)$.\\
    $(\omega)$ If $[\omega]\dash{r}{\alpha_n}\Gamma,A(n)$ with $\alpha_n<\alpha$ for every $n<\omega$ for some formula $A(x)$, then $[\omega]\dash{r}{\alpha}\Gamma,\forall n<\omega A(n)$.\\
    $(Cut)$ If $[\omega]\dash{r}{\alpha_0}\Gamma, A$ and $[\omega]\dash{r}{k_0}\Gamma,\neg A$ with $k_0<\omega$ for some formula $A$, then $[\omega]\dash{\max(r,\mathrm{com}(A))}{\alpha}\Gamma$.\\
\end{definition}
Again, by the Deduction Theorem, we have
\begin{equation*}
    {\sf W-KPl}\vdash\Gamma\text{ if and only if }[\omega]\dash{0}{k}\neg A_1,\dots,\neg A_l,\Gamma
\end{equation*}
for some $k<\omega$ and some axioms $A_1,\dots,A_l$ of {\sf W-KPl}.
Actually, the next Lemma shows that we can assume that the axioms $A_1,\dots,A_l$ above are not instances of Mathematical Induction. We refer to \cite{pohlerspaper} and \cite{pohlersbook}.

\begin{Lemma}$\label{induction}$
    If $[\omega]\dash{0}{k}\neg A_1,\dots,\neg A_m,\neg MI_1,\dots,\neg MI_l,\Gamma$ with $A_1,\dots,A_m$ instances of axioms of {\sf W-KPl} different from Mathematical Induction and $MI_1,\dots,MI_l$ instances of Mathematical Induction, then $[\omega]\dash{0}{\alpha}\neg A_1,\dots,\neg A_m,\Gamma$ for some $\alpha<\omega$.
\end{Lemma}
We show that we can embed the augmented one-sided sequent calculus into $\RS$ removing $\neg Lim$.
\begin{Lemma}$\label{limw}$
    Let $\Gamma(\vec a)$ be a finite set of $\Sigma$-formulas such that $[\omega]\dash{}{\alpha}\neg Lim,\Gamma(\vec a)$. Let $\vec s$ with $|\vec s|<\Omega_n$. Then $\H_{\omega^{\Omega_\omega+3\alpha}}[\vec s,t]\dash{\Omega_n+1}{\Omega_\omega+3\alpha}\Gamma(\vec s)^{\Ln}$.
\end{Lemma}
\begin{proof}
    We proceed by induction on $\alpha$. If $\Gamma(\vec a)$ is a logical axiom then the result holds.\\
    We suppose that the last {\sf W-KPl}-rule applied has principal formula in $\Gamma(\vec a)$, or is $(Cut)$. Then we apply the induction hypothesis to the premises and then we apply the adequate rule in the infinitary system. We remark the simple but key difference between this case and those in \S\ref{SectKPlr} that when this last {\sf W-KPl}-rule is the $\omega$-rule, the corresponding derivation in the infinitary proof system might end up with length greater than $\Omega_\omega$. \\
    We suppose now that the principal formula of the last applied rule is $\neg Lim$. Then we have the premise
    \begin{equation*}
        [\omega]\dash{}{\alpha_0}\neg Lim, \forall y (\neg Ad(y)\vee a_i\notin y),\Gamma(\vec a).
    \end{equation*}
    By Lemma $\ref{inversion}$, we get 
    \begin{equation*}
        [\omega]\dash{}{\alpha_0}\neg Lim, \neg Ad(b)\vee a_i\notin b,\Gamma(\vec a).
    \end{equation*}
    for some variable $b$ not occuring in $\Gamma(\vec a)$. By the induction hypothesis, we obtain
    \begin{equation*}
        \H_{\omega^{\Omega_\omega+3\alpha_0}}[\vec s]\dash{\Omega_n+1}{\Omega_\omega+3\alpha_0}\neg Ad(r)\vee s_i\notin r,\Gamma(\vec s)^{\Ln}.
    \end{equation*}
    for all $r$ with $|r|<\Omega_{n+1}$. An application of $(b\forall)$ gives
    \begin{equation}\label{cut1}
        \H_{\omega^{\Omega_\omega+3\alpha_0}}[\vec s]\dash{\Omega_n+1}{\Omega_\omega+3\alpha_0}\forall y\in \mathbb L_{\Omega_{n+1}}(X)(\neg Ad(y)\vee s_i\notin y),\Gamma(\vec s)^{\Ln}.
    \end{equation}
    On the other hand, we can show
    \begin{equation}\label{cut2}
        \H_{\omega^{\Omega_\omega+3\alpha_0}}[\vec s]\dash{\Omega_{n+1}+1}{\Omega_\omega+3\alpha_0}\exists y\in \mathbb L_{\Omega_{n+1}}(X)(Ad(y)\wedge s_i\in y),\Gamma(\vec s)^{\Ln}.
    \end{equation}
    Therefore, an application of $(Cut)$ on $\eqref{cut1}$ and $\eqref{cut2}$ yields 
    \begin{equation*}
        \H_{\omega^{\Omega_\omega+3\alpha_0}}[\vec s]\dash{\Omega_{n+1}+1}{\Omega_\omega+3\alpha_0}\Gamma(\vec s)^{\Ln}.
    \end{equation*}
    We notice that the formula removed by $(Cut)$ has rank $\Omega_{n+1}$ and so the bound of the complexity of the cuts has not grown. Now, by the Collapsing Theorem, we get
    \begin{equation*}
            \H_\gamma[\vec s,t]\dash{\psi_{n}\gamma}{\psi_{n}\gamma}\Gamma(\vec s)^{\Ln},
    \end{equation*}
    where $\gamma=\omega^{\Omega_\omega+3\alpha_0+2}$. By predicative cut elimination, we obtain
    \begin{equation}\label{predcutt}
        \H_\gamma[\vec s,t]\dash{\Omega_n+1}{\phi(\psi_n\gamma)(\psi_n\gamma)}\Gamma(\vec s)^{\Ln}.
    \end{equation}
    Since $\gamma<\omega^{\Omega_\omega+3\alpha}$ and $\phi(\psi_n\gamma)(\psi_n\gamma)<\Omega_\omega+3\alpha$, we can rewrite $\eqref{predcutt}$ as
    \begin{equation*}
        \H_\gamma[\vec s,t]\dash{\Omega_n+1}{\Omega_\omega+3\alpha}\Gamma(\vec s)^{\Ln},
    \end{equation*}
    as desired.
\end{proof}

$\setcounter{equation}{0}$
Now, we can embed {\sf W-KPl} derivations.
\begin{Theorem}[Embedding Theorem for {\sf W-KPl}]$\label{EmbedW}$
    Let $\Gamma(\vec a)$ be a finite set of $\Sigma^{\Omega_n}$-formulas. If ${\sf W-KPl}\vdash \Gamma(\vec a)$ then, for any $\vec s$ there is $\alpha<\epsilon_0$ such that
    \begin{equation*}
        \H_{\Omega_\omega+3\alpha}[\vec s]\dash{\Omega_{n}+1}{\Omega_{\omega}+\alpha}\Gamma(\vec s)^{\Ln}
    \end{equation*}
\end{Theorem}
\begin{proof}
    We suppose ${\sf W-KPl}\vdash\Gamma(\vec a)$. Then, by Lemma $\ref{induction}$, we have
    \begin{equation*}
        [\omega]\dash{0}{\alpha}\neg Lim,\neg A_1,\dots,\neg A_l,\Gamma(\vec a)
    \end{equation*}
    for some axioms $A_1,\dots, A_l$ of ${\sf W-KPl}$ different from mathematical induction and $(Lim)$. Now, by Lemma $\ref{limw}$, we get
    \begin{equation*}
        \H_{\Omega_\omega+3\alpha}[\vec s]\dash{\Omega_n+1}{\Omega_\omega+3\alpha}\Gamma(\vec s)^{\Ln}.
    \end{equation*}
\end{proof}

We will use the following operation to classify the total set-recursive-from-$\omega$ functions of {\sf W-KPl}.
\begin{definition}
    We define for every $\alpha<\epsilon_0$ the operation $\check G_\alpha(x)$ as
    \begin{equation*}
        \check G_\alpha(x)=L_{\psi_0(\Omega_{\omega}\cdot\omega^{\alpha+2})}(x)
    \end{equation*}
\end{definition}
Again, following Reading Convention \ref{read}, the definition of $\Omega_\omega$ and $\psi_0$ depend on $x$, and hence so does the definition of $\check G_\alpha$.\\
Finally, we classify the total set-recursive-from-$\omega$ functions of {\sf W-KPl}. The result is analogous to Theorems $\ref{main}$ and $\ref{totalr}$.
\begin{Theorem}$\label{totalw}$
    Let $f$ be a provably total set-recursive-from-$\omega$ function in {\sf W-KPl}. Then, there is $\alpha<\epsilon_0$ such that
    \begin{equation*}
        V\vDash\forall x(f(x)\in \check G_\alpha(x)).
    \end{equation*}
\end{Theorem}
\begin{proof}
    Let $A_f(\cdot,\cdot)$ be the $\Sigma$ formula that defines $f$ in any admissible set. In particular, we have that ${\sf W-KPl}\vdash Ad(b)\rightarrow[\forall x\in b\exists!y\in b\ A_f(x,y)^b]$.\\
    Now, we fix a set $X$ and we let $\theta$ be the set-theoretic rank of $X$, as in Reading Convention $\ref{read}$. By Theorem $\ref{EmbedW}$, we have
    \begin{equation*}
        \H_{\Omega_\omega+3\alpha}\dash{\Omega_{0}+1}{\Omega_{\omega}+\alpha}Ad(\mathbb L_{\Omega_0}(X))\rightarrow \forall x\in \mathbb L_{\Omega_0}(X)\exists!y\in \mathbb L_{\Omega_0}(X) A_f(x,y)^{\mathbb L_{\Omega_0}(X)}.
    \end{equation*}
    Following the same reasoning as in the proof of Theorem $\ref{main}$, we get
    \begin{equation}\label{wmain1}
        \H_{\Omega_\omega+3\alpha}\dash{\Omega_{0}+1}{\Omega_{\omega}+\alpha}\neg Ad(\mathbb L_{\Omega_0}(X)),\forall x\in \mathbb L_{\Omega_0}(X)\exists!y\in \mathbb L_{\Omega_0}(X) A_f(x,y)^{\mathbb L_{\Omega_0}(X)}.
    \end{equation}
    On the other hand, we have by Lemma $\ref{Ad}$ and Lemma $\ref{weak}$
    \begin{equation}\label{wmain2}
        \H_{\Omega_\omega+3\alpha}\dash{\Omega_{0}+1}{\Omega_{\omega}+\alpha}Ad(\mathbb L_{\Omega_0}(X)),\forall x\in \mathbb L_{\Omega_0}(X)\exists!y\in \mathbb L_{\Omega_0}(X) A_f(x,y)^{\mathbb L_{\Omega_0}(X)}.
    \end{equation}
    We apply $(Cut)$ to \eqref{wmain1} and \eqref{wmain2} to obtain
    \begin{equation*}
        \H_{\Omega_\omega+3\alpha}\dash{\Omega_{1}+1}{\Omega_{\omega}+\alpha+1}\forall x\in \mathbb L_{\Omega_0}(X)\exists!y\in \mathbb L_{\Omega_0}(X) A_f(x,y)^{\mathbb L_{\Omega_0}(X)}.
    \end{equation*}
    Since $\mathrm{rk}(Ad(\mathbb L_{\Omega_0}(X)))=\Omega_0+5$, we bound the complexity of the cuts by $\Omega_1+1$. We now apply Lemma $\ref{inversion}$ (Inversion) two times to obtain
    \begin{equation}\label{wmain3}
        \H_{\Omega_\omega+3\alpha}\dash{\Omega_{1}+1}{\Omega_{\omega}+\alpha+1}\exists y\in \mathbb L_{\Omega_0}(X) A_f(\overline X,y)^{\mathbb L_{\Omega_0}(X)}.
    \end{equation}
    We are now going to eliminate cuts from the derivation \eqref{wmain3}.
    By the Collapsing Theorem \ref{collapsing}, we get, as $\omega^{\Omega_\omega+\alpha+1}=\Omega_\omega\cdot\omega^{\alpha+1}$,
    \begin{equation*}
        \H_{\Omega_\omega\cdot\omega^{\alpha+1}}\dash{\psi_0(\Omega_\omega\cdot\omega^{\alpha+1})}{\psi_0(\Omega_\omega\cdot\omega^{\alpha+1})}\exists y\in\mathbb L_{\Omega_0}(X) A_f(\overline X,y)^{\mathbb L_{\Omega_0}(X)}.
    \end{equation*}
    By Theorem $\ref{cut elim}$ (Predicative Cut Elimination), we have
    \begin{equation*}
        \H_{\Omega_\omega\cdot\omega^{\alpha+1}}\dash 0{\beta}\exists y\in\mathbb L_{\Omega_0}(X)A_f(\overline X,y)^{\mathbb L_{\Omega_0}(X)}.
    \end{equation*}
    where $\beta=\phi(\psi_0(\Omega_\omega\cdot\omega^{\alpha+1}))(\psi_0(\Omega_\omega\cdot\omega^{\alpha+1}))$. By the Boundedness Lemma, we get 
    \begin{center}
        $\H_{\Omega_\omega\cdot\omega^{\alpha+1}}\dash0\beta \exists y\in\mathbb{L}_\beta (X) A_f(\overline X,y)^{\mathbb{L}_\beta (X)}$.
    \end{center}
    It follows that
    \begin{equation}\label{wmain4}
        L_\beta(X)\vDash \exists y A_f(X,y).
    \end{equation}
    Now, we observe that $\Omega_\omega\cdot \omega^{\alpha+1}\in B_0(\Omega_\omega\cdot\omega^{\alpha+2})$ because $\alpha<\epsilon_0<\Gamma_0<\psi_0(\Omega_\omega\cdot\omega^{\alpha+2})$. This means that $\psi_0(\Omega_\omega\cdot\omega^{\alpha+1})<\psi_0(\Omega_\omega\cdot\omega^{\alpha+2})$.\\
    Hence, by \eqref{wmain4}, we get
    \begin{equation*}
        f(X)\in L_\beta(X)\subseteq L_{\psi_0(\Omega_\omega\cdot\omega^{\alpha+2})}=\check G_\alpha(X),
    \end{equation*}
    as desired.
\end{proof}
$\setcounter{equation}{0}$
\section{The well-ordering and totality proofs}\label{SectWellOrderingProofs}
The proof of the classification Theorems $\ref{main}$, $\ref{totalr}$ and $\ref{totalw}$ cannot be carried out within {\sf KPl}, ${\sf KPl^r}$ and {\sf W-KPl} respectively, as they imply the consistency of the theories. However, they can ``almost'' be carried out internally, in the sense that one can prove the schemata establishing the conclusions of the theorems restricting the provably total functions according to the lengths of the proofs. 

In this section, we show that for every set $X$ with set-theoretic rank $\theta$ and for every $n<\omega$, the terms in $T(\theta)$ below $\psi_0(e_n)$ can be well-ordered within {\sf KPl}. Similarly, we show that, for every $n<\omega$, ${\sf KPl^r}$ proves the well-ordering of $\psi_0(\Omega_n)$ and that, for every $n<\omega$, {\sf W-KPl} proves that $\psi_0(\Omega_\omega\cdot\omega_n)$ can be well-ordered, where $\omega_0=1$ and $\omega_{n+1}=\omega^{\omega_n}$. Thus, for each $n<\omega$, the function $G_n$ (respectively, $\hat G_n$ and $\check G_n$) is provably total in {\sf KPl} (respectively, ${\sf KPl^r}$ and {\sf W-KPl}). From these facts, we then recover slight strengthenings of the classification theorems.

We fix a set $X$ and its set-theoretic rank $\theta$ and follow Reading Convention $\ref{read}$. The ordinals $e_n$ for $n<\omega$ are therefore fixed (see Definition $\ref{en}$). An element of $T(\theta)$ can be well-ordered if it belongs to the following set.
\begin{definition} We define the set $\acc$ of \textit{accessible} elements of $T(\theta)$.
    \begin{center}
        $\acc=\Big\{a\in T(\theta):a\prec \Omega_0\wedge\exists\alpha\exists f\big[\mathrm{dom}(f)=\alpha\wedge rg(f)=\{b:b\prec a\}\wedge\forall \delta,\gamma\in \mathrm{dom}(f)(\delta<\gamma\rightarrow f(\delta)\prec f(\gamma))\big]\Big\}$.
    \end{center}
\end{definition}
Given $a\in\acc$, there is a unique ordinal $\alpha$ isomorphic to $a$, the order-type of $\{b:b\prec a\}$. Moreover, the isomorphism between $\alpha$ and $a$ is unique. This ensures that the following notions are well-defined. 
\begin{definition}$\label{of}$
    Given $a\in\acc$, we write $o(a)$ and $f_a$ to denote the ordinal and the function such that $\mathrm{dom}(f_a)=o(a)$, $\mathrm{ran}(f_a)=\{b:b\preceq a\}$ and
    \begin{equation*}
        \forall\delta,\gamma<o(a)[\delta<\gamma\rightarrow f_a(\delta)\prec f_a(\gamma)].
    \end{equation*}
\end{definition}
Additionally, we notice that the formula $a\in\acc$ is $\Sigma$.
\subsection{The well-ordering and totality proof for \texorpdfstring{${\sf KPl^r}$}{KPlr}}
We start working in ${\sf KPl^r}$. We study some closure properties about $\acc$. The next lemma states the useful result that for any $X$ we can perform restricted induction in $\acc$ within ${\sf KPl^r}$ and full induction within {\sf KPl}.

\begin{Lemma}$\label{indacc}$
    Let $A(a)$ be any $\Delta_0$-formula in the language $\mathcal{L}'$. Then
    \begin{center}
        ${\sf KPl^r}\vdash\forall\theta\in\ord\big[\forall a\in \acc\big(\forall b\prec a\ A(b)\rightarrow A(a)\big)\rightarrow \forall a\in\acc\ A(a)\big]$.
    \end{center}
    Moreover, given any $\mathcal{L'}$-formula $B$ we have
    \begin{center}
        ${\sf KPl}\vdash\forall\theta\in\ord\big[\forall a\in \acc\big(\forall b\prec a\ B(b)\rightarrow  B(a)\big)\rightarrow \forall a\in\acc\ B(a)\big]$.
    \end{center}
\end{Lemma}
\begin{proof}
    We argue informally in ${\sf KPl^r}$. We fix $\theta$ and, for contradiction, we suppose that 
    \begin{equation}\label{pind1}
        \forall a\in\acc[\forall b\prec a\ A(b)\rightarrow A(a)]
    \end{equation}
    but $\neg A(a)$ for some $\Delta_0$-formula $A$ and some $a\in \acc$. We claim that there is an $a_0$ such that 
$\neg A(a_0)$ and $o(a_0)$ is minimal among $\{o(a):a\in \acc\wedge \neg A(a)\}$. To see this, let $a$ be such that $\neg A(a)$. 
Let $\rho$ be the restriction of the mapping $b \mapsto o(b)$ to terms $b \prec a$. Then $\rho$ is a total function on its domain and thus $\rho$ is $\Delta_1$. Moreover, its range is a subset of $\Omega_0$, so $\rho \in \Delta_0^{L_{\Omega_0}(X)}$. Thus, by the axiom of Induction restricted to $\Delta_0$-formulas, there is a least ordinal $\alpha_0$ such  that
\[\exists a_0 \in T(\theta)\, (\alpha_0 = o(a_0) \wedge \neg A(a_0)).\]
Thus, if $a_0$ is such that $\alpha_0 = o(a_0)$, then $a_0$ is
as claimed.

We have that if $b\prec a_0$ then $o(b)<o(a_0)$. Therefore, we obtain that $A(b)$ holds for any $b\prec a_0$, by the minimality of $a_0$. That is, we have
    \begin{center}
        $\forall b\prec a_0\ A(b)$.
    \end{center}
    By \eqref{pind1}, we get $A(a_0)$, a contradiction with $\neg A(a_0)$.
    
    For the unrestricted case in {\sf KPl}, the argument is the same but using full Class Induction.
\end{proof} 
$\setcounter{equation}{0}$
The next lemma gives two criteria to verify that some string belongs to $\acc$.
\begin{Lemma}$\label{close}$The following are provable in ${\sf KP}$. For any $X$ with set-theoretic rank $\theta\in\ord$, assuming that $L_{\Omega_0}(X)$ exists:
    \begin{enumerate}
        \item for any $a,b\in T(\theta)$, if $a\prec b\in\acc$ then $a\in \acc$,
        \item for any $b\in T(\theta)$, if for all $ a\prec b$ we have that $a\in \acc$, then $b\in \acc$.
    \end{enumerate}
\end{Lemma}
 \begin{proof}
     1. We suppose that $b\in\acc$ and $a\prec b$. Let $o(b)$ and $f_b$ be as in Definition $\ref{of}$. We define
     \begin{equation*}
         o(a)=\{\delta\in o(b):f_b(\delta)\preceq a\}\text{\ and\ }f_a=f_b\upharpoonright(o(a)+1).
     \end{equation*}
     Then $f_a:o(a)+1\rightarrow \{c:c\preceq a\}$ well-orders $\{c:c\preceq a\}$, showing that $a\in \acc$.\\
     \ \\
     2. We assume that $\forall a\prec b(a\in\acc)$, and want to show that $b\in\acc$. Using $\Delta_0$-Separation, we define the set $\{a\in T(\theta):a\prec b\}=\{a:a\prec b\}$. We define by $\Delta_0$-Collection the set
     \begin{equation*}
         F=\{f_a:a\prec b\}.
     \end{equation*}
     Finally, let $f_b=\bigcup_{a\prec b}F\cup\{(\sup(o(a)+1:a\prec b),b)\}$ by Union, with $\mathrm{dom}(f_b)=\sup(o(a)+1:a\prec b)=:o(b)$. Then, clearly $f_b$ well-orders $\{a:a\preceq b\}$, showing that $b\in\acc$.

 We are allowed to use $\Delta_0$-Collection because all the elements that appear in the proof of Lemma $\ref{close}$.2 belong to $L_{\Omega_0}(X)$, which is admissible. Indeed, we show the following. 
 \begin{claim}
     For any $a\in\acc$, the ordinal $o(a)$ and the function $f_a$ belong to  $L_{\Omega_0}(X)$.
 \end{claim}
 \renewcommand\qedsymbol{$\blacksquare$}
 \begin{proof}
     We proceed by induction in $\acc$, by Lemma $\ref{indacc}$, in the admissible set $L_{\Omega_0}(X)$. Let $a\in\acc$. We suppose $\forall b\prec a\big(o(b),f_b\in L_{\Omega_0}(X)\big)$ and we show that $o(a),f_a\in L_{\Omega_0}(X)$.\\
     We can form the set $\{o(b):b\prec a\}$ by $\Sigma$-Collection (in virtue of Theorem 4.4 in \cite{barwise}) within $L_{\Omega_0}(X)$ since the function $F:b\mapsto o(b)$ can be defined in a $\Sigma$ way as follows:
     \begin{center}
         $F$ is a function $\wedge \mathrm{dom}(F)=\{b:b\prec a\}\wedge \mathrm{ran}(F)\in\mathrm{Ord}\wedge \forall \big(b,o(b)\big)\in F\exists f\big[f\text{ is a function from $o(b)$ to $\{c:c\prec b\}$}\wedge \forall \delta,\gamma<o(b)\big(\delta<\gamma\rightarrow f(\delta)<f(\gamma)\big)\big]$. 
     \end{center}
     Then, we have that $o(a)=\sup\{o(b)+1:b\prec a\}$. Since each $o(b)\prec\Omega_0$, we also have $o(a)\prec\Omega_0$ and so $o(a)\in L_{\Omega_0}(X)$.\\
     Now, $f_a=\bigcup_{b\prec a}f_b\cup \{(o(a),a)\}$. By the induction hypothesis, each $f_b\in L_{\Omega_0}(X)$. We observe that the function $F:b\mapsto f_b$ can be defined by a $\Sigma$-formula, as follows.
     \begin{center}
    $F\text{ is a function }\wedge \mathrm{dom}(F)=\{b:b\prec a\}\wedge \mathrm{ran}(F)=\{\text{partial functions from $o(a)$ to $\{b:b\prec a\}$}\}\wedge \forall (b,f_b)\in F\big[\mathrm{dom}(f_b)=o(b)\wedge \mathrm{ran}(f_b)=\{c:c\prec b\}\wedge\forall \delta,\gamma<o(b)\big(\delta<\gamma\rightarrow f_b(\delta)<f_b(\gamma)\big)\big]\}$.
\end{center}
     Therefore, we can use again $\Sigma$-Collection in $L_{\Omega_0}(X)$ to form the set $\{f_b:b\prec a\}$. Thus, by Union, we form the set $\bigcup_{b\prec a}f_b$. Moreover, since $o(a)$ and $a$ belong to $L_{\Omega_0}(X)$, then also $\{(o(a),a)\}\in L_{\Omega_0}(X)$ by Pair. Hence, we obtain $f_a\in L_{\Omega_0}(X)$.
 \end{proof}
 \renewcommand\qedsymbol{$\square$}
 This concludes the proof.
\end{proof}
We notice that any model of {\sf KPl} that contains a set $X$ also contains $L_{\Omega_0}(X)$. Therefore, Lemma \ref{close} is also provable in {\sf KPl}.\\
We will show that $\acc$ is closed under addition and under the Veblen function. Even if the addition of two elements in $T(\theta)$ is not well-defined, we can take the convention to let $a+b$ denote the term that $a+b$ intends to mean, as follows.
 \begin{readingconvention}
     Let $a,b\in T(\theta)$. We write $a+b$ to denote the term in $T(\theta)$ that represents the ordinal obtained by the addition of the ordinal that the term $a$ represents with the ordinal that the term $b$ represents. Formally, we could define a $T(\theta)$-primitive recursive function $+:T(\theta)\times T(\theta)\rightarrow T(\theta)$ that identifies $a+b$ with the string in $T(\theta)$ representing $a+b$.\\
     Similarly, we write $\phi ab$ to denote the term in $T(\theta)$ representing the ordinal obtained by applying the Veblen function to the ordinal that the term $a$ represents and the ordinal that the term $b$ represents. Again, we can do this in a $T(\theta)$-primitive recursive way with a function $\phi:T(\theta)\times T(\theta)\rightarrow T(\theta)$.
 \end{readingconvention}
 \begin{Lemma}$\label{closeop}$
     The following is provable in ${\sf KP}$. For any $X$ with set-theoretic rank $\theta\in\ord$, assuming that $L_{\Omega_0}(X)$ exists:
     \begin{enumerate}
         \item $a,b\in\acc\rightarrow a+b\in\acc$,
         \item $a,b\in\acc\rightarrow\phi ab\in\acc$.
     \end{enumerate}
 \end{Lemma}
 \begin{proof}
     1. Given any $a\in \acc$, we show that $\forall b\in\acc (a+b\in\acc)$ by induction on $b$ using Lemma \ref{indacc}. So we assume $\forall c\prec b(a+c\in\acc)$ and we show $a+b\in\acc$. If we show $\forall d\prec a+b(d\in \acc)$ then it will follow that $a+b\in\acc$ by Lemma \ref{close}.2, and we will be done.\\
     So, let $d\prec a+b$. If $d\preceq a$, then $d\in\acc$ by Lemma $\ref{close}$.1 since $a\in\acc$. So we assume $a\prec d\prec a+b$. This means that $d=a+c$ for some $c\prec b$. Therefore, we get $d\in\acc$ by the induction hypothesis.\\
     \ \\
     2. We show that $\forall a\in\acc\forall b\in\acc (\phi ab\in\acc)$ by induction in $\acc$ using Lemma $\ref{indacc}$. So we assume
     \begin{equation}\label{funac1}
         \forall c\prec a[\forall b\in\acc(\phi cb\in\acc)]
     \end{equation}
     and we show $\forall b\in\acc (\phi ab\in\acc)$. We proceed by a subsidiary induction on $b$ by Lemma $\ref{indacc}$, and so we assume that
     \begin{equation}\label{funac2}
         \forall e\prec b(\phi ae\in\acc)
     \end{equation}
     to show $\phi ab\in \acc$. We will show that $\forall x\prec\phi ab(x\in\acc)$ and use Lemma $\ref{close}$.2 to conclude $\phi ab\in\acc$, finishing the proof.\\
     So let $x\prec \phi ab$. We show $x\in \acc$ by induction on the complexity $Cx$ of $x$.\\
     \ \\
     Case 1. We suppose that $x\equiv x_0+\cdots+x_n$ with $n>0$. Since $Cx_i<Cx$ for any $i\in\{1,\dots,n\}$, by the induction hypothesis we have $x_i\in \acc$ for all $i\in\{1,\dots,n\}$. Hence, by Item 1. of this lemma we obtain that $x\in\acc$.\\
     \ \\
     Case 2. We suppose that $x\equiv \phi yz$.\\
     Subcase 2.1. If $y\prec a$ and $z\prec\phi ab$ then $Cy<Cx$ and so $y\in\acc$ by the induction hypothesis, and therefore $x\in\acc$ by \eqref{funac1}.\\
     \ \\
     Subcase 2.2. If $y=a$, that is, if $x\equiv \phi az$, then necessarily $z\prec b$ (otherwise $x=\phi ab$). Then $x\in \acc$ by \eqref{funac2}.\\
     \ \\
     Subcase 2.3. If $a\prec y$ then $x=\phi yz\prec b$ (otherwise $\phi ab\prec x$) and so, by Lemma $\ref{close}$.1 we get $x\in\acc$.\\
     \ \\
     Case 3. We suppose that $x$ is strongly critical. If $x\preceq a$ then $x\in\acc$ by \eqref{funac1}.\\
     So we assume $a\prec x$. Then, we have $x\prec b$ (otherwise we would have $\phi ab\prec x$ as $x$ is strongly critical). Hence, by \eqref{funac2} we get $x\in\acc$.
 \end{proof}
 $\setcounter{equation}{0}$
The initial ordinals also belong to $\acc$ (except, obviously, for the ordinals $\Omega_n$, since $\acc\subseteq \Omega_0$). By Lemma $\ref{close}$.2 we vacuously have $0\in\acc$. We show that the strongly critical ordinals up to $\Gamma_\theta$ belong to $\acc$.
 \begin{Lemma}$\label{gammaacc}$
 $\forall\theta\in\ord\ \forall\beta\leq\theta\ \Gamma_\beta\in\acc.$

 \end{Lemma}
 \begin{proof}
     We proceed by induction on $\beta$. That is, we show $\forall \beta\leq\theta[\forall \alpha<\beta(\Gamma_\alpha\in\acc)\rightarrow \Gamma_\beta\in\acc]$. So let $\beta\leq\theta$ and we assume $\forall \alpha<\beta(\Gamma_\alpha\in\acc)$. To prove that $\Gamma_\beta\in\acc$, we show that if $b\prec\Gamma_\beta$ then $b\in\acc$ to apply Lemma $\ref{close}$.2 and end the proof.\\
     So let $b\prec\Gamma_\beta$. We show $b\in\acc$ by induction on the complexity $Cb$ of $b$.\\
     \ \\
     Case 1. If $b\equiv 0$ then Lemma $\ref{close}$.2 vacuously yields $0\in\acc$. If $b\equiv\Gamma_\delta$ for some $\delta<\beta$ then $b\in\acc$ by assumption. It cannot be the case that $b\equiv \Omega_n$ for some $n\leq\omega$ because $\Gamma_\theta\prec\Omega_n$ for any $n\leq\omega$.\\
     \ \\
     Case 2. If $b\equiv b_1+\cdots+b_n$ for some $n>0$ then by induction hypothesis $b_1,\dots,b_n\in\acc$ and by Lemma $\ref{closeop}$.1 we get $b\in \acc$. Similarly, if $b\equiv\phi cd$ then by induction hypothesis $c,d\in\acc$ and so $b\in\acc$ by Lemma $\ref{closeop}$.2.\\
     \ \\
     Case 3. It cannot be the case that $b\equiv \psi_nc$ since $\psi_nc>\Gamma_\theta$ for any $n<\omega$ and any $c$.
 \end{proof}
 $\setcounter{equation}{0}$
We will use the following result many times to show that some ordinal of the form $\psi_0(a)$ is accessible.
\begin{Lemma}\label{psi0acc}
    The following is provable in {\sf KP}. Let $X$ with set-theoretic rank $\theta\in\ord$. We assume that $L_{\Omega_0}(X)$ exists. Let $\psi_0(a)\in T(\theta)$ such that for any $\psi_0(b)\in T(\theta)$ with $b\prec a$ we have $\psi_0(b)\in\acc$. Then $\psi_0(a)\in\acc$.
\end{Lemma}
\begin{proof}
    If $a=0$ then obviously $\psi_0(0)\in\acc$. So we assume $\alpha\succ0$. Then, there is some $b\in T(\theta)$ such that $\psi_0(b)\in T(\theta)$ and $\psi_0(a)=\big(\psi_0(b)\big)^\Gamma$ by Lemma $\ref{311}$. By hypothesis, we have $\psi_0(b)\in\acc$. Let $\psi_0(b)<\xi<\psi_0(a)$. Then, as $\psi_0(a)$ is the first strongly critical above $\psi_0(b)$, we have that $\xi$ is written from lower ordinals using $+$ and $\phi$. Therefore, an easy induction using Lemma $\ref{closeop}$ gives $\xi\in\acc$. Hence, we conclude by Lemma $\ref{close}$.2 that $\psi_0(a)\in \acc$.
\end{proof}
We now give the well-ordering proof for ${\sf KPl^r}$. We want to show that ${\sf KPl^r}$ proves that $\psi_0(\Omega_n)$ is accessible for any $n<\omega$ by first showing that $\psi_0(b)$ is accessible for any $b\prec\Omega_n$. Unfortunately, to conclude $\psi_0(\Omega_n)\in\acc$ we should use Lemma $\ref{psi0acc}$, which needs Class Induction. To avoid this problem and carry the proof within ${\sf KPl^r}$, we redefine the terms we employ using only finitely-many $\Omega_i$'s. So we fix $X$ with set-theoretic rank $\theta$ and we fix $n<\omega$.
\begin{definition}
    We define $\hat T^n(\theta)$ as the closure of $\{0\}\cup\{\Gamma_\beta:\beta\leq\theta\}\cup\{\Omega_i:i\leq n\}$ under $+,\phi,\hat\psi^n_i$ for $i\leq n$, where $\hat\psi^n_i$ is defined as $\psi_i$ but using as a closure set $\hat B^n_i(\alpha)$, defined from initial ordinals $\{0\}\cup\{\Gamma_\beta:\beta\leq\theta\}\cup\{\Omega_i:i\leq n\}$ as in Definition \ref{psi}.
\end{definition}
The key idea is that, in $T(\theta)$, ordinals larger that $\Omega_n$ do not give rise to terms below $\psi_0(\Omega_n)$, as we will show in Claim \ref{claimisomterms}. From this fact follows that the terms in $\hat T^n(\theta)\upharpoonright\hat\psi_0^n(\Omega_n)$ are order-isomorphic to the terms in $T(\theta)\upharpoonright\psi_0(\Omega_n)$, as we show next.
\begin{Lemma}\label{isomterms}Let $\beta\in T(\theta)$. Then,
    $\hat T^n(\theta)\upharpoonright\hat\psi_0^n(\Omega_n)\cong T(\theta)\upharpoonright\psi_0(\Omega_n)$.
\end{Lemma}
\begin{proof}
    First, we prove that no term in $T(\theta)$ below $\psi_0(\Omega_n)$ has any ordinal above $\Omega_\omega\cdot\beta+\Omega_n$ in its construction.
    \begin{claim}\label{claimisomterms}
        Let $a\in T(\theta)$ such that $a\prec \psi_0(\Omega_n)$. Then, no ordinal above $\Omega_n$ occurs in the construction of $a$.
    \end{claim}
    \renewcommand\qedsymbol{$\blacksquare$}
    \begin{proof}
        Case 1. We assume that $a$ is an initial ordinal. Since $a\prec\psi_0(\Omega_n)\prec\Omega_0$, we have that $a$ is either $0$ or some $\Gamma_\delta$ for $\delta\leq\theta$ and so $a\prec\Omega_n$.\\
    Case 2. If $a=b+c$ then $b,c\prec\psi_0(\Omega_n)$ as $\psi_0(\Omega_n)$ is additive principal, and by induction hypothesis we get the result.\\
    Case 3. If $a=\phi bc$ then $b,c\prec\psi_0(\Omega_n)$ as $\psi_0(\Omega_n)$ is strongly critical and by induction hypothesis we get the result.\\
    Case 4. We assume that $a=\psi_0(b)$. We cannot use the induction hypothesis on $b$ because $b$ may be above $\psi_0(\Omega_n)$. But the key here is that $b$ is the argument of the $\psi_0$ function. We note that $b\prec\Omega_n$ (we have that $a=\psi_0(b)\prec\psi_0(\Omega_n)$ and so $b\prec\Omega_n$). We show by induction on $b$ that $b$ does not contain any ordinal above $\Omega_{n}$.\\
    \ \\
    Case 4.1. We assume that $b$ is an initial ordinal. If $b\succ \Omega_n$ then $a\prec\psi_0(\Omega_n)\prec\psi_0(b)=a$, a contradiction.\\
    Case 4.2. We assume $b=b_1+b_2$. Then, by induction hypothesis we get the result.\\
    Case 4.3. We assume $b=\phi b_1b_2$. Then, by induction hypothesis we get the result.\\
    Case 4.4. We assume $b=\psi_i(b_1)$ for some $i>n$. Since terms are written in normal form, by Definition \ref{normal form} we have $\psi_i(b_1)\in B_0\big(\psi_i(b_1)\big)$, as $\psi_i(b_1)$ is taken as an argument of the function $\psi_0$. However, the only way to obtain $\psi_i(b_1)$ in $B_0\big(\psi_i(b_1)\big)$ is to apply $\psi_i$ to $b_1$. So, by Definition $\ref{psi}$, we need $b_1\prec\psi_i(b_1)=b\prec\Omega_n$. Hence, $b_1\prec\Omega_n$ and by the induction hypothesis we have that no ordinal above $\Omega_{n}$ appears in $b_1$. Moreover, since $\psi_i(b_1)=b\prec\Omega_n$, we also get that $i\leq n$ because $\psi_i$ collapses $b_1$ below $\Omega_n$.
    \end{proof}
    In particular, we have shown that all the $\psi_i$ that appear in the construction of $a\prec \psi_0(\Omega_n)$ are such that $i\leq n$.
    Hence, the isomorphism is obtained naturally as follows.\\
    \renewcommand\qedsymbol{$\square$}
    We define $f:T(\theta)\upharpoonright\psi_0(\Omega_\omega\cdot\beta+\Omega_n)\longrightarrow\hat T^n(\theta)\upharpoonright\hat \psi^n_0(\Omega_\omega\cdot\beta+\Omega_n)$ as
    \begin{itemize}
        \item $f(0)=0$, $f(\Gamma_\beta)=\Gamma_\beta$ for all $\beta\leq\theta$ and $f(\Omega_i)=\Omega_i$ for all $i\leq n$;
        \item $f(a+b)=f(a)+f(b)$;
        \item $f(\phi ab)=\phi f(a)f(b)$;
        \item $f(\psi_i(a))=\hat\psi_i^n(f(a))$ for all $i\leq n$.
    \end{itemize}
\end{proof}

We obtain a version of Lemma $\ref{psi0acc}$ for $\hat T^n(\theta)$. We notice that the next lemma makes sense since the construction of $\hat T^n(\theta)$ uses ordinals below $\Omega_{n+1}$ and can be carried out in the model $L_{\Omega_{n+1}}(X)$, where $X$ has rank $\theta$.
\begin{Lemma}\label{rpsi0acc}
    Fix $X$ with rank $\theta$. Then,
    \begin{equation*}
        L_{\Omega_{n+1}}(X)\vDash\forall \hat\psi_0^n(a)\in \hat T^n(\theta) \big[\big(\forall b\prec a\ \hat\psi^n_0(b)\in\acc\big)\rightarrow\hat\psi_0^n(a)\in \acc\big].
    \end{equation*}
\end{Lemma}
\begin{proof}
    Fixing $X$ with rank $\theta$, the set $\hat T^n(\theta)$ is fully constructible in $L_{\Omega_{n+1}}(X)$ as it needs ordinals strictly below $\Omega_{n+1}$. Since $L_{\Omega_{n+1}}(X)$ is a model of {\sf KP}, where $L_{\Omega_0}(X)$ exists, we can repeat the arguments in Lemma \ref{psi0acc} inside of $L_{\Omega_{n+1}}(X)$.
\end{proof}

\begin{Theorem}\label{hatacc}
    For every $n<\omega$,
    \begin{equation*}
        {\sf KPl^r}\vdash\forall \theta\in \mathrm{Ord}\ \hat\psi^n_0(\Omega_n)\in \acc.
    \end{equation*}
\end{Theorem}
\begin{proof}
    We fix $n<\omega$ and we show the following claim.
    \begin{claim}$\label{clLnacc}$Let $X$ with rank $\theta$. Then,
        \begin{equation*}
            L_{\Omega_{n+1}}(X)\vDash \forall a\in \hat T^n(\theta)[\hat\psi^n_0(a)\in \hat T^n(\theta)\rightarrow \hat \psi^n_0(a)\in \acc].
        \end{equation*}
    \end{claim}
    \renewcommand\qedsymbol{$\blacksquare$}
    \begin{proof}
    We suppose the claim is false. Then, as $L_{\Omega_{n+1}}(X)$ is admissible, there is a minimum $a$ such that $L_{\Omega_{n+1}}\vDash\psi_0(a)\in T(\theta)$ but $L_{\Omega_{n+1}}(X)\nvDash \psi_0(a)\in\acc$. Therefore,  for all $b\prec a$ with $L_{\Omega_{n+1}}(X)\vDash \psi_0(b)\in T(\theta)$ we have
    \begin{equation*}
        L_{\Omega_{n+1}}(X)\vDash \psi_0(b).
    \end{equation*}
    By Lemma \ref{rpsi0acc}, we get $L_{\Omega_{n+1}}(X)\vDash \psi_0(a)\in\acc$, a contradiction. So we conclude that Claim $\ref{clLnacc}$ holds.
    \end{proof}
    \renewcommand\qedsymbol{$\square$}
    Now, by Claim \ref{clLnacc} and since the formula $\forall a<\Omega_{n+1}$ is $\Sigma$, we get
    \begin{equation*}
        {\sf KPl^r}\vdash\forall \theta\in\mathrm{Ord}\  \forall a\prec \Omega_{n+1}\ \hat\psi^n_0(a)\in\acc.
    \end{equation*}
    In particular, we have shown that ${\sf KPl^r}\vdash \forall\theta\in\mathrm{Ord}\ \hat\psi^n_0(\Omega_n)\in\acc$.
\end{proof}
By Lemma \ref{isomterms}, we immediately obtain that $\psi_0(\Omega_n)$ is provably wellfounded in ${\sf KPl^r}$ from Theorem \ref{hatacc}.
\begin{Theorem}\label{wellordkplr}
    For any $n<\omega$,
    \begin{equation*}
        {\sf KPl^r}\vdash \forall \theta\in \mathrm{Ord}\ \psi_0(\Omega_n)\in\acc.
    \end{equation*}
\end{Theorem}
We can finally prove the totality of $\hat G_n$ in ${\sf KPl^r}$.
\begin{Theorem}
    For every $n<\omega$ and for any set $X$, $\hat G_n$ is a provably total set-recursive-from-$\omega$ function in ${\sf KPl^r}$.
\end{Theorem}
\begin{proof}
    For any set $X$, we can specify the rank $\theta$ of $X$ in a $\Delta_0$-way, as in \cite{barwise}. Now, by Theorem \ref{wellordkplr}, we can find in ${\sf KPl^r}$ an ordinal that is order-isomorphic to $\psi_0(\Omega_n)$ and define $L_{\psi_0(\Omega_n)}(X)$ by $\Sigma$-recursion, which shows that $\hat G_n$ is set-recursive-from-$\omega$. Moreover, the set $L_{\psi_0(\Omega_n)}(X)$ is constructible in $L_{\Omega_0}(X)$, which shows that $\hat G_n(X)$ is constructible in any admissible set containing $X$.
\end{proof}

\subsection{The well-ordering and totality proof for {\sf W-KPl}}
Now, we prove that {\sf W-KPl} shows the well-ordering of $\psi_0(\Omega_\omega\cdot \omega_n)$ for every $n<\omega$. We will use greek letters $\alpha,\beta,\dots$ to denote both real ordinals and elements in $T(\theta)$ and the relation symbol $<$ to denote both the real ordinal order and the order $\prec$ on $T(\theta)$.\\
We begin by stating two results that characterize {\sf W-KPl} in the sense that ${\sf KPl^r}$ does not prove them.
\begin{Lemma}\label{wkpllon}
    \begin{equation*}
        {\sf W-KPl}\vdash \forall X\forall n<\omega\ L_{\Omega_n}(X)\text{ exists}.
    \end{equation*}
    Moreover, 
    \begin{equation*}
        {\sf W-KPl}\vdash \forall \gamma\in T(\theta)[\gamma<\Omega_\omega\rightarrow \exists n<\omega(\gamma<\Omega_n)].
    \end{equation*}
\end{Lemma}
Now, given some $\alpha$, showing that $\psi_0(\Omega_\omega\cdot \alpha)\in\acc$ amounts to showing that $\alpha$ belongs to the following class.
\begin{definition}
    $X_\theta=\{\alpha\in T(\theta)\ | \psi_0(\Omega_\omega\cdot\alpha)\in T(\theta)\rightarrow \psi_0(\Omega_\omega\cdot\alpha)\in \acc\}$.
\end{definition}
So, our objective is to show $\omega_n\in X_\theta$ for all $n<\omega$. We have to do some work first. For a definable class $U$, we define the formula $\prog(U)$, that stands for $U$ is \textit{progressive}. We notice that $\prog(U)$ is the hypothesis of the induction principle.
\begin{definition}
    Let $U$ be a definable class. We define the formula $\prog(U)$ as follows:
    \begin{equation*}
        \prog(U)=\forall \alpha[\forall \beta<\alpha(\beta\in U)\rightarrow \alpha\in U].
    \end{equation*}
    We also define the jump of $U$ as
    $\mathcal{J}(U)=\{\beta\in T(\theta)\ |\forall \alpha(\alpha\subseteq U\rightarrow \alpha+\omega^\beta\subseteq U)\}$.
\end{definition}

The sketch of the proof of $\omega_n\in X_{\theta}$ is the following. First, we will show that $X_{\theta}$ is progressive. Afterwards, we will prove that whenever a class $U$ is progressive, then the jump $\mathcal{J}(U)$ is also progressive. We will use this last result to show that if $U$ is progressive, then $\omega_n\in U$ for all $n<\omega$ (we will use $\mathcal{J}(U)$ in this proof). We will conclude that $\prog(X_{\theta})\rightarrow \omega_n\in X_{\theta}$, showing that $\omega_n\in X_{\theta}$. So we start by proving that $X_{\theta}$ is progressive. First, we need the following result.
\begin{Lemma}\label{nextpsi}
    Let $n<\omega$ such that $L_{\Omega_n}(X)\vDash \psi_0(d)\in\acc$. Then, if $\psi_0(d+1)\in T(\theta)$, we also have $L_{\Omega_n}(X)\vDash \psi_0(d+1)\in \acc$.
\end{Lemma}

\begin{proof}
    We show that, in $L_{\Omega_n}(X)$ where $\psi_0(d)\in\acc$, we have that $\acc\upharpoonright\psi_0(d+1)$ is closed under $+$ and $\phi$.\\
    Let $a,b\in T(\theta)$ such that $a,b\prec\psi_0(d+1)$ and $L_{\Omega_n(X)}\vDash a,b\in\acc$. This means that $a,b,o(a),o(b),f_a,f_b\in L_{\Omega_n}(X)$. Therefore, we have $o(a)+o(b),\phi\big(o(a)\big)\big(o(b)\big)\in L_{\Omega_n}(X)$, and, as $\psi_0(d+1)$ is strongly critical, we get $o(a)+o(b),\phi\big(o(a)\big)\big(o(b)\big)<\psi_0(d+1)$. In $L_{\Omega_n}(X)$, we show that $a+b\in\acc$ and $\phi ab\in\acc$.\\
    \ \\
    We define $f_{a+b}:\{c\in T(\theta):c\prec a+b\}\longrightarrow o(a)+o(b)$ as
    \begin{equation*}
        f_{a+b}(c)=\begin{cases}
      f_a(c) & \text{if  $c\prec a$,}\\
      o(a)+f_b(b_1) & \text{if $c=a+b_1$ with $b_1\prec b$.}
    \end{cases}   
    \end{equation*}
    We get $f_{a+b}\in L_{\Omega_n(X)}$. Clearly, $f_{a+b}$ is an isomorphism, showing that $o(a+b)=o(a)+o(b)$ and $L_{\Omega_n}(X)\vDash a+b\in\acc$.\\
    \ \\
    We obtain analogously that $L_{\Omega_n}(X)\vDash\phi ab\in \acc$ by defining $f_{\phi ab}:\{c\in T(\theta):c\prec \phi ab\}\longrightarrow \phi\big(o(a)\big)\big(o(b)\big)$ as follows.\\
    If $c\prec b$ we put $f_{\phi ab}(c)=f_b(c)$.\\
    If $c=c_0+\cdots+c_m$ we put $f_{\phi ab}(c_0)+\cdots f_{\phi ab}(c_m)$ (by induction).\\
    \ \\
    If $c=\phi c_0c_1$. We have three cases.\\
    If $c_0=a$, then $c_1\prec b$ and we put $f_{\phi ab}(c)=\phi\big(o(a)\big)\big(f_b(c_1)\big)$.\\
    If $c_0\succ a$ then $c\prec b$ (otherwise we would have $c=\phi c_0c_1\succ\phi ab$), and we have $f_{\phi ab}(c)=f_b(c)$.\\
    Finally, if $c_0\prec a$ then, as $c_1<c$, we put $f_{\phi ab}(c)=\phi\big(f_{\phi ab}(c_0)\big)\big(f_{\phi ab}(c_1)\big)$ (by induction).\\
    \ \\
    The key idea is that in the construction of every $c\in[\psi_0(d),\psi_0(d+1))$ we only use $+$ and $\phi$, applied to lower ordinals. So now, by Lemma \ref{close}, we get that $\psi_0(d+1)\in\acc$.
    \end{proof}

\begin{Lemma}$\label{progx}$
    ${\sf W-KPl}\vdash\forall\theta\in\ord\ \prog(X_\theta)$.
\end{Lemma}
\begin{proof}
    We argue informally in {\sf W-KPl}. We suppose that $\forall \beta<\alpha(\beta\in X_\theta)$. That is, we assume that 
    \begin{equation}\label{wprogx1}
        \forall \beta<\alpha[\psi_0(\Omega_\omega\cdot\beta)\in T(\theta)\rightarrow\psi_0(\Omega_\omega\cdot\beta)\in \acc]
    \end{equation}
    and show $\alpha\in X_\theta$. So, we assume $\psi_0(\Omega_\omega\cdot\alpha)\in T(\theta)$ and we show $\psi_0(\Omega_\omega\cdot\alpha)\in \acc$. First, we suppose that $\alpha$ is a successor: $\alpha=\beta+1$. Then
    \begin{equation*}
        \psi_0(\Omega_\omega\cdot\alpha)=\psi_0(\Omega_\omega\cdot \beta+\Omega_\omega).
    \end{equation*}
    \renewcommand\qedsymbol{$\blacksquare$}
    Now, we prove the following claim
    \begin{claim}\label{lon}
        ${\sf W-KPl}\vdash \forall \gamma\in T(\theta)\ [\psi_0(\Omega_\omega\cdot\beta+\gamma)\in T(\theta)\rightarrow \psi_0(\Omega_\omega\cdot \beta+\gamma)\in \acc]$.
    \end{claim}
    \begin{proof}

        We work informally in {\sf W-KPl}. Towards a contradiction, we assume the claim does not hold. Then, the set $A=\{n<\omega:\exists\gamma<\Omega_n[\psi_0(\Omega_\omega\cdot\beta+\gamma)\notin\acc]\}$ is not empty. By natural induction, which we have access to in this system, there is a least $n\in A$. This means that, for some $\gamma<\Omega_n$,
        \begin{equation}\label{n+1}
            L_{\Omega_{n+1}}(X)\vDash\psi_0(\Omega_\omega\cdot\beta+\gamma)\notin \acc.
        \end{equation}
        Working in $L_{\Omega_n}(X)$, which is admissible and thus satisfies full Induction, let $\gamma$ be minimal satisfying \eqref{n+1}. But then for all $\delta<\gamma$ we have that $\psi_0(\Omega_\omega\cdot\beta+\delta)\in\acc$. But by Lemma \ref{psi0acc} (which is provable in {\sf KP}), we obtain that $\psi_0(\Omega_\omega\cdot\beta+\gamma)\in\acc$, a contradiction.
     \end{proof}
    \renewcommand\qedsymbol{$\square$}
    By Claim $\ref{lon}$ and using natural induction within {\sf W-KPl}, we get
    \begin{equation*}
        {\sf W-KPl}\vdash \forall n<\omega[\psi_0(\Omega_\omega\cdot \beta+\Omega_n)\in \acc],
    \end{equation*}
    which leads to
    \begin{equation*}
        {\sf W-KPl}\vdash \psi_0(\Omega_\omega\cdot \beta+\Omega_\omega)\in\acc.
    \end{equation*}
    Now, we assume that $\alpha$ is a limit and argue informally within {\sf W-KPl}. Let $\delta<\psi_0(\Omega_\omega\cdot\alpha)$. Then, there is some $\beta<\alpha$ such that $\delta<\psi_0(\Omega_\omega\cdot \beta)\in \acc$. By Lemma $\ref{close}$.1, we get $\delta \in \acc$. Hence, by Lemma $\ref{close}$.2, we obtain $\psi_0(\Omega_\omega\cdot\alpha)\in\acc$.
\end{proof}
$\setcounter{equation}{0}$
Now, we show that $\mathcal{J}(U)$ is progressive whenever $U$ is progressive.
\begin{Lemma}$\label{progju}$
    ${\sf W-KPl}\vdash \prog(U)\rightarrow \prog\big(\mathcal{J}(U)\big)$.
\end{Lemma}
\begin{proof}
    We assume 
    \begin{equation}\label{wprog1}
        \prog(U)
    \end{equation}
    and we want to show $\prog\big(\mathcal{J}(U)\big)$. So suppose that
    \begin{equation}\label{wprog2}
        \forall \beta<\alpha\big(\beta\in \mathcal{J}(U)\big).
    \end{equation}
    Our aim is to prove that $\alpha\in\mathcal{J}(U)$. This means that we need to show
    \begin{equation*}
        \forall \xi[\xi\subseteq U\rightarrow (\xi+\omega^\alpha)\subseteq U].
    \end{equation*}
    So, let $\xi$ such that
    \begin{equation}\label{wprog3}
        \xi \subseteq U.
    \end{equation}
    Let $\delta\in \xi+\omega^\alpha$. We show $\delta\in U$ by splitting cases based on the relation between $\delta$ and $\xi$.\\
    \ \\
    Case 1. If $\delta<\xi$ then $\delta\in U$ by \eqref{wprog3}.\\
    \ \\
    Case 2. We suppose $\delta=\xi$. Then, by \eqref{wprog3}, we have $\forall \gamma<\delta(\gamma\in U)$. But then, by \eqref{wprog1}, we get $\delta\in U$.\\
    \ \\
    Case 3. We suppose $\xi<\delta$. Therefore, we have $\xi<\delta< \xi+\omega^\alpha$. Thus, there are $\delta_1,\dots,\delta_n$ such that $\delta_n\leq\dots\leq \delta_1\prec \alpha$ and
    \begin{equation*}
        \delta=\xi+\omega^{\delta_1}+\cdots+\omega^{\delta_n}
    \end{equation*}
    By \eqref{wprog2}, we have $\delta_1\in \mathcal{J}(U)$. This means that $\xi+\omega^{\delta_1}\subseteq U$. By the same reasoning $n-1$ times we get
    \begin{equation*}
        \delta=\xi+\omega^{\delta_1}+\cdots\omega^{\delta_{n}}\subseteq U.
    \end{equation*}
    Now, by \eqref{wprog1}, we obtain that $\delta\in U$.
\end{proof}
$\setcounter{equation}{0}$
Finally, we show that if $U$ is progressive then $\omega_n\in U$ for any $n<\omega$ by induction on $n$. Simultaneously, we show $\prog(U)\rightarrow \omega_n\subseteq U$ as we need this result to show the inductive step.
\begin{Lemma}$\label{omeganu}$
    For any $n<\omega$ and for any definable class $U$
    \begin{equation*}
        {\sf W-KPl}\vdash \prog(U)\rightarrow \omega_n\subseteq U\wedge \omega_n\in U.
    \end{equation*}
\end{Lemma}
\begin{proof}
    We proceed by induction on $n$.\\
    Case $n=0$. By definition, we have $\omega_0=1$. But, as $\emptyset\subseteq U$, we have $\emptyset\in U$ since $\prog(U)$ holds and so $1\subseteq U$, giving $1\in U$.\\
    \ \\
    Inductive case. We suppose that for any definable class $U$ the result holds for $\omega_n$. This means that it holds for $\mathcal{J}(U)$, and so
    \begin{equation}\label{wnu1}
        \prog(\mathcal{J}(U))\rightarrow \omega_n\subseteq\mathcal{J}(U)\wedge \omega_n\in\mathcal{J}(U).
    \end{equation}
    But since we have $\prog(U)$, by Lemma $\ref{progju}$ we also have $\prog\big(\mathcal{J}(U)\big)$, and therefore we get by \eqref{wnu1}
    \begin{equation*}
        \omega_n\subseteq\mathcal{J}(U)\wedge \omega_n\in\mathcal{J}(U).
    \end{equation*}
    Since $\omega_n\in\mathcal{J}(U)$, we obtain by definition
    \begin{equation*}
        \forall \xi(\xi\subseteq U\rightarrow \xi+\omega^{\omega_n}\subseteq U).
    \end{equation*}
    Taking $\xi=0$ gives $\omega^{\omega_n}=\omega_{n+1}\subseteq U$. Since we assume $\prog(U)$, we also obtain $\omega_{n+1}\in U$.
\end{proof}
$\setcounter{equation}{0}$
Finally, we give the well-ordering proof for {\sf W-KPl}.
\begin{Theorem}\label{wellordw}
    For every $n<\omega$,
    \begin{equation*}
        {\sf W-KPl}\vdash\forall\theta\in\ord\ \psi_0(\Omega_\omega\cdot \omega_n)\in \acc.
    \end{equation*}
\end{Theorem}
\begin{proof}
    By Lemma $\ref{progx}$, we have $\prog(X_\theta)$. Applying Lemma $\ref{omeganu}$ to $X_\theta$, we get $\omega_n\in X_\theta$ for any $n<\omega$. But this means
    \begin{equation*}
        \psi_0(\Omega_\omega\cdot\omega_n)\in\acc,
    \end{equation*}
as desired.
\end{proof}

We can finally prove the totality of $\check G_\alpha$ in {\sf W-KPl}.
\begin{Theorem}
    For every $\alpha<\epsilon_0$ and for any set $X$, $\check G_\alpha$ is a provably total set-recursive-from-$\omega$ function in {\sf W-KPl}.
\end{Theorem}
\begin{proof}
    For any set $X$, we can specify the rank $\theta$ of $X$ in a $\Delta_0$-way, as in \cite{barwise}. Now, by Theorem \ref{wellordw}, we can find in {\sf W-KPl} an ordinal that is order-isomorphic to $\psi_0(\Omega_\omega\cdot\omega^{\alpha+2})$ and define $L_{\psi_0(\Omega_\omega\cdot\omega^{\alpha+2})}(X)$ by $\Sigma$-recursion, which shows that $\check G_n$ is set-recursive-from-$\omega$. Moreover, the set $L_{\psi_0(\Omega_\omega\cdot\omega^{\alpha+2})}(X)$ is constructible in $L_{\Omega_0}(X)$, which shows that $\check G_n(X)$ is constructible in any admissible set containing $X$.
\end{proof}

\begin{remark}
The class $X_\theta$ is $\Sigma_1$-definable  (uniformly on $X$). One can check by direct computation that $\mathcal{J}^n(X_\theta)$ is $\Pi_{n+1}$-definable. 
\end{remark}
\subsection{The well-ordering and totality proof for {\sf KPl}}
We will show now that {\sf KPl} proves the well-ordering of $\psi_0(e_n)$ for any $n<\omega$. We will proceed in a similar way to the previous subsection. To show $\psi_0(e_n)\in \acc$, we will show that $e_n$ belongs to the following class.
\begin{definition}
    $Y_\theta=\{\alpha\in T(\theta)|\psi_0(\alpha)\in T(\theta)\rightarrow \psi_0(\alpha)\in \acc\}$.
\end{definition}
The following result is analogous to Lemma $\ref{progx}$.
\begin{Lemma}\label{progy}
    ${\sf KPl}\vdash\forall\theta\in\ord\ \prog(Y_\theta)$.
\end{Lemma}
Now, given a definable class $U$ such that $e_0\in U$ and $U$ is progressive, we show that $e_n\in U$ for every $n<\omega$.
\begin{Lemma}\label{enu}
    For any $n<\omega$ and any definable class $U$
    \begin{equation*}
        {\sf KPl}\vdash \forall\theta\in\ord[\prog(U)\rightarrow e_n\in U\wedge e_n\subseteq U].
    \end{equation*}
\end{Lemma}
\begin{proof}
    We fix $\theta$ and we proceed by induction on $n$.\\
    Case $n=0$. We have $e_0=\Omega_\omega+1$ by definition. We need the following claim.
    \renewcommand\qedsymbol{$\blacksquare$}
    \begin{claim}\label{psi0omega1}
        $\psi_0(\Omega_\omega+1)\in\acc$
    \end{claim}
    \begin{proof}
        We argue informally in ${\sf KPl}$. We first show $\psi_0(\Omega_\omega)\in\acc$. So let $\alpha<\psi_0(\Omega_\omega)$. Then, $\alpha<\psi_0(\Omega_n)$ for some $n<\omega$. But by Theorem $\ref{wellordkplr}$, we have $\psi_0(\Omega_n)\in \acc$. Hence, $\alpha\in \acc$ by Lemma $\ref{close}$.1. Therefore, by Lemma $\ref{close}$.2, we obtain $\psi_0(\Omega_\omega)\in\acc$.\\
        Now, we have $\psi_0(\Omega_\omega+1)=\big(\psi_0(\Omega_\omega)\big)^\Gamma$ by Lemma $\ref{311}$.7, as $\Omega_\omega\in B_0(\Omega_\omega)$ by Definition $\ref{psi}$. Therefore, by Lemma $\ref{psi0acc}$, we get $\psi_0(\Omega_\omega+1)\in\acc$.
    \end{proof}
    \renewcommand\qedsymbol{$\square$}
    Now, we show that 
    \begin{equation}\label{enu1}
      \forall \alpha\in\acc[\alpha=\psi_0(\beta)\rightarrow \beta\in U]
    \end{equation}
    by induction in $\acc$, owing to Lemma $\ref{indacc}$. We suppose that $\forall \delta<\alpha(\delta=\psi_0(\gamma)\rightarrow \gamma\in U)$, and we suppose that $\alpha=\psi_0(\beta)$. Then $\psi_0(\beta)$ is the first strongly critical ordinal after $\psi_0(\gamma)$, for some $\gamma<\beta$. Therefore, by the induction hypothesis, we have $\psi_0(\gamma)\in\acc$. Lemma $\ref{psi0acc}$ yields $\psi_0(\beta)\in \acc$.\\
    By Claim $\ref{psi0omega1}$, we have $\psi_0(\Omega_\omega+1)\in\acc$. Hence, by \eqref{enu1} with $\alpha=\psi_0(\Omega_\omega+1)$, we get $e_0=\Omega_\omega+1\in U$.\\
    \ \\
    Inductive case. We suppose that for any definable class $U$, the result holds for $e_n$. Then, it holds for $\mathcal{J}(U)$. But by Lemma $\ref{progju}$, we have $\prog\big(\mathcal{J}(U)\big)$, and so we get
    \begin{equation*}
        e_n\subseteq \mathcal{J}(U)\wedge e_n\in\mathcal{J}(U).
    \end{equation*}
    Since $e_n\in\mathcal{J}(U)$, we obtain by definition 
    \begin{equation*}
        \forall \xi(\xi\subseteq U\rightarrow\xi+\omega^{e_n}\subseteq U).
    \end{equation*}
    Hence, taking $\xi=0$, we obtain $e_{n+1}=\omega^{e_n}\subseteq U$. Finally, as $U$ is progressive, we also obtain $e_n\in U$.
\end{proof}
$\setcounter{equation}{0}$
\begin{Theorem}\label{TheoremWOKPl}
    For every $n<\omega$
    \begin{equation*}
        {\sf KPl}\vdash\forall\theta\in\ord\ \psi_0(e_n)\in\acc.
    \end{equation*}
\end{Theorem}
\begin{proof}
    By Lemma $\ref{progy}$, we have $\prog(Y_\theta)$. Applying Lemma $\ref{enu}$ to $Y_\theta$, we get $e_n\in Y_\theta$ for any $n<\omega$. But this means $\psi_0(e_n)\in\acc$, as desired.
\end{proof}

We can finally prove the totality of $G_n$ in {\sf W-KPl}.
\begin{Theorem}
    For every $n<\omega$ and for any set $X$, $G_\alpha$ is a provably total set-recursive-from-$\omega$ function in {\sf KPl}.
\end{Theorem}
\begin{proof}
    For any set $X$, we can specify the rank $\theta$ of $X$ in a $\Delta_0$-way, as in \cite{barwise}. Now, by Theorem \ref{TheoremWOKPl}, we can find in {\sf W-KPl} an ordinal that is order-isomorphic to $\psi_0(e_n)$ and define $L_{\psi_0(e_n)}(X)$ by $\Sigma$-recursion, which shows that $G_n$ is set-recursive-from-$\omega$. Moreover, the set $L_{\psi_0(e_n)}(X)$ is constructible in $L_{\Omega_0}(X)$, which shows that $G_n(X)$ is constructible in any admissible set containing $X$.
\end{proof}

\section{Conclusions}\label{SectConclusions}
According to the results in \S\ref{SectTheoremMainKPl} and \S\ref{SectTheoremMainWKPl}, there exist families $\{g_i\}_{i\in\mathbb{N}}$ of ordinal-valued set-recursive-from-$\omega$ functions such that 
\[\forall x\, \big(f(x) \in L_{g_i(x)}(x)\big)\]
holds whenever $f$ is a provably total function of $\mathsf{KPl}$, $\mathsf{KPl}^r$, or $\mathsf{W-KPl}$, respectively. Thus, the methods of ordinal analysis can be employed to bound the growth rates of set-recursive-from-$\omega$ functions similarly to how they yield bounds on the growth rates of recursive functions on $\mathbb{N}$. Moreover, these results are optimal according to \S\ref{SectWellOrderingProofs} in the sense that each of these functions $g_i$ is a provably total set-recursive-from-$\omega$ function in the respective theory. 

Combining these facts with the \textit{proofs} of the results from \S\ref{SectTheoremMainKPl} and \S\ref{SectTheoremMainWKPl} yields slight sharpenings. We mention the case of $\mathsf{KPl}$; the others are analogous.

\begin{Theorem}\label{TheoremKPlInternal}
Let $G_n(X) = L_{\psi_0(e_{n+3})}(X)$ be defined as in \S\ref{SectTheoremMainKPl}.
Suppose that $f$ is a set-recursive-from-$\omega$ function which $\mathsf{KPl}$ proves is total. Then, there exists $n\in\mathbb{N}$ such that
\[\mathsf{KPl}\vdash \forall x\, \big( f(x) \in G_n(x)\big).\]
\end{Theorem}
\proof
Let $f$ be a set-recursive-from-$\omega$ function which $\mathsf{KPl}$ proves is total.
Let $n$ be obtained as in the proof of Theorem \ref{main}. By observing the proof, we see that only terms below $e_{n+3}$ in $T(\theta)$ occur in it. By Theorem \ref{TheoremWOKPl}, $\mathsf{KPl}$ proves that $\psi_0(e_{n+3})$ is wellfounded and thus can carry out all steps of the proof. Thus, working in $\mathsf{KPl}$, we can prove the instance of Theorem \ref{main} for $f$.

Working in $\mathsf{KPl}$, we see that $f$ is a provably total set-recursive-from-$\omega$ function of $\mathsf{KPl}$ and thus we derive the conclusion of Theorem \ref{main}, i.e., that 
\[V \models \forall x\,\big(f(x) \in G_n(x)\big),\]
which is what we wanted to show.
\endproof

\subsection{Acknowledgments}
The results up to and including \S\ref{SectTheoremMainKPl} were part of the second author's Master's thesis \cite{tfm}, written under the supervision of the other two authors. The work of the first and second authors was partially supported by FWF grant STA-139. The third author acknowledges the support of the Generalitat de Catalunya (Government of Catalonia) under grants 2022 DI 051 (Departament d’Empresa i Coneixement) and 2021 SGR 00348 and support of the Spanish Ministry of Science and Innovation under grant PID2023-149556NB-I00, \emph{Dynamics of Gödel Incompleteness} (DoGI). The authors would like to acknowledge the support of the Erwin Schrödinger Institute in Vienna as part of the thematic program ``Reverse Mathematics'' in 2025.

\printbibliography

\end{document}